\documentclass[10pt,reqno]{amsart}
\usepackage{amsmath,amsthm,amssymb,amscd,url,enumerate}
\usepackage[T1]{fontenc}
\usepackage{tgpagella}
\usepackage{graphicx}
\usepackage{tikz}
\usepackage{tikz-cd}
\usepackage[all]{xy}
\usepackage{mathrsfs}
\usepackage{multirow}
\usepackage{array}
\usepackage{geometry}
\usepackage{caption} % caption for minipage figures
\usepackage{etoolbox} % quote in italics
\AtBeginEnvironment{quote}{\itshape}
\usepackage{hyperref}
\hypersetup{
    colorlinks=true,
    linkcolor=blue,
    filecolor=magenta,      
    urlcolor=cyan,
    pdftitle={CIRM Summer School Notes -- Stange},
    }

\calclayout

\theoremstyle{plain} 
\newtheorem{theorem}{Theorem} 
\newtheorem{conjecture}[theorem]{Conjecture}
\newtheorem{proposition}[theorem]{Proposition}

\newtheorem{corollary}[theorem]{Corollary}
\newtheorem{exercise}[theorem]{Exercise}

\theoremstyle{definition}
\newtheorem{definition}[theorem]{Definition}

\theoremstyle{remark}

\numberwithin{theorem}{section}
\numberwithin{figure}{section}

\newcommand{\TITLE}{An illustrated introduction to the arithmetic of Apollonian circle packings, continued fractions, and other thin orbits}
\newcommand{\TITLERUNNING}{Illustrated arithmetic of Apollonian circle packings, continued fractions, and thin orbits}
\newcommand{\DATE}{\today}
\newcommand{\VERSION}{1}

\newcommand{\dkron}[2]{\left(\dfrac{#1}{#2}\right)}%Kronecker symbol dfrac
\newcommand{\kron}[2]{\left(\frac{#1}{#2}\right)}%Kronecker symbol frac
  {\end{list}}

  {\end{list}}
\newcommand{\tightoverset}[2]{%
  \mathop{#2}\limits^{\vbox to -.5ex{\kern-1.05ex\hbox{$#1$}\vss}}}
\renewcommand{\a}{\alpha}

\def\Pcal{{\mathcal P}}

\def\Scal{{\mathcal S}}

\newcommand{\CC}{\mathbb{C}}
\newcommand{\DD}{\mathbb{D}}

\newcommand{\PP}{\mathbb{P}}
\newcommand{\QQ}{\mathbb{Q}}
\newcommand{\RR}{\mathbb{R}}
\newcommand{\ZZ}{\mathbb{Z}}

\renewcommand{\gcd}{{\operatorname{gcd}}}
\newcommand{\GL}{\operatorname{GL}}

\newcommand{\MOD}[1]{~(\textup{mod}~#1)}
\renewcommand{\pmod}{\MOD}

\newcommand{\ord}{\operatorname{ord}}

\renewcommand{\setminus}{\smallsetminus}

\newcommand\PSL{\operatorname{PSL}}
\newcommand\SL{\operatorname{SL}}
\newcommand\PGL{\operatorname{PGL}}
\newcommand\OK{\mathcal{O}_K}

\newcommand{\acosh}{\operatorname{acosh}}

\newif\ifcomments
\commentstrue  % change this to false to hide our rainbow comments :-)

\definecolor{myyellow}{rgb}{1.0, 0.75, 0.0}
\definecolor{mygreen}{rgb}{0.35, 0.71, 0.0}

\ifcomments
\newcommand{\KS}[1]{\textcolor{blue}{{\sf (ToDo:} {\sl{#1})}}}
\else
\newcommand{\KS}[1]{}
\fi

\newcommand{\HH}{\mathbb{H}}
\title[\TITLERUNNING]{\TITLE}

\author{Katherine E. Stange}

\date{\DATE, Draft \#\VERSION}
\address{%
Department of Mathematics, University of Colorado,
Campux Box 395, Boulder, Colorado 80309-0395}
\email{kstange@math.colorado.edu}
\thanks{}

\begin{document}

\begin{abstract}
	These notes cover and expand upon the material for two summer schools:  
The first, which was held at CIRM, Marseille, France, July 10-14, 2023, as part of \emph{Renormalization and Visualization for packing, billiard and surfaces},  was titled \emph{Number theory as a door to geometry, dynamics and illustration}.  
The second was held at NSU IMS in Singapore, June 3-7, 2024, as part of \emph{Computational Aspects of Thin Groups}, and was titled \emph{Integral packings and number theory}.  Both courses were put together by a number theorist for students and researchers in other fields.  They cover a web of ideas relating to Apollonian circle packings, integral orbits, thin groups, hyperbolic geometry, continued fractions, and Diophantine approximation.  The connection of geometry and dynamics to number theory gives an opportunity to illustrate arithmetic by appealing to our visual intuition.
\end{abstract}

\maketitle

\begin{center}
\textcolor{red}{THIS IS A DRAFT; \\ PLEASE REPORT ERRORS}
\end{center}

\tableofcontents

\section{Introduction}

These notes were meant to serve for two separate summer schools.  The first, which was held at CIRM, Marseille, France, July 10-14, 2023, as part of \emph{Renormalization and Visualization for packing, billiard and surfaces},  was titled \emph{Number theory as a door to geometry, dynamics and illustration}.  The abstract was:

\begin{quote}
The course will explore several related topics in number theory with dynamical and/or geometric facets:  continued fractions, Diophantine approximation, and Apollonian circle packings.  We will focus on both theoretical and experimental tools; a parallel goal will be to experience the role of visualization and illustration in mathematical research.  

In covering background material, the approach will emphasize the visual and dynamical:

\begin{enumerate}
	\item Continued fractions, quadratic forms, and Diophantine approximation.
	\item Hyperbolic geometry, Minkowski space, and Kleinian groups.
\end{enumerate}

With these tools at hand, we will study some areas of current research:

\begin{enumerate}
	\item The geometry of Diophantine approximation and continued fractions in the complex plane, including algebraic starscapes and Schmidt arrangements.  
	\item Apollonian circle packings, with an emphasis on their surprising relationships to the preceding topics.
\end{enumerate}
\end{quote}

The second was held at NSU IMS in Singapore, June 3-7, 2024, as part of \emph{Computational Aspects of Thin Groups}, and is titled \emph{Integral packings and number theory}.  The abstract for this course was:

\begin{quote}
The course will use Apollonian circle packings as a central example for connections between number theory and thin groups.  The symmetries of such a packing are governed by a thin group called the Apollonian group, and the curvatures form an orbit of that group.  Our goal is to study such orbits, particularly in the case that the orbit consists entirely of integers.  Some of the topics that are entwined with the study of these packings include quadratic forms, hyperbolic geometry in 2 and 3 dimensions, arithmetic geometry, continued fractions, spectral theory of graphs and strong approximation.  I will give a tour of the area, with the goal of introducing the number theory perspective on such problems, highlighting the tools at hand, and finishing by considering the wider class of problems that can be phrased as questions about the arithmetic of thin orbits.
\end{quote}

In both cases, the audience did not consist of number theorists, and my goal is to highlight connections between number theory and other domains, and to share some of the number theorist's perspective on ideas familiar in other domains.  The notes will contain a number of examples of the ways that number theory leads to geometry and dynamics.   As a number theorist, I see these relationships as number theory questions with geometric answers, but this is surely not the only way to see them.  

Three motivating questions are:

\begin{enumerate}
	\item Which complex numbers can be well approximated by algebraic numbers (of various flavours)?
	\item What curvatures appear in a primitive integral Apollonian circle packing or another thin orbit?
	\item For imaginary quadratic fields $K$, how is the number theory of $\mathcal{O}_K$ visible in the Bianchi group $\PSL_2(\mathcal{O}_K)$?
\end{enumerate}
It turns out these are all connected to one another by their underlying geometry.  And they are all generalizations of very classical questions in number theory about the integer solutions to equations, the distribution of rationals in the real line, and the study of continued fraction expansions.

I must admit that a particular love of the author is the \emph{Apollonian circle packing}.  If one is interested in the classical number theoretic topics of $\PSL_2(\ZZ)$, continued fractions, or quadratic forms, and one begins to wonder about higher dimensional analogues, one comes very naturally to study $\PSL_2(\ZZ[i])$, Schmidt's continued fractions, and Hermitian forms.  Studying these geometrically, one discovers that the essential new riddle that lies wrapped in the mystery inside the enigma\ldots is the Apollonian circle packing.  It is a halfway house, an unavoidable geometric feature, but an orbit of a group that is not itself an arithmetic group.  It begs the same questions we ask for Diophantine equations, without being one itself.  This is its allure.

Finally, these topics are all examples of the interplay between research and the illustration of mathematics.  The use of computers to explore mathematics with our highly evolved visual cortices is not only productive, but also \emph{sensual}, in the language of Conway \cite{ConwayFung}.  One of my goals in these notes is to emphasize the utility and beauty of these explorations, and make a case for their wider use in research.

These notes are written by a number theorist.  I apologize in advance for the inevitable injustice in making connections to other domains which I cannot fully follow up.  I hope, however, that this insufficiency will serve as a motivation.  There are a great number of interesting unexplored species in this jungle.

One final note to the reader:  in the off chance that your enthusiasm causes you to rush headlong (as mine sometimes does), I would warn you that some important statements are contained in the exercises, so their statements are an important part of the narrative.

\subsubsection{Acknowledgements}  
I owe a great debt of gratitude to Jayadev Athreya for inspiration and support over many years, and for his invitation to participate in his Morlet semester at Centre international de rencontres math\'ematiques, which eventually led to these notes.
Thank you to all the organizers of the Research School, \emph{Renormalization and Visualization for packing, billiard and surfaces}, held July 10-14, 2023, for which these lecture notes were originally drafted, and to the organizers of the Programme \emph{Computational Aspects of Thin Groups}, held June 3-7, 2024, for which these lecture notes were adapted and expanded.  These organizers include Jayadev Athreya and Nicalas Bedaride in the first case; and Bettina Eick, Eamonn O'Brien, Alan W. Reid, and Ser Peow Tan in the second.
Thank you to all the students and researchers who attended and contributed feedback; and gratitude to the Centre international de rencontres math\'ematiques (CIRM) and the National University of Singapore Institute for Mathematical Sciences (NUS IMS) for hosting and supporting these conferences.  Thank you to James Rickards for TA help at the IMS school.
Thank you to Elena Fuchs, Daniel Martin and James Rickards for consultation, and to many for sharing permission to use their images (individually noted in each case).  Thank you to Glen Whitney for feedback.
Thank you to Summer Haag for her detailed reading and feedback, and willingness to dive into the minus signs.
And finally, thank you to all my intrepid collaborators on all the projects which are mentioned in these notes.

\section{A number theory perspective}

\subsection{Diophantine problems}

Given a system of polynomial equations, what are the integer or rational solutions?  Such problems are called \emph{Diophantine problems} and make up a significant part of modern and ancient number theory.

Some examples of Diophantine problems are:
\begin{enumerate}
	\item Solutions to a single equation $f(x_1,\ldots,x_n) = 0$, amongst the most famous of which are quadratic forms $ax^2 + bxy + cy^2 = m$ and elliptic curves $y^2 = x^3 + ax + b$.
	\item Generalizing the previous case, solutions to systems of equations in several variables, which represent algebraic varieties such as abelian varieties, curves and surfaces.  
	\item Many problems reduce to or can be phrased as relating to Diophantine problems, such as the famous congruent number problem, the units in a ring of integers (e.g., Pell's equation), and the arithmetic study of function iteration (arithmetic dynamics).
	\item Among the more significant examples are algebraic varieties arising as moduli spaces in the study of other problems (e.g., modular curves), where knowing the points amounts to classifying mathematical objects of some type.
	\item The study of matrix groups and algebraic groups more generally.  This includes most/many Lie groups.
	\item In the context of a group such as an algebraic group, the study of the elements of an orbit.
\end{enumerate}

The depth and breadth of the partial list above amply motivates the study of Diophantine equations.

As we shall see, in the process of trying to understand these problems, we are led naturally to look for solutions modulo $p$, in the $p$-adics $\QQ_p$, in the real numbers $\RR$ and in the complex numbers $\CC$.  Knowledge of the solutions in these other realms all have a role to play in the original realm -- $\QQ$ or $\ZZ$ -- which is in some sense the hardest place to look for solutions.  We are also led naturally to generalize these questions to number fields and their rings of integers.

\subsection{Quadratic forms}
\label{sec:QFone}

There are some natural starting points for the study of Diophantine equations.  After linear equations such as $ax+by=c$ (which one solves with the \emph{Euclidean algorithm}, which is central to everything), the historical and logical next step is quadratic equations and quadratic forms.  These will have an important role to play throughout the entirety of these notes.

An \emph{$n$-ary quadratic form} is a homogeneous polynomial in $n$ variables of degree $2$.  (The term \emph{form} refers to the homogeneity.)
Therefore, a \emph{binary quadratic form} is a homogeneous polynomial in two variables of degree $2$, which will necessarily have the shape
\[
	f(x,y) = ax^2 + bxy + cy^2
\]
for some coefficients $a,b,c$, which we will take to be in $\RR$, to begin.  

Binary quadratic forms have a beautiful Diophantine theory, whose geometric incarnations we will meet later.  The main question is to ask:  what are the values represented by a quadratic form?  For the number theorist, we ask especially for forms with integer coordinates, evaluated on integer inputs.  A famous example is the following:

\begin{theorem}[Fermat]
The quadratic form $x^2 + y^2$ represents exactly those odd primes which are $1$ modulo $4$.  
\end{theorem}

There are many proofs of this fact, arising from many different tools.  It can be seen as a property of the Gaussian integers $\ZZ[i]$ and quadratic reciprocity (Corollary~\ref{cor:twosquares} below), a result of the fact that $\pi > 2$ via the geometry of lattices (Exercise~\ref{ex:pi2}), using dynamics (Exercise~\ref{ex:zagier}, or proven using Jacobi sums \cite{IrelandRosen} or even partition theory \cite{Partition}.

The \emph{discriminant} $\Delta_f$ of the form $f$ is $b^2 - 4ac$.
If the form takes both negative and positive values when evaluated at $(x,y) \in \ZZ$, we call it \emph{indefinite}.  Otherwise, there are two cases:  it can be \emph{semi-definite} (taking the value $0$ on some non-trivial input) or \emph{definite} (if it does not).  If it is definite, its values away from $(0,0)$ all have the same sign, so it is either \emph{positive} or \emph{negative}.  

As number theorists, our interest will be in the case $a,b,c \in \QQ$.  Up to a scalar multiple, it is convenient to take $a,b,c \in \ZZ$ having no common factor.  This is called a \emph{primitive integral binary quadratic form}.  
We are mainly interested in \emph{positive definite primitive integral binary quadratic forms}.  This cumbersome terminology teases the limit of human patience, so some authors abbreviate such forms as `PDPIBQFs' or something similarly awkward; we will just say `integral quadratic form' and imply all the abovementioned qualifiers unless stated otherwise.  If we say `real quadratic form' we mean a positive definite binary quadratic form with real coefficients.

We wish to abstract away the choice of coordinates inherent in this definition:  we would like to identify $f(x,y)$ and $g(X,Y)$ if they are related by a linear change of variables $X = \alpha x+ \beta y$, $Y =  \gamma x + \delta y$, where $\alpha, \beta, \delta, \gamma \in \ZZ$ and $\alpha \delta - \beta \gamma = 1$.   We write this as a matrix action of $\SL_2(\ZZ)$:
\[
	g(X,Y) = M \cdot f(x,y) = f( M \cdot (x,y)^t ), \quad
	M = \begin{pmatrix}
		\alpha & \beta \\ \gamma & \delta
	\end{pmatrix} \in \SL_2(\ZZ).
\]
We call two such forms \emph{properly equivalent} or just \emph{equivalent} (there is also a notion of $\GL_2(\ZZ)$-equivalence, but we will prefer this one for now).  Equivalent forms share much of their behaviour.
Two equivalent forms will represent the same set of integers upon integer inputs, and will have the same discriminant.  If $f$ is primitive, then $M \cdot f$ is also primitive.  
Finally, observe that the matrix $-I \in \SL_2(\ZZ)$ has no effect on a form $f$, so that it is just as reasonable to talk about the action of $\PSL_2(\ZZ) := \SL_2(\ZZ)/\pm I$.  

\begin{exercise}
	Verify the statement that proper equivalence preserves primitivity, discriminant and the set of integers represented.
\end{exercise}

\begin{exercise}
	\label{ex:pi2}
  This exercise outlines a geometric proof that $x^2 + y^2$ represents exactly those odd primes that are $1$ modulo $4$.  It turns on the fact that $\pi > 2$.
	\begin{enumerate}
		\item Observe that $x^2 + y^2$ cannot represent anything which is $3$ modulo $4$.
		\item Let $p \equiv 1 \pmod{4}$.  Show that $-1$ is a square modulo $p$ (this can be accomplished using the group theory of $(\ZZ/p\ZZ)^*$, which is cyclic).  Conclude that $p$ divides $m^2 + 1$ for some $m$.  
		\item Prove Minkowski's Theorem in dimension two:  Any convex set in $\RR^2$ which is symmetric about the origin and of volume exceeding $4$ contains a non-zero integer point. 
		\item Consider the lattice $\Lambda$ of $\RR^2$ generated as the $\ZZ$-span of $(1,m)$ and $(0,p)$.  Show that all elements $\mathbf{v} \in \Lambda$ have norm squared $|| \mathbf{v} ||^2$ divisible by $p$.  Compute the covolume of this lattice (the area of the fundamental parallelogram).
		\item Conclude the argument using Minkowski's Theorem to show there is an element $\mathbf{v} \in \Lambda$ having $||\mathbf{v}||^2 = p$.
	\end{enumerate}
\end{exercise}

\subsection{Local-to-global}

Consider the question of when a positive integer $n$ is a sum of three squares:  $n = x^2 + y^2 + z^2$.  It was first observed by Legendre that this has a solution $(x,y,z)$ if and only if $n$ is not $7$ modulo $8$.  We can rephrase this to say that the equation has a solution in $\ZZ$ if and only if it has a solution modulo every odd prime, and modulo $8$.  This characterizes $n$'s representability in terms of the `local' congruence conditions `at' each prime.  In fact, Fermat's Theorem, that an odd prime can be represented as a sum of squares if and only if it is $1 \pmod{4}$, can be viewed from this lens also.  This is called a \emph{local-to-global} principle.

To formalize this just slightly, let $f_i(x_1,\ldots,x_n)=0$ be a system of polynomial equations with coefficients in $\QQ$.  Suppose this system has a solution in $\QQ$, say $(x_1, \ldots, x_n) \in \QQ^n$.  Since $\QQ$ embeds in $\RR$ and in $\QQ_p$ for all $p$, the $\QQ$-solution, called a \emph{global solution}, will also provide \emph{local solutions} everywhere, i.e. solutions in $\QQ_p$ and $\RR$.  It is traditional in this setting to consider $\infty$ a prime of $\QQ$, along with the usual primes, calling all of these \emph{places} or \emph{valuations}; the primes are \emph{finite places} and $\infty$ is the \emph{infinite place}.  The field $\RR =: \QQ_\infty$ is the completion of $\QQ$ at the place $\infty$ and $\QQ_p$ is the completion at the place $p$, for each prime $p$.  These completions, called \emph{local fields}, are actually easier fields in which to study polynomial solutions.  All of this has generalizations for number fields, but we will stick to $\QQ$ for the moment.

Thus we have observed that global solutions imply local ones.  In the case of quadratic forms, the converse is true:

\begin{theorem}[Hasse-Minkowski Theorem]
	Let $Q(x_1,\ldots,x_n)$ be an $n$-ary quadratic form with coefficients in $\QQ$.  Then $Q = 0$ has a solution in $\QQ$ (i.e. $(x_1, \ldots x_n) \in \QQ^n$) if and only if it has a solution in $\QQ_p$ for all $p$ and in $\RR$.
\end{theorem}

There is a version of this theorem over number fields more generally.  For a proof see \cite[Theorem 14.3.3]{Voight}; it is non-trivial.  When such a converse holds, i.e. there are global solutions if and only if there are solutions everywhere locally, then we say that a variety satisfies the \emph{Hasse principle} or a \emph{local-to-global principle}.  
The three squares example of Legendre at the beginning of this section demonstrates an integral flavour of the local-to-global principle, obtained by looking at the integers $\ZZ$ and $\ZZ_p$.

The Hasse principle (when it holds) is powerful.  In particular, determining if a variety $X$ has solutions over a local field is a finite computation (for $\QQ_p$, check for solutions modulo $p$ and then use Hensel's Lemma).  By contrast, determining the existence of global solutions for a variety without the Hasse principal can be quite difficult.  In fact, the general problem of determining if there are integer solutions to a Diophantine problem is undecideable; this is known as Hilbert's Tenth Problem \cite{H10}.

\begin{exercise}
 Show that $x^2 + y^2 +7z^2 = 0$ has no non-trivial $\RR$-solutions, no non-trivial $\QQ_7$-solutions and no non-trivial $\QQ$-solutions.
\end{exercise}

\subsection{Quadratic reciprocity}

Quadratic reciprocity was first observed by Legendre and Euler and proved by Gauss.  Whereas Sunzi's Theorem\footnote{More commonly known as the Chinese Remainder Theorem} can be viewed as a statement that coprime moduli are `independent' in a certain way, quadratic reciprocity describes a deep way in which $\ZZ/p\ZZ$ and $\ZZ/q\ZZ$, for primes $p \neq q$ \emph{do} interact.  

\begin{definition}\label{def:legendre}
	The \emph{Legendre symbol} is the defined for integer $a$ and prime $p$ as:
\[
	\left( \frac{a}{p} \right) = 
	\left\{
		\begin{array}{ll}
			1 & a \text{ is a non-zero square modulo } p \\
			-1 & a \text{ is not a square modulo } p \\
			0 & a \text{ is zero modulo } p \\
		\end{array}
		\right.
	\]
\end{definition}

\begin{theorem}[Quadratic reciprocity]
	\label{thm:quadrec}
	Let $p$ and $q$ be distinct odd primes.  Then
	\[
		\left( \frac{p}{q} \right)
		\left( \frac{q}{p} \right)
		=
		(-1)^{\frac{p-1}{2}\frac{q-1}{2}}
	\]
	Furthermore,
	\begin{enumerate}
		\item $-1$ is a square modulo $p$ if and only if $p \equiv 1 \pmod{4}$
		\item $2$ is a square modulo $p$ if and only if $p \equiv \pm 1 \pmod{8}$
	\end{enumerate}
\end{theorem}

It is said that there are hundreds of proofs of quadratic reciprocity in the literature\footnote{Indeed, Lemmermeyer maintains a list currently over 300 entries: \url{https://www.mathi.uni-heidelberg.de/~flemmermeyer/qrg_proofs.html}.}.  Among the themes recurring in these proofs we often see aspects of the Fourier transform, or the signs of permutations, among others.  We now provide a proof of Fermat's Theorem which depends on quadratic reciprocity.

\begin{corollary}
	\label{cor:twosquares}
	The integers $n$ which can be represented as a sum of two integer squares are exactly those for which, for each prime $p$ that divides $n$, either (a) $p$ divides $n$ to an even power; or (b) $p \equiv 1,2 \pmod{4}$.
\end{corollary}

\begin{proof}
	First one shows that a prime $p \equiv 3 \pmod{4}$ cannot be represented, simply because there's no solution modulo $4$.  

	Then one shows that primes equivalent to $1 \pmod{4}$ \emph{can} be represented, as follows.  By quadratic reciprocity, such a prime $p = 4k+1$ has the property that $-1$ is a square modulo $p$.  Therefore $4k \equiv -1 \equiv x^2 \pmod{p}$.  Thus $np = x^2 + 1$, a sum of squares.  Thus $(x+i)(x-i) = np$ in $\ZZ[i]$, which has unique factorization.  Since the gcd of $x+i$ and $x-i$ is at most a factor of $2i$, and $p$ is odd, we see that $p$ cannot be prime in $\ZZ[i]$ (since otherwise, it would divide both, by symmetry under conjugation).  Therefore $p = (a+bi)(c+di)$, where neither factor is a unit times an integer, so $abcd \neq 0$.  Then since $p \in \ZZ$, $c+di = k(a-bi)$ for some $k \in \ZZ$.  Since $p$ is prime, $k = \pm 1$.  So we have $p = a^2 +b^2$.

	Finally, observe that sums of squares are exactly the norms of Gaussian integers.  We have shown that every prime $p \equiv 1 \pmod{4}$ is such a norm, and no prime $p\equiv 3 \pmod{4}$ is such a norm.  Also, $2 = 1^2 + 1^2 = N(1+i)$.  By unique factorization, the norms are exactly those which can be produced as products of norms of Gaussian primes, which we have now essentially classified.
\end{proof}

\begin{exercise} (Zagier \cite{Zagier})
	\label{ex:zagier}
	There is a short proof of Corollary~\ref{cor:twosquares} using an argument about fixed points.   
	Suppose $p \equiv 1 \pmod{4}$.
	Let $S = \{ (x,y,z) \in \ZZ^{>0} : x^2 + 4yz = p \}$ be the set of natural numbers solving $x^2 + 4yz = p$.  
	\begin{enumerate}
		\item Show the following is an involution of $S$ and has exactly one fixed point:
			\[
				(x, y, z)\mapsto\begin{cases}(x+2z, z, y-x-z) &\text{if }  x < y-z;\\(2y-x, y, x-y+z)&\text{if } y-z < x < 2y;\\(x-2y, x-y+z, y)&\text{if } x > 2y.
 \end{cases}
				\]
			\item What other (more trivial) involutions are there?
		\item Prove $p$ can be written as a sum of two squares.
	\end{enumerate}
\end{exercise}

\subsection{Brauer-Manin obstructions}

A famous counterexample to the Hasse principle due to Selmer is given by
\[
	3x^3 + 4y^3 + 5z^3 = 0
\]
which has local solutions everywhere but no global solutions.
When the Hasse principle fails, many of the failures are captured by a \emph{Brauer-Manin obstruction}, which arises from reciprocity laws.  
We will illustrate the phenomenon by an example of a local-to-global failure due to Iskovskikh \cite{Iskovskikh}:
\[
	y^2 + z^2 = (3-x^2)(x^2-2)
\]

\begin{proposition}
	The equation above has local solutions everywhere, but no global solutions.
\end{proposition}

\begin{proof}
	To show that it has a solution in $\RR$, take $x^2 > 2$ but $x^2 < 3$ so the right hand side is positive.  For $\QQ_p$, it suffices to find solutions modulo $p$ and lift by Hensel's lemma; we leave this as an exercise.

	To show there are no global $\QQ$-solutions, first, let us homogenize the equation, so we can look for $\ZZ$ solutions:
\[
	X: y^2 + z^2 = (3t^2-x^2)(x^2-2t^2) =: A(t,x) B(t,x).
\]
We may assume that there is no common factor to the quadruple $x,y,z,t$.

First we examine the equation modulo $4$.  By running through the possibilities for $x$ and $t$ being even or odd, we restrict the possible values of $(A,B)$ modulo $4$: $\{ (0,0), (2,3), (3,1), (3,2) \}$.
However, $(A,B) \equiv (0,0) \pmod{4}$ if and only if $x,y$ are both even if and only if $x,y,z,t$ are all even.  But we have ruled out such a common factor, so we have the list:
\[
	\{ (2,3), (3,1), (3,2) \}.
\]

We will now consider primes $p$ dividing at least one of $A$ or $B$.

First, suppose $p$ divides both $x$ and $t$.  Since $x,y,z,t$ cannot have a common factor, we may assume without loss of generality that $p$ does not divide $z$.  Then considering the equation modulo $p$, we find that $-1 \equiv y^2/z^2$ is a square modulo $p$.  Since our previous work ruled out $p=2$ (as $x$ and $t$ are not both even), we conclude that $p \equiv 1 \pmod{4}$ by the first supplementary law of quadratic reciprocity (Theorem~\ref{thm:quadrec}(1)).

%But any common factor of $x,y,z,t$ can be scaled out, so we need only consider solutions with $(x,t)=1$, 
%where $(x,t)=1$, s
%So $\gcd(A,B)=1$. 
%First, the left hand side must be one of $0,1,2 \pmod{4}$ (since the only squares modulo $4$ are $0,1$).  Therefore, the values of $(A,B)$ modulo $4$ that are possible are
%\[
%	\{ (0,*), (*,0), (1,1), (3,3), (2,1), (1,2), (2,3), (3,2) \}
%\]
%We also rule out $(0,0)$ by coprimality, and intersect the sets, so that modulo $4$ we have
%\[
%	\{ (2,3), (3,2) \}.
%\]

Second, suppose $p$ divides at most one of $x$ and $t$.  Then $p$ cannot divide both $A$ and $B$.  Let $p^k$ be the maximum power to which it appears.  On the left, $p^k$ divides a sum of squares, so $p^k$ must be $0$, $1$ or $2 \pmod{4}$ (Corollary~\ref{cor:twosquares}).  

Hence in either case, all the maximal prime powers that divide $A$ or $B$ are $0$, $1$ or $2 \pmod{4}$ and therefore $(A,B) \pmod{4}$ can only be a member of the list
\[
	\{ (0,0), (1,1), (2,2), (0,1), (1,0), (2,0), (0,2), (1,2), (2,1) \}
\]
Comparing with the list from the first part, we discover that there are no solutions.
\end{proof}

The proof is completely elementary with the exception of the use of the first supplementary law of quadratic reciprocity (once directly and once in the form of Corollary~\ref{cor:twosquares}).  Therefore we think of this obstruction as arising from quadratic reciprocity.

\begin{exercise}Complete the proof above by finding solutions in $\QQ_p$ for each $p$.
\end{exercise}

%There is an $\alpha \in \Br X$ which becomes $(3-x^2,-1)_v = (x^2-2,-1)_v$ (as their product is $(y^2 + z^2,-1)_v = 1$).  
%One shows that this takes values $1$ away from $\QQ_2$, but $-1$ at $\QQ_2$, hence $\sum_v \inv_v \neq 0$.

\subsection{The modular group $\PSL_2(\ZZ)$}

We met $\SL_2(\ZZ)$ and $\PSL_2(\ZZ)$ as natural groups of symmetries on quadratic forms.  They have many other roles to play in mathematics, some of them geometric.  This leads to certain connections between quadratic forms and geometry.  Recalling these essential and beautiful facts is our next task.

We will begin with the action of $\SL_2(\ZZ)$ on the upper half plane $\HH^2_U$.  We define the \emph{upper half plane} as
\[
	\HH^2_U = \{ z \in  \CC : \Im(z) > 0 \} \subseteq \CC.
\]
The notation $\Im$ is for `imaginary part'\footnote{This symbol is \LaTeX's command \texttt{\textbackslash Im} and was traditionally used in typesetting old european books.}; later $\Re$ will be for real part.
The action is via M\"obius transformations:
\[
	\begin{pmatrix} a & b \\ c & d \end{pmatrix} \cdot z = \frac{az + b}{cz+d}
\]
in terms of the usual $\CC$ structure.
The fact that the matrix $-I$ acts as the trivial M\"obius transformation motivates our use of the projectivization $\PSL_2(\ZZ) := \SL_2(\ZZ)/ \pm I$.  

\begin{exercise}
	M\"obius transformations with coefficients from $\CC$, i.e. $\PSL_2(\CC)$, act on the extended complex plane $\widehat{\CC} := \CC \cup \{ \infty \}$.
	\begin{enumerate}
		\item Amongst these, show that the M\"obius transformations preserving $\HH^2_U$ are exactly those with real coefficients and positive determinant.
		\item In $\widehat{\CC}$, we consider straight lines to be circles through $\infty$.  Show that M\"obius transformations take circles to circles and preserve angles.
	\end{enumerate}
\end{exercise}

A standard fundamental region for the action of $\PSL_2(\ZZ)$ on $\HH^2_U$ is as follows.
\begin{theorem}[{for example, \cite[Proposition 1.5]{SilvermanAdvanced}}]
	\label{thm:F}
	Let
	\[
		\mathcal{F}  = \left\{ z \in \HH^2_U : |z| > 1, \frac{1}{2} < \Re(z) \le \frac{1}{2}, \left(|z|=1 \Rightarrow \Re(z) \ge 0\right) \right\}.
	\]
	Then $\mathcal{F}$ is a fundamental domain for the action of $\PSL_2(\ZZ)$ on $\HH^2_U$, meaning that for all $z$, exactly one $\PSL_2(\ZZ)$-translate of $z$ lies in $\mathcal{F}$.  
\end{theorem}
To visualize this theorem, it is useful to consider two important elements of $\PSL_2(\ZZ)$:
\[
	S = \begin{pmatrix} 0 & -1 \\ 1 & 0 \end{pmatrix}, \quad T = \begin{pmatrix} 1 & 1 \\ 0 & 1 \end{pmatrix}.
\]
The action of $S$ and $T$ as M\"obius transformations is inversion $z \mapsto -1/z$ and translation $z \mapsto z + 1$, respectively.
See Figure~\ref{fig:sl2} for an image of $\mathcal{F}$ and some of its translates under $S$ and $T$.

\begin{exercise}
	\begin{enumerate}
		\item Show that if $M = \begin{pmatrix} a & b \\ c & d \end{pmatrix} \in \PSL_2(\ZZ)$, then 
		\[
			\Im( M \cdot z) = \frac{ \Im(z) }{ |cz+d|^2 }.
		\]
	\item 	Prove that any $\PSL_2(\ZZ)$ orbit intersects $\mathcal{F}$ (this proves a part of Theorem~\ref{thm:F}).  To do so, first show that there is an element $M \in \PSL_2(\ZZ)$ which maximizes the imaginary part of $M \cdot z$.  Then illustrate how to move this element of the orbit into $\mathcal{F}$ using $S$ and $T$.
	\end{enumerate}
\end{exercise}

The $\PSL_2(\ZZ)$-stabilizer of each point $z \in \mathcal{F}$ is trivial, with the exception of the following special points:
	\[
		\operatorname{Stab}(i) = \{ I, S \}, \quad
		\operatorname{Stab}(e^{2\pi i/6}) = \{ I, TS, (TS)^2 \}.
	\]

\begin{figure}
	\includegraphics[width=6in]{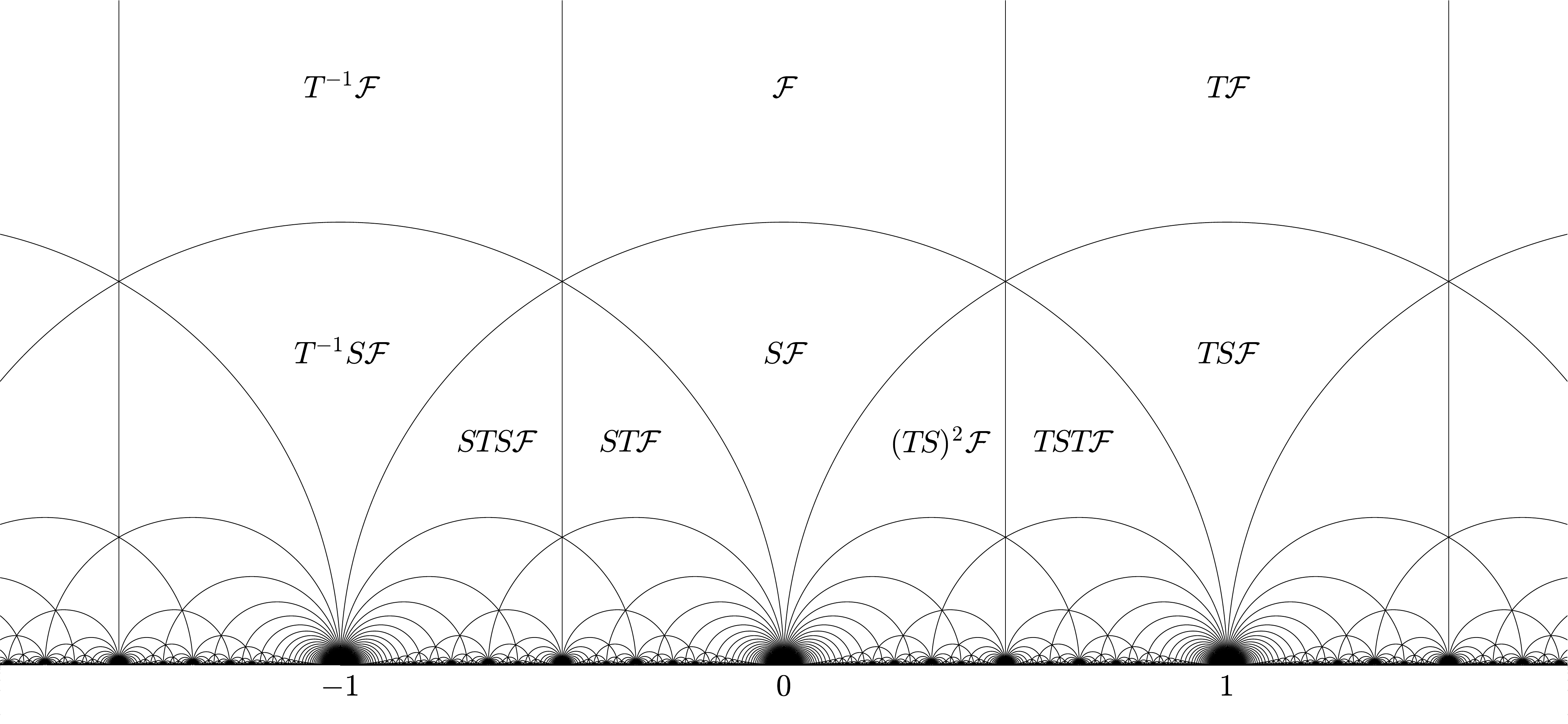}
	\caption{The fundamental region $\mathcal{F}$ of $\SL_2(\ZZ)$ and some of its translates.}
	\label{fig:sl2}
\end{figure}

In fact, these two elements $S$ and $T$ generate $\PSL_2(\ZZ)$.
\begin{theorem}[{for example, \cite[Corollary 1.6]{SilvermanAdvanced}}]
	The group $\PSL_2(\ZZ)$ is generated by $S$ and $T$.
	In fact, $\PSL_2(\ZZ)$ is the free product of the subgroups generated by $S$ and $ST$, of order $2$ and $3$ respectively.  We have
\[
	\PSL_2(\ZZ) \cong \langle S, T : S^2 = 1, (ST)^3 = 1 \rangle.
\]
\end{theorem}

\subsection{Quadratic forms in the upper half plane}
\label{sec:quadratic-forms}

Our next goal is to parametrize real quadratic forms $f(x,y) := ax^2 + bxy + cy^2$ in some useful way.  %Such a form is naturally associated to the quadratic polynomial $f(x,1) = ax^2 + bx + c$ which, since we are in the positive definite case, will have two distinct complex roots $z$ and $\overline{z}$.  
To a quadratic form $f(x,y) = ax^2 + bxy + cy^2$, we might associate the polynomial in one variable $f(x,1) = ax^2 + bx + c$; and to such a minimal polynomial we might naturally associate the quadratic form $ax^2 + bxy + cy^2$.  This association preserves discriminant.

A minimal polynomial with $\Delta < 0$ has two complex roots, one of which will lie in $\HH^2_U$.  Thus, to positive definite quadratic forms, we may associate points in the upper half plane, and this association is clearly a bijection, at least as long as we take the form $f$ up to scaling by an element of $\RR^*$.

However, associating the root $z$ (or the pair $(z, \overline{z})$) to the form $f$ breaks a symmetry, at the very least between $x$ and $y$, but more generally it is not invariant under our preferred $\SL_2(\ZZ)$-equivalence.  First of all, it is more natural to say that the solutions to $f(x,y) = ax^2 + bxy + cy^2 = 0$ are the two projective points $[z: 1]$ and $[\overline{z}: 1]$ in $\PP^1(\CC)$, and therefore to identify a solution $[z:1]$ with $[\lambda z: \lambda 1]$, for $\lambda \in \CC^*$.  Then, from the perspective of $\SL_2(\ZZ)$-equivalence of quadratic forms, we wish to identify the roots of $f(x,y)$ with those of $M \cdot f(x,y)$, for each $M = \begin{pmatrix} \alpha & \beta \\ \gamma & \delta \end{pmatrix} \in \SL_2(\ZZ)$, which are 
\[
	[\alpha z + \beta: \gamma z + \delta ],\quad [\alpha \overline{z} + \beta: \gamma \overline{z} + \delta ].
\]
Because we are in projective space, it is more natural to projectivize and identify the roots under $\PSL_2(\ZZ)$-equivalence, since $-I$ has no effect on the form at all\footnote{By which I mean, $f(-x,-y) = f(x,y)$.}.

This motivates the following theorem.

\begin{theorem}
	\label{thm:quad-equi}
	Define a map $\rho$ on the collection of primitive integral binary quadratic forms up to $\RR^*$ scaling, taking values in $\HH^2_U$, by letting $\rho(f) = z$, where $z$ is the root of $f(x,1)$ in $\HH^2_U$.  Then this map is $\PSL_2(\ZZ)$-invariant, where the action of $f(x,y)$ is by proper equivalence, and the action on $\HH^2_U$ is by M\"obius transformation.  In other words, $M \cdot \rho(f) = \rho(M \cdot f)$ for $M \in \PSL_2(\ZZ)$.
\end{theorem}

\begin{exercise}
	Prove Theorem~\ref{thm:quad-equi}.
\end{exercise}

The following exercise addresses classical questions of Gauss and Legendre:  how many equivalence classes of primitive integral quadratic forms exist for a fixed discriminant?

\begin{exercise}
	\label{ex:reduc}
	\begin{enumerate}
		\item Show that every real quadratic form is equivalent to one of the form $Ax^2 + Bxy + Cy^2$ satisfying (i) $|B| \le A \le C$ and (ii) $B \ge 0$ whenever one of the $\le$ in part (i) is an equality.  Hint: use the upper half plane.
	\item Use the previous exercise to determine how many inequivalent primitive integral quadratic forms there are of discriminant $-4$ and $-20$.
	\item Fix $\Delta < 0$.  Prove that there are finitely many distinct equivalence classes of integral quadratic forms of discriminant $\Delta$.  Can you give a bound?
	\end{enumerate}
\end{exercise}

\subsection{Lattices in the upper half plane}

The upper half plane also parametrizes certain lattices\footnote{A common place to first meet this idea is in the complex theory of elliptic curves.}.  A lattice is a discrete\footnote{In the metric topology of $\CC$.} subgroup of the additive group of $\CC$.  We say it is rank two if it is isomorphic to $\ZZ^2$ as a $\ZZ$-module.

\begin{exercise}
	Show that a lattice in $\CC$ is of rank two if and only if it spans $\CC$ as an $\RR$-vector space.
\end{exercise}

The fundamental observation is that $\PP^1(\CC)$ is (almost) in bijection with rank two lattice bases (up to scaling) in $\CC$, via $[z:w] \leftrightarrow w\ZZ + z\ZZ$.  I say \emph{almost} because points for which $z/w \in \RR$ give rise to $\ZZ$-modules which do not span $\CC$ (being either rank one or not discrete).  If we restrict to the upper half plane $\HH^2_U \subseteq \widehat{\CC} \cong \PP^1(\CC)$, then this associates to each $z \in \HH^2_U$ the rank two lattice $\ZZ + z\ZZ$.  As mentioned, we must consider lattices only up to \emph{homothety}, i.e., scaling by $\CC^*$.  
That is, we say two lattices $\Lambda_1$ and $\Lambda_2$ are homothetic, writing $\Lambda_1 \sim \Lambda_2$, if $\Lambda_1 = \lambda \Lambda_2$ for some $\lambda \in \CC^*$.  It is not hard to verify this is an equivalence relation.

A rank two lattice in $\CC$ comes with an \emph{orientation}:  if the angle from the first basis vector to the second is less than $\pi$, then it is \emph{positively oriented} and otherwise it is \emph{negatively oriented}.  
That $\Im(z) > 0$ implies the associated lattices are positively oriented.  

\begin{theorem}
	The upper half plane $\HH^2_U$ is in bijection with homothety classes of positively oriented rank two lattice bases in $\CC$, by the following map:
	\[
		z \mapsto \ZZ + z\ZZ.
	\]
	Furthermore, the bijection is $\PSL_2(\ZZ)$-equivariant, where $\PSL_2(\ZZ)$ acts on lattices by change of basis:
	\[
		\begin{pmatrix} a & b \\ c & d \end{pmatrix} \cdot (\ZZ + z\ZZ) = (cz+d)\ZZ + (az+b)\ZZ.
	\]
\end{theorem}

\begin{exercise}
	Prove the theorem.
\end{exercise}

\subsection{Lattices and quadratic forms}

Let us return to the question raised above:  What is preserved under the action of $\PSL_2(\ZZ)$ on quadratic forms, if not the roots $[z:1]$ and $[\overline z:1]$?  The answer:  the collection of the totality of roots of all equivalent forms.  By the last section, one way to encapsulate this data is as a $\ZZ$-lattice, whose bases are in bijection with the roots.  If $[z:1] \in \PP^1(\CC)$ represents a root of the form $f(x,y)$, then we set $\Lambda_f := \ZZ + z\ZZ$.  
%Viewing $\Lambda_f \subseteq \widehat{\CC} \cong \PP^1(\CC)$ via $\alpha \mapsto [\alpha:1]$, the root $[z:1]$ lies in $\Lambda_f$.  
%The $\PP^1(\CC)$ root of $M \cdot f(x,y)$ is $M^{-1} \cdot [z:1] = [z':1]$, and we observe that $\ZZ + z'\ZZ \sim \ZZ + z\ZZ$.
%This is invariant under the action of $\SL_2(\ZZ)$, which changes the basis with which the lattice is presented, but not the lattice itself.  The other bases of the lattice correspond to the roots of forms equivalent to $f$. 
Since we always have two conjugate roots, we always have two conjugate lattices, exactly one of which is positively oriented, and we can choose that lattice to associate to our quadratic form.  %This lattice should be taken up to \emph{homothety}, i.e. scaling by an element of $\CC^*$, because the point $[\alpha, \beta]$ associated to a lattice $\alpha \ZZ + \beta \ZZ$ lies in $\PP^1(\CC)$.  

From an equivalence class of lattices up to homothety, we can also recover the quadratic form:  use a homothety to write the lattice $\Lambda$ as $\ZZ +  z\ZZ$ where $z \in \HH^2_U$, take the minimal polynomial $ax^2 + bx + c$ of $z$, and let $f(x,y) = ax^2 + bxy + cy^2$, defined up to scaling by $\RR^*$.  
Another way to recover the form from the lattice is to take $\alpha \ZZ + \beta \ZZ$ to the quadratic form $N(x,y) = N(y \alpha + x \beta)$ where $N(z) = |z|^2$, the square of the complex absolute value.  Thus, the lattice's relationship to the form can be seen in two interesting ways:  first, as the lattice whose bases form the totality of roots of all equivalent forms; and second, as the lattice whose vector lengths are the values of the quadratic form.  It is not perhaps entirely obvious that these are ideas are the same.

\begin{exercise}
	Let $z \in \HH^2_U$.  Define 
	\[
		f(x,y) := \begin{pmatrix} x & y \end{pmatrix} \begin{pmatrix} 1 \\ -z \end{pmatrix} \overline{
				\begin{pmatrix} 1 & -z \end{pmatrix} \begin{pmatrix} x \\ y \end{pmatrix} }.% = (x - zy)\overline{(x-zy)} = (x-zy)(x - \overline{z}y).
			\]
			From this, recover the two interpretations given above of the lattice $\ZZ + z\ZZ$ (which is $\SL_2(\ZZ)$-equivalent to $z\ZZ - \ZZ$) as carrying information about the form $f$.
\end{exercise}

In light of this connection, we will use the following notation:
\[
	\mathcal{Q}_\RR := \left\{ \substack{ \text{positive definite real binary quadratic forms} \\ \text{up to scaling by $\RR^*$} } \right\}
\]
\[
       \mathcal{L}_\RR :=	\left\{ \substack{
		\text{homothety classes of positively oriented rank two $\ZZ$-lattices $\Lambda = \alpha \ZZ + \beta \ZZ$ in $\CC$}
} \right\}
\]

When we wish to focus on integral forms, we define
\[
	\mathcal{Q}_\ZZ := \left\{ \substack{ \text{positive definite primitive} \\ \text{integral binary quadratic forms} } \right\}
\]
\[
       \mathcal{L}_\ZZ :=	\left\{ \substack{
		\text{homothety classes of positively oriented rank two $\ZZ$-lattices $\Lambda = \alpha \ZZ + \beta \ZZ$ in $\CC$,} \\
	\text{where $\alpha/\beta$ is an algebraic number of degree $2$.}
} \right\}
\]

Then our description so far can be filled out to obtain the following classical result.
\begin{theorem}
	\label{thm:quadformbij}
	There is a $\PSL_2(\ZZ)$-equivariant bijection between $\mathcal{Q}_\RR$ and $\mathcal{L}_\RR$, restricting to a bijection between $\mathcal{Q}_\ZZ$ and $\mathcal{L}_\ZZ$.
\end{theorem}

\begin{proof}
	The $\RR$ case is essentially a collection of our work so far.  Observe that the map from forms to lattices, on $\mathcal{Q}_\ZZ$, returns lattices $\ZZ + z\ZZ$ where $z$ is a quadratic algebraic number.  Conversely, if $z$ is quadratic algebraic, then the form has $\QQ$ coefficients and can be normalized to lie in $\mathcal{Q}_\ZZ$.
\end{proof}

In fact, the lattices of $\mathcal{L}_\ZZ$ are exactly those which arise as fractional ideals of imaginary quadratic fields.

\begin{exercise}
The \emph{order} of a rank-two lattice $\Lambda = \alpha \ZZ + \beta \ZZ \subseteq \CC$ is $\ord(\Lambda):= \{ z \in \CC : z \Lambda \subseteq \Lambda \}$.  
	Suppose the lattice is associated under Theorem~\ref{thm:quadformbij} to an integral quadratic form.  Show that the order is a subring of an imaginary quadratic field.  Which field?
\end{exercise}

\subsection{Diophantine approximation}

We now return to one of the most basic questions of number theory, which can be asked about the real line, but answered with the geometry of the upper half plane.
How do real numbers lie in relation to rational numbers in $\RR$?  

One simple way to answer this question is to describe the decimal system, which is a sort of addressing system for real numbers by successive approximation by rationals with $10$-power denominator.  There are, however, many ways to approximate a real number by rationals.  \emph{Diophantine approximation} asks us to approximate a real number $\alpha$ by the `best' rationals $p/q$ in the sense that $|p/q - \alpha|$ is small while $|q|$ is simultaneously small.  One way to measure the quality of an approximation is to study the quantity
\[
   -\log_q\left( \min_{p} \left|\frac{p}{q} - \alpha \right|\right).
\]
Since we can always expect an approximation to within $\frac{1}{q}$, this quantity is bounded below by $1$.  But we can often do significantly better.  
Phrased in what is sometimes a more traditional way, we ask, for an exponent $\eta$, whether we can find, or can find infinitely many, $p/q$ satisfying
\[
	\left| \frac{p}{q} - \alpha \right| < \frac{1}{q^\eta}.
\]
Further evidence for the appropriateness of this measure of `goodness' of an approximation is given by Dirichlet's Theorem.
 
\begin{theorem}[Dirichlet]
  \label{thm:dirichlet}
  Let $\alpha \in \RR$.  Then $\alpha$ is irrational if and only if there exist infinitely many distinct $p/q \in \QQ$ such that
  \begin{equation}
	  \label{eqn:dirichlet}
	  \left| \frac{p}{q} - \alpha \right| < \frac{1}{q^2}.
  \end{equation}
\end{theorem}

We interpret this as saying that rational numbers are ``poorly approximable'' and irrationals are ``well approximable.''  

\begin{proof}
	Consider an irrational $\alpha \in \RR$.  Divide the unit interval $[0,1]$ into $Q$ even subintervals, where $Q > 0$ is an integer.  Amongst the $Q+1$ real numbers $0, \alpha, 2\alpha, \ldots, Q\alpha$, there must be two whose fractional parts\footnote{The fractional part of $x$ is $x - \lfloor x \rfloor$, the number $x$ minus the largest integer less than or equal to $x$.  It lies in the unit inteval $[0,1)$.} fall into the same subinterval.  Call these $i\alpha$ and $j\alpha$, where $0 \le i < j \le Q$.
	Then
	\[
		\left| (j-i)\alpha - p \right| < \frac{1}{Q}
	\]
	for an appropriate choice of integer $p$.  Then let $q:= j-i \le Q$ and
	\[
		\left| \frac{p}{q} - \alpha \right| < \frac{1}{qQ} \le \frac{1}{q^2}.
	\]
	Thus we have found one example.  To generate another example, distinct from any $p_i/q_i$ that may have come before, we choose $Q'$ so that
	\[
		\left| p_i - q_i\alpha \right| > \frac{1}{Q'}
	\]
	for all $i$ (notice that this is possible because $\alpha$ is irrational, so $q_i\alpha$ is never an integer), and then repeat the argument above.  In this way, we generate infinitely many distinct $p/q \in \QQ$ having the desired property.

	On the other hand, rational numbers `repel' one another in the sense that for any distinct $p_1/q_1, p_2/q_2 \in \QQ$,
\[
	\left|  \frac{p_1}{q_1} - \frac{p_2}{q_2} \right| \ge \frac{1}{q_1q_2}.
      \]
	In particular, for $q_2 > q_1$,
\[
	\left|  \frac{p_1}{q_1} - \frac{p_2}{q_2} \right| > \frac{1}{q_2^2}.
      \]
      This means that if $\alpha$ is rational, \eqref{eqn:dirichlet} can only be satisfied finitely many times.
\end{proof}

The theorem is illustrated in Figure~\ref{fig:ratsDisks}.
Dirichlet's theorem is sharp in the sense that it fails for any exponent exceeding $2$ on the right side.  

\begin{figure}
	  \includegraphics[width=5.5in]{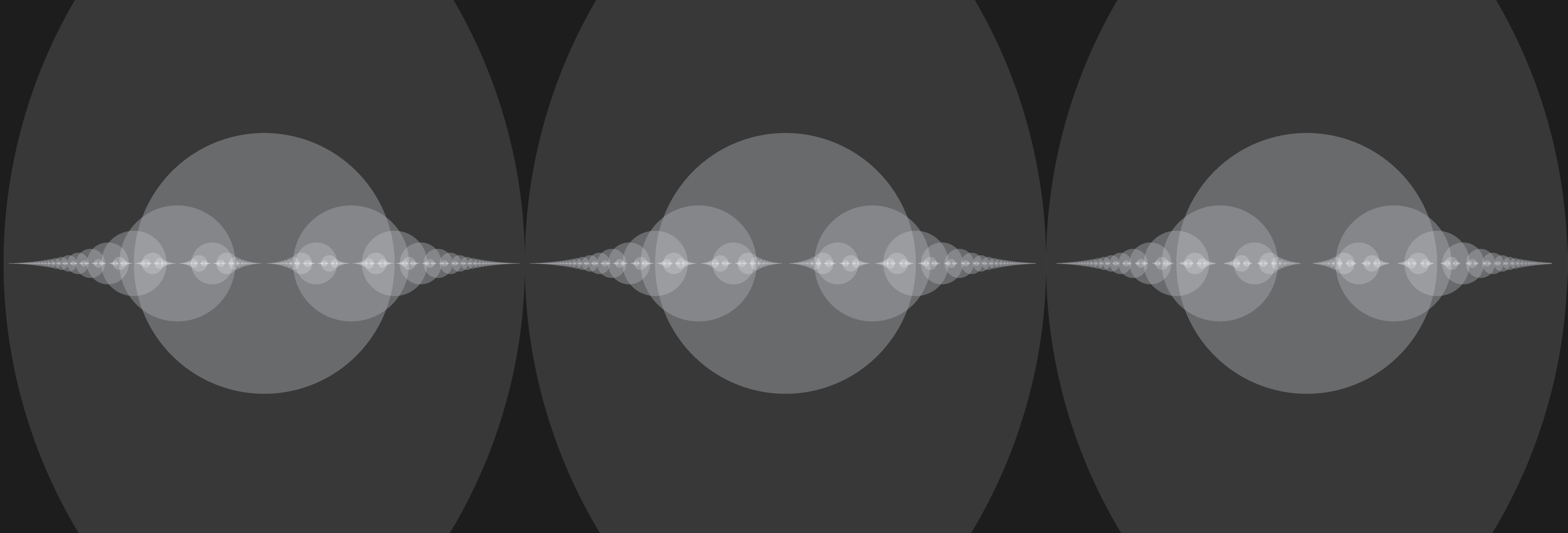}
	  \caption{Translucent white disks of radius $1/q^2$ at $p/q$ on a black background (so a lighter colour indicates an overlap of more disks).  The real line passing horizontally through the middle of the picture ranges from $-1$ to $2$; each integer appears as a distinct `pinch point.'  Dirichlet's theorem states that the points lying under infinitely many disks are exactly the irrationals.  (The illusory three `eggs' are actually pairwise overlaps of the disks of radius $1$ centred at each integer.) Image:  Edmund Harriss, Steve Trettel and Katherine E. Stange.}
	  \label{fig:ratsDisks}
\end{figure}

\begin{exercise}
	We will show that the golden ratio $\alpha = \frac{1+\sqrt{5}}{2}$ is particularly poorly approximable\footnote{Enjoy this: \url{https://slate.com/technology/2021/06/golden-ratio-phi-irrational-number-ellenberg-shape.html}}.  Let $f$ be the minimal polynomial for $\alpha$.  Let $p/q \in \QQ$.  Obtain a lower bound for $f(p/q)$ in terms of $q$.  Factoring $f(p/q) = (p/q - \alpha)(p/q - \overline\alpha)$, and bounding the second factor above, prove that there exists some constant $K$ such that $\left| \frac{p}{q} - \alpha \right| \ge \frac{1}{Kq^2}$ for all but finitely many rationals $p/q$.
Fun fact:  $\alpha$ is best approximated by ratios $F_n/F_{n-1}$ where $F_n$ is the Fibonacci sequence.
\end{exercise}

\subsection{The Farey subdivision}

How do we find the set of `good approximations,' that is to say, solutions to \eqref{eqn:dirichlet}?  The pigeonhole principle proof above provides an algorithm, albeit a slow one.  There is a geometric story, however, which gives rise to a very efficient algorithm:  the theory of continued fractions.  My favourite version of this story is due to Caroline Series \cite{SeriesIntel, Series}.

	\begin{figure}
	  \includegraphics[width=3.5in]{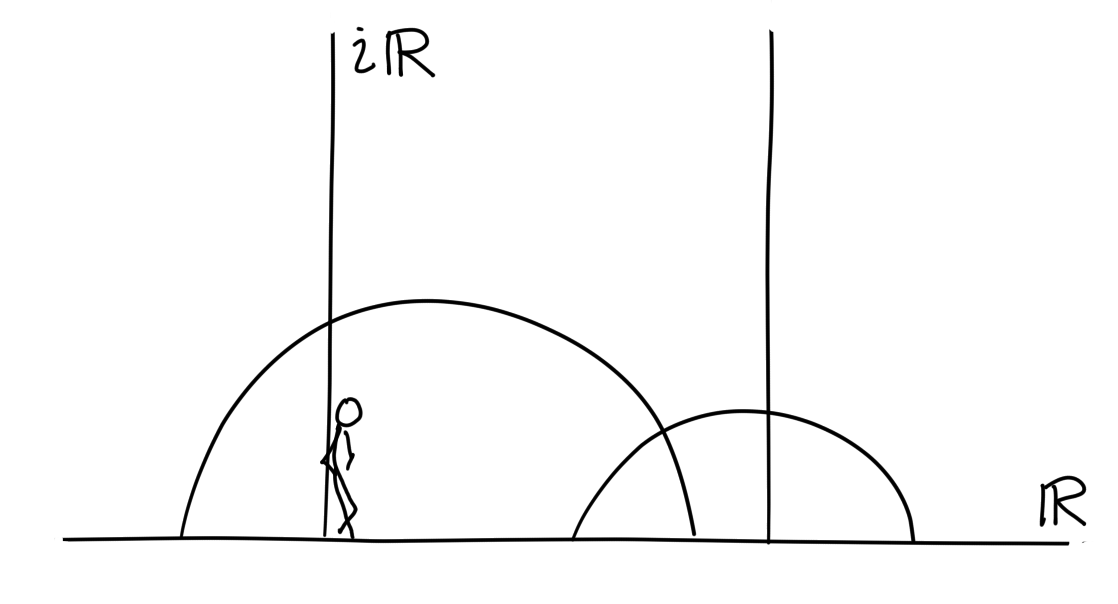}
	  \caption{The upper half plane, with example geodesics.}
	  \label{fig:upper}
  \end{figure}

We need a little of the geometry of $\HH^2_U$ here, which we will detail further in a subsequent section.  In particular, $\HH^2_U$ is a model of the hyperbolic plane in which geodesics are exactly the straight vertical lines, and the upper half-circles centered on the real line (see Figure~\ref{fig:upper}).  
For Series' story, we study the \emph{Farey Tesselation}, shown in Figure~\ref{fig:farey}.  To generate this image, do the following:
\begin{enumerate}
	\item draw vertical lines upward from each integer;
	\item connect each pair of consecutive integers by a geodesic;
	\item for each pair of rational numbers $p_1/q_1$ and $p_2/q_2$ connected by a geodesic, define their \emph{mediant} $(p_1+p_2)/(q_1+q_2)$ and draw a geodesic connecting $p_1/q_1$ to the mediant, as well as a geodesic connecting $p_2/q_2$ to the mediant;
	\item repeat the last step, ad infinitum.
\end{enumerate}

%But it can also be described in an equivalent way.  Drawing an arc from $0$ to $1$ to start, we apply the following recursive process.  For each arc from $a/b$ to $c/d$, choose the point $(a+c)/(b+d)$, called the \emph{mediant}, which will lie between $a/b$ and $c/d$, and then draw arcs from $a/b$ to $(a+c)/(b+d)$ and from $(a+c)/(b+d)$ to $c/d$.  

This mediant operation is actually quite natural if we view $\QQ \cup \{ \infty \}$ as $\PP^1(\ZZ)$, the projectivization\footnote{For any ring $R$, my notation for a projective space is: $\PP^n(R) := \{ \mathbf{x} := (x_1,\ldots,x_n) \neq 0 : x_i \in R \}/(\mathbf{x} = \lambda \mathbf{x}, \lambda \in R^*)$.  In topological or geometric contexts, authors write $F\PP^n$ for the projective space over the field $F$.} of the square lattice $\ZZ^2$.  Each rational number, by taking it in reduced form, corresponds to an element of $\ZZ^2$ with a sightline to the origin (i.e. a primitive vector).  Then the mediant operation is vector addition of these primitive vectors (reduced fractions).  One way to `see' this is to consider the lattice $\ZZ^2$ in the plane.  Standing at the origin and looking out, we see the vertices of the lattice as a copy of $\PP^1(\ZZ)$ (see Figure~\ref{fig:ratsLattice}). The `nearby' points are those with small vector entries, and these occur earlier in the Farey subdivision.  Visually, if the lattice points are indicated by spheres, they would appear as larger dots (compare Figures~\ref{fig:ratsLattice} and \ref{fig:ratsDots}).  

This process actually generates all the rational numbers, as the following exercise demonstrates.

\begin{figure}
	  \includegraphics[width=5.5in]{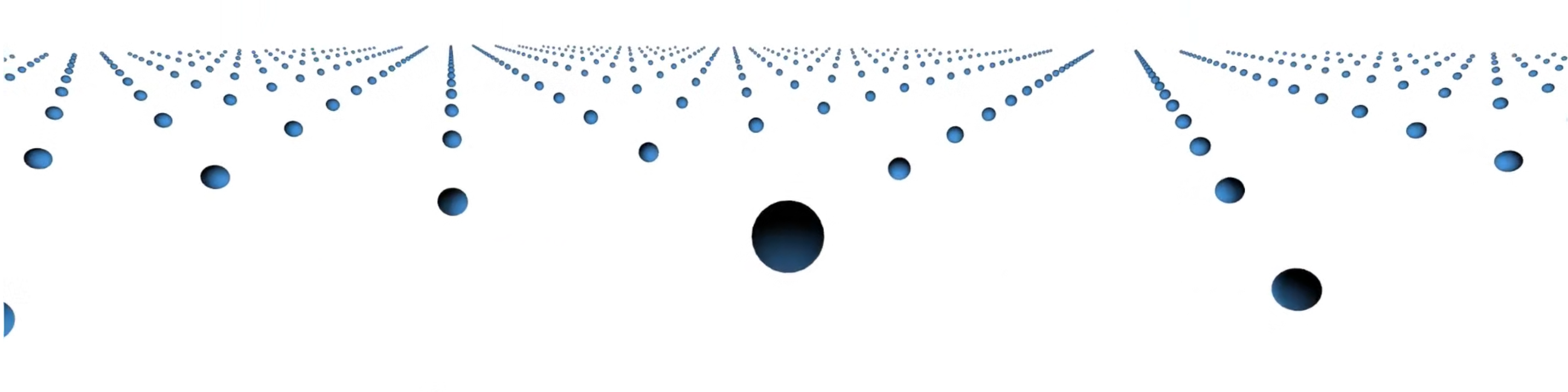}
	  \caption{The view floating slightly above the origin of $\ZZ^2$.}
	  \label{fig:ratsLattice}
\end{figure}

\begin{figure}
	  \includegraphics[width=5.5in]{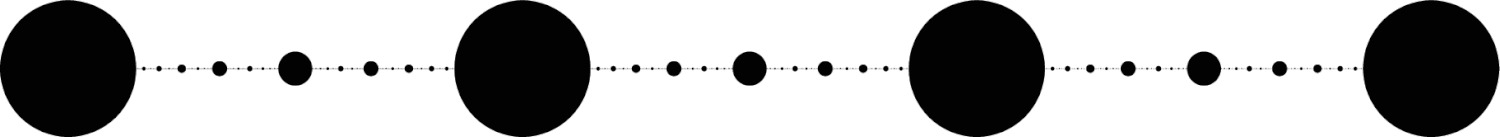}
	  \caption{The rationals $p/q$ indicated with disks of radius $\frac{3}{20q^2}$; similar to the view of $\ZZ^2$ from the origin in Figure~\ref{fig:ratsLattice} (Image: \cite[Figure 15]{HST}).}
	  \label{fig:ratsDots}
\end{figure}

\begin{exercise}
	\label{exercise:euclidgauss}
	Begin with any positive rational number in lowest form $p/q$.  Show that there exists $u/v \in \QQ$ such that $pv-uq = 1$.  Then perform a Euclidean-style algorithm on the vectors $\begin{pmatrix}p \\ q\end{pmatrix}$ and $\begin{pmatrix} u \\ v \end{pmatrix}$, repeatedly substracting one from the other until attaining the pair of standard basis vectors for $\ZZ^2$.  Explain why this is a proof that the Farey tesselation generates all rational numbers as geodesic endpoints.
\end{exercise}

	In this image, there is a `bubble,' or geodesic, lying above the unit interval $[0,1]$ and then it is further subdivided into two bubbles, each of which is further subdivided, etc. 
	This gives a way to `organize' the real line:  compare this to the more common decimal subdivision in Figure~\ref{fig:compare}.  In either system, we can specify the location of a real number by indicating which interval it lies in, then which subinterval of that, etc., iteratively.  In the decimal system, we choose the leftmost endpoint of each successive interval, and call the resulting sequence the decimal expansion.  

\begin{figure}
	\includegraphics[width=6in]{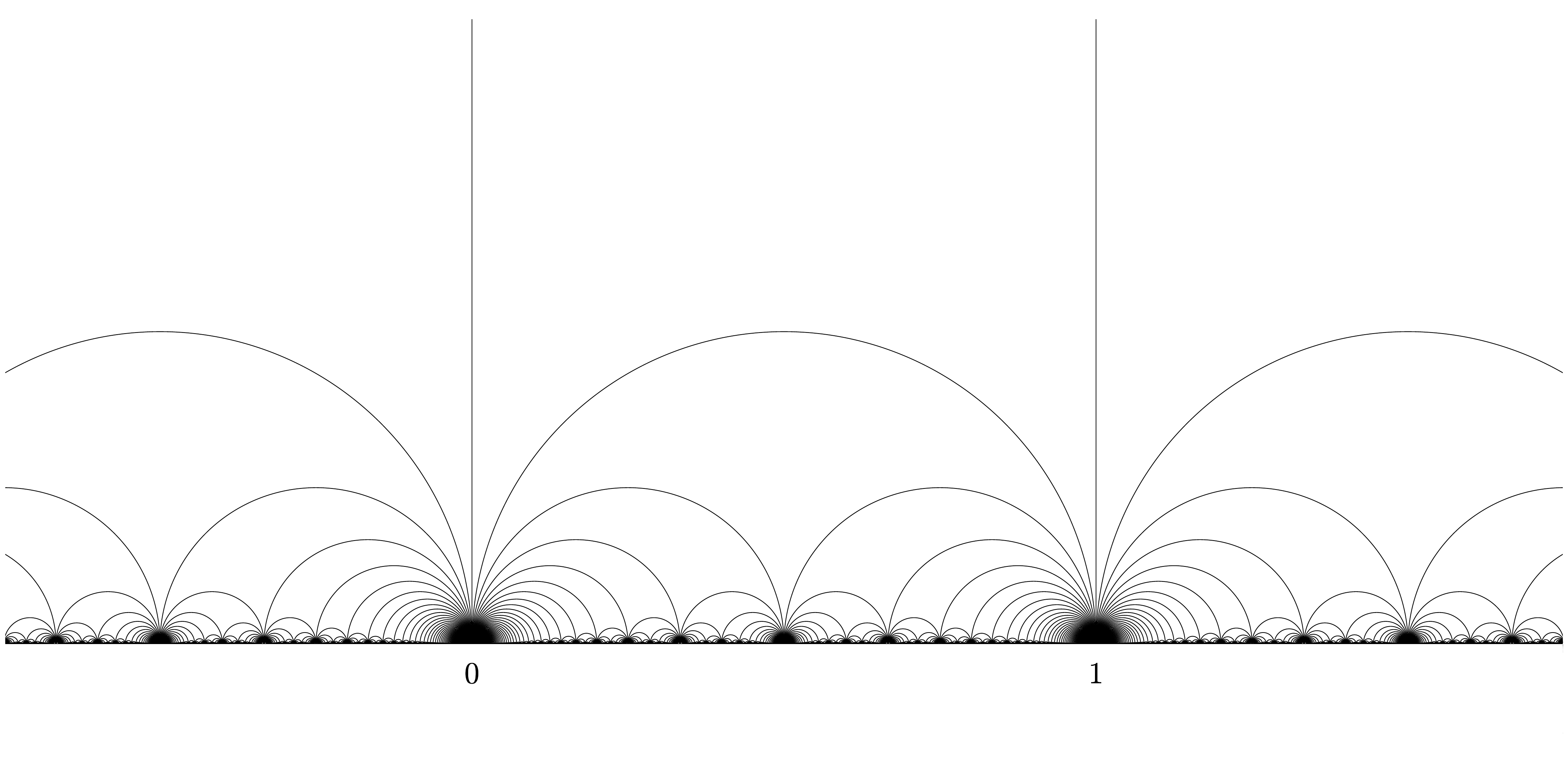}
	\caption{The Farey subdivision of the upper half plane, i.e. orbit of the geodesic from $0$ to $\infty$ under $\PSL_2(\ZZ)$.}
	\label{fig:farey}
\end{figure}

\begin{figure}
	\includegraphics[width=2.8in]{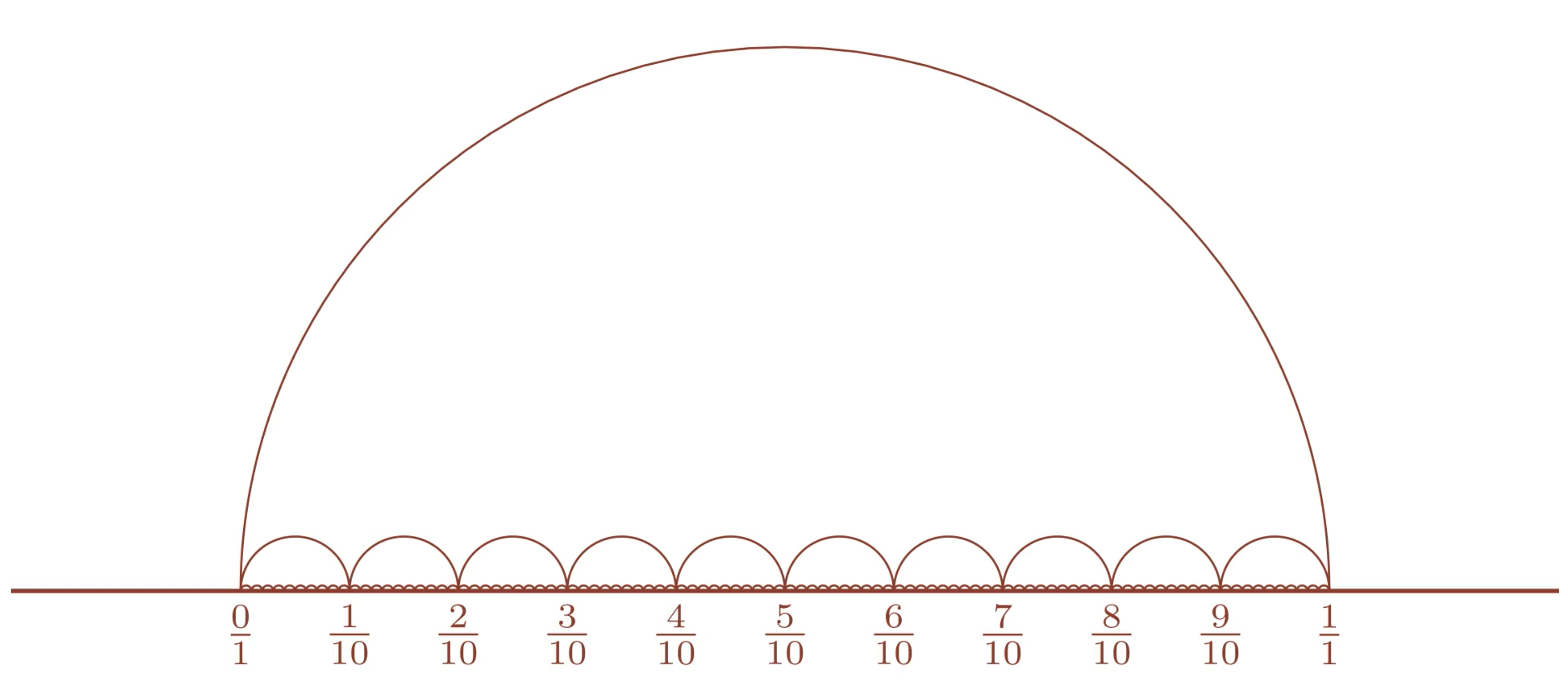} \;\;
	\includegraphics[width=2.8in]{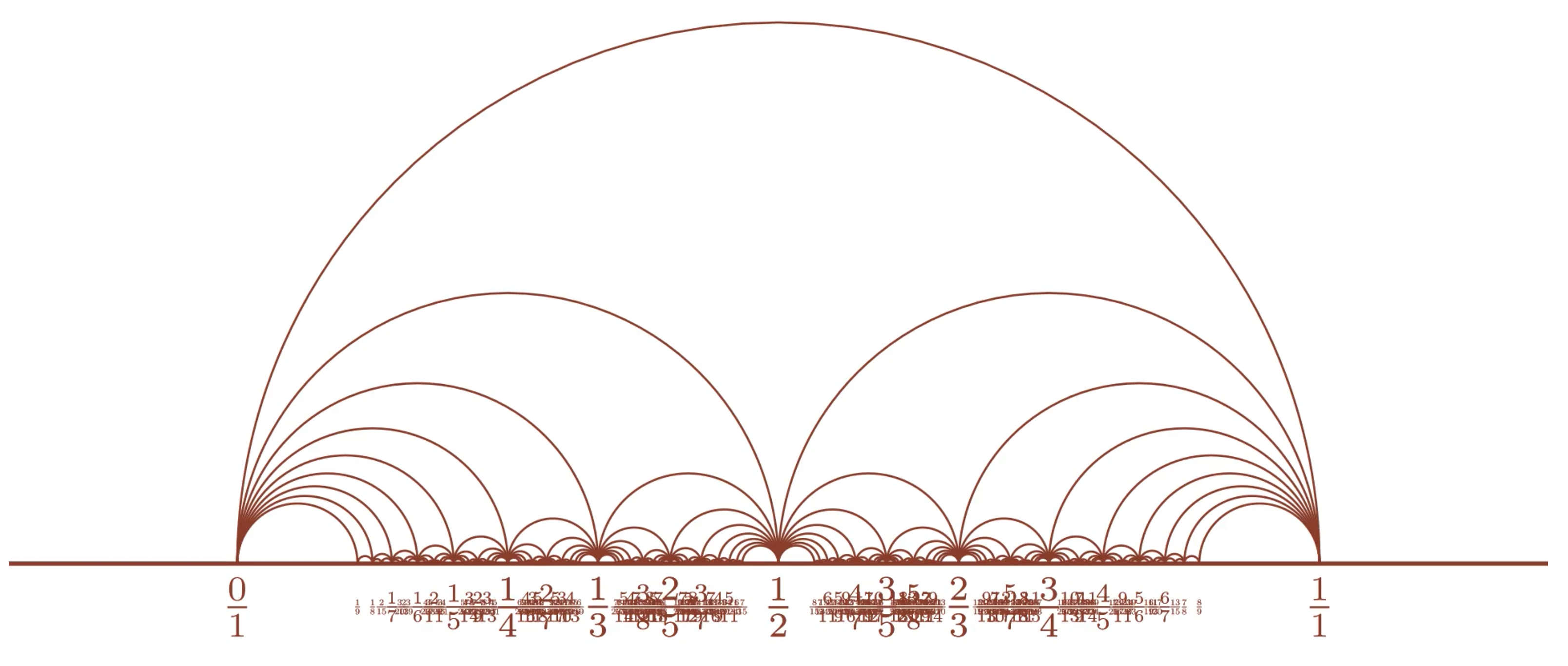}
	\caption{At left, the decimal subdivision of the unit interval, to three levels; at right the Farey subdivision, to nine levels.}
	\label{fig:compare}
\end{figure}

What if we use the Farey subdivision?  We obtain the \emph{continued fraction expansion}.  
To describe the `path' to $\alpha$ through the Farey `froth' of semicircles, we use the language of geodesics.  It was long known that continued fractions were connected to the paths of geodesics in the upper half plane and to $\PSL_2(\ZZ)$, but it was Caroline Series who realized that the Farey subdivision is perhaps the most natural way to describe this relationship.

\begin{figure}
	  \includegraphics[width=1.5in]{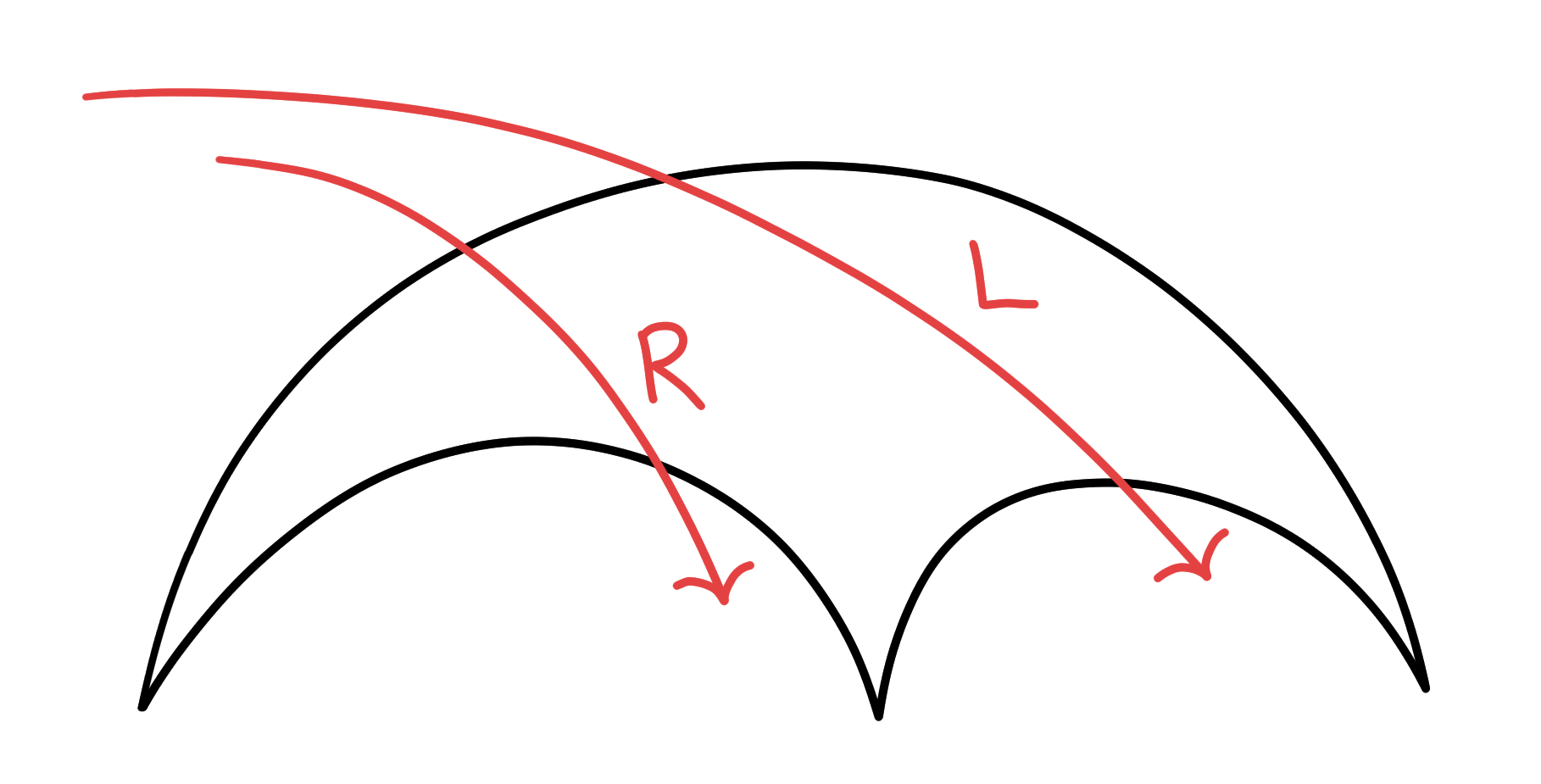}
	  \caption{The upper geodesic cut (of this hyperbolic triangle) is labelled L, since the region to the left (viewed by an ant riding the geodesic flow) has one cusp.  We similarly label the lower cut R.  As we look down at the page, this appears counter-intuitive, but to the ant crawling along the page in the direction indicated by the geodesic, this is the intuitive left or right forking choice.}
	  \label{fig:cutting}
\end{figure}

\begin{figure}
	  \includegraphics[width=5.5in]{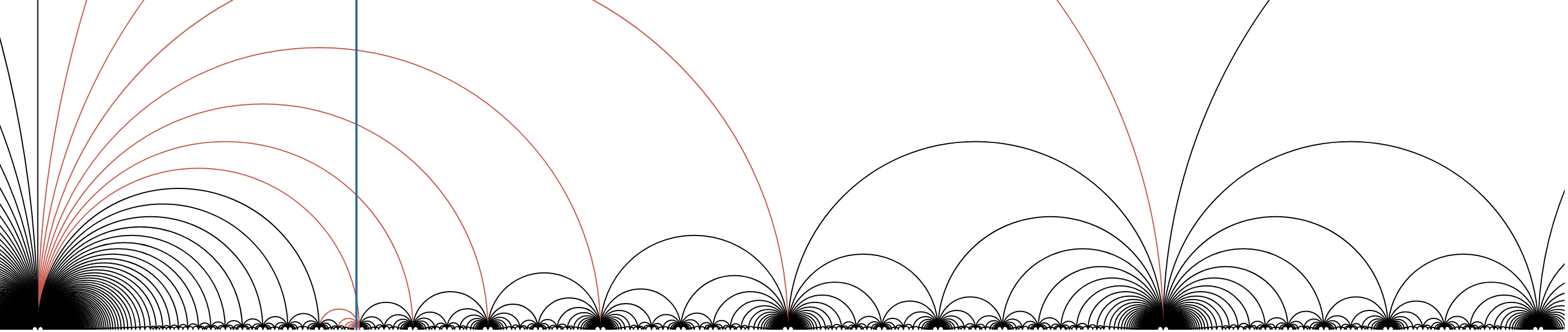}\\
	  \includegraphics[width=5.5in]{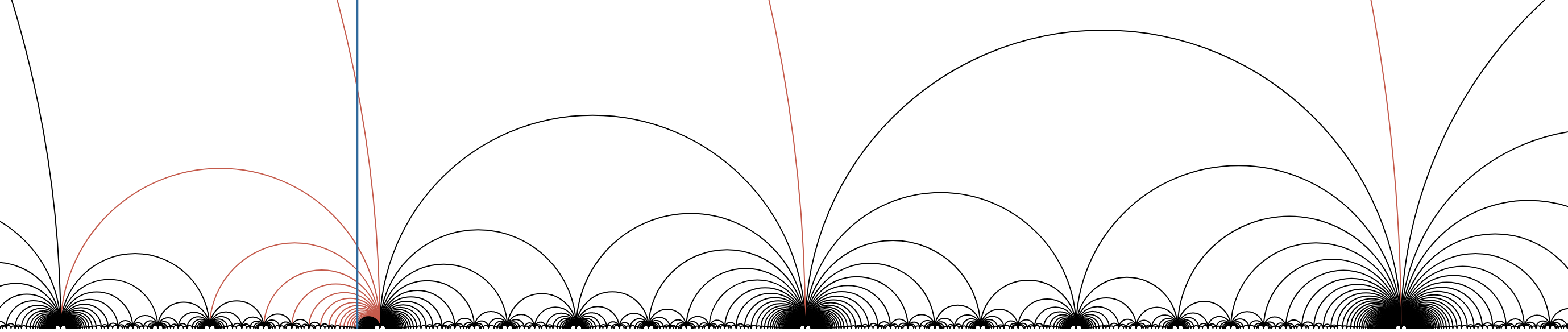}
	  \caption{An example cutting sequence for $\pi$, as in Theorem~\ref{thm:cutting}, showing the first two convergents, $3$ and $22/7$.  Top image:  the thick dark cluster at left is the integer $3$ and the vertical line is indicating $\pi$.  Bottom image:  a detail of the first image, where we see $\pi$ just to the left of $22/7$.  In both images, the red geodesics indicate the subdivisions crossed as we approach $\pi$.  The greater the number of red geodesics converging to a point, the better the approximation (and the larger the corresponding continued fraction coefficient).  After using $L^3$ to enter the region between $3$ and $4$, the top image shows $R^7$ (red lines emanating from $3$ are the walls crossed); the bottom image shows $L^{15}$.  The continued fraction expansion of $\pi$ begins $L^3R^7L^{15}R^{1}L^{292}R^{1}L^{1}R^{1} \cdots$.}
	  \label{fig:cutting-example}
\end{figure}

	\begin{definition}
		Consider any geodesic, heading out from the imaginary axis to a positive real number $\alpha$.  Label it by L or R as it passes through each region of the Farey tessellation (or subdivision), as follows.  The geodesic cuts each triangular fundamental region it passes through into two parts, one having one cusp, and one having two.  Label it L if the region to the left of the geodesic (from the perspective of its direction of travel) has one cusp, and R otherwise (Figure~\ref{fig:cutting}).  If the geodesic heads directly into a cusp, label with either an L or an R.  The resulting word (generated from left to right), is the \emph{Series continued fraction}.
	\end{definition}

The sequence $L^{a_0} R^{a_1} \cdots $ is an example of a \emph{cutting sequence}.   See Figure~\ref{fig:cutting-example} for an example.

\begin{exercise}
	\label{ex:expansions1}
	Determine the beginning of the continued fraction expansion of $e$, and the two full expansions for $17/5$.   
\end{exercise}

Returning to the lattice interpretation, the element $\alpha \in \RR$ corresponds to a ray from the origin at slope $\alpha$.  If $\alpha$ is irrational, it will never hit a lattice point.  However, there are `near misses:' lattice points that it passes close to as it heads out from the origin.  These are the corners of the fundamental triangles of the Farey tessellation at which we `turn' left or right. 

\subsection{The matrix continued fraction expansion}

To formalize the continued fraction process and the good rational approximations it will produce, we consider $\SL_2^+(\ZZ)$ (the $+$ indicates that we include only those matrices whose entries are non-negative) acting on $\PP^1(\QQ)$.  This monoid (group without inverses), or semigroup\footnote{A semigroup is a group without the requirement for inverses or an identity.  In the literature of continued fractions, $\SL_2^+(\ZZ)$ is often called a semigroup, although we would normally include the identity.}, is generated from the identity matrix by two matrices:
\[
	M_L = \begin{pmatrix} 1 & 1 \\ 0 & 1 \end{pmatrix}, \quad
	M_R = \begin{pmatrix} 1 & 0 \\ 1 & 1 \end{pmatrix}.
\]
Phrased another way, $\SL_2^+(\ZZ)$ is exactly the collection of matrices formed as finite words (including the empty word) in the words $M_L$ and $M_R$.
Form a tree whose root is the identity matrix, and for each matrix, the left child is obtained by multiplying on the right by $M_L$ and the right child is obtained by multiplying on the right by $M_R$, as illustrated in Figure~\ref{fig:matrices}.  

\begin{exercise}
	Prove that $\SL_2^+(\ZZ)$ is generated as a semigroup by $M_L$ and $M_R$.  (See Exercise~\ref{exercise:euclidgauss}.)
	Show that $\PSL_2(\ZZ)$ is generated as a group from $\SL_2^+(\ZZ)$ by the addition of $\begin{pmatrix} 0 & -1 \\ 1 & 0 \end{pmatrix}$.
\end{exercise}

\begin{figure}
\tiny
\[
        \xymatrix @R=1.25pc @C=-1.40pc {
                & & & & & & & &  {\begin{pmatrix} 1 & 0 \\ 0 & 1 \end{pmatrix}} \ar@{-}[lllld] \ar@{-}[rrrrd] & & & & & & & & \\
                & & & & {\begin{pmatrix} 1 & 1 \\ 0 & 1 \end{pmatrix}}  \ar@{-}[lld] \ar@{-}[rrd] & & & & & & & & {\begin{pmatrix} 1 & 0 \\ 1 & 1 \end{pmatrix}}   \ar@{-}[lld] \ar@{-}[rrd] & & & & & & & & & \\
                & &
                {\begin{pmatrix} 1 & 2 \\ 0 & 1 \end{pmatrix}} 
                \ar@{-}[ld] \ar@{-}[rd] 
                & & & &
                {\begin{pmatrix} 2 & 1 \\ 1 & 1 \end{pmatrix}}
                \ar@{-}[ld] \ar@{-}[rd] 
                                   & & & &
                {\begin{pmatrix} 1 & 1 \\ 1 & 2 \end{pmatrix}} 
                \ar@{-}[ld] \ar@{-}[rd] 
                                   & & & &
                {\begin{pmatrix} 1 & 0 \\ 2 & 1 \end{pmatrix}}
                \ar@{-}[ld] \ar@{-}[rd] 
                & & & & &\\
                & {\begin{pmatrix} 1 & 3 \\ 0 & 1 \end{pmatrix}} & &  
                {\begin{pmatrix} 3 & 2 \\ 1 & 1 \end{pmatrix}} & &  
                {\begin{pmatrix} 2 & 3 \\ 1 & 2 \end{pmatrix}} & &   
                {\begin{pmatrix} 3 & 1 \\ 2 & 1 \end{pmatrix}} & &    
                {\begin{pmatrix} 1 & 2 \\ 1 & 3 \end{pmatrix}} & &  
                {\begin{pmatrix} 2 & 1 \\ 3 & 2 \end{pmatrix}} & &  
                {\begin{pmatrix} 1 & 1 \\ 2 & 3 \end{pmatrix}} & &   
                {\begin{pmatrix} 1 & 0 \\ 3 & 1 \end{pmatrix}} & & \\
        }
\]
\normalsize
	\caption{The $\SL_2^+(\ZZ)$ matrix tree, shown to four levels.}
\label{fig:matrices}
\end{figure}
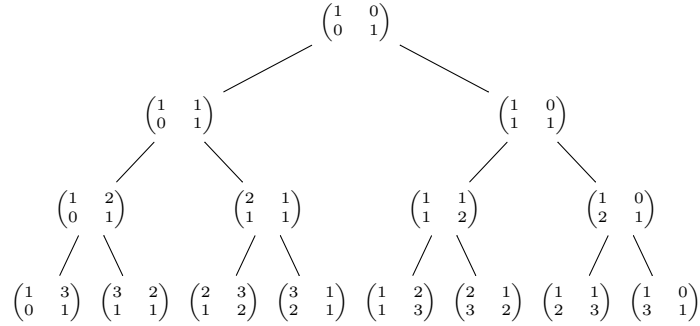

From the previous exercise, one can conclude that the Farey tesselation can equivalently be described as the orbit in $\HH^2_U$ of the geodesic from $0$ to $\infty$ under $\PSL_2(\ZZ)$.

Each geodesic in the picture, once we choose a direction, corresponds to a unique element $\begin{pmatrix} p_1 & p_2 \\ q_1 & q_2 \end{pmatrix}$ by writing its rational endpoints $p_1/q_1$ and $p_2/q_2$ as column vectors.  In fact, this matrix lies in $\PSL_2(\ZZ)$ as the image of the vertical imaginary axis (in a downward direction), which corresponds to $\begin{pmatrix} 1 & 0 \\ 0 & 1 \end{pmatrix}$, connecting $0$ to $\infty$.  Therefore, the pairs of rational numbers connected by a geodesic have the property that $p_1q_2 - p_2q_1 = 1$, which earns them the name of a \emph{unimodal pair}.

The continued fraction of a real number $\alpha$ is more commonly defined by an iterative algebraic process, generating an expression
\begin{equation}
	\label{eqn:ctd-trad}
		\alpha = a_0 + \dfrac{1}{ a_1 + \dfrac{1}{a_2 + \dfrac{1}{a_3 + \cdots }}}.
\end{equation}
The finite truncations of this expression give rational numbers
\[
		p_k/q_k = a_0 + \dfrac{1}{ a_1 + \dfrac{1}{a_2 + \dfrac{1}{a_3 + \cdots + 1/a_k}}},
\]
called the \emph{convergents} of $\alpha$.  The $a_i$ are called the \emph{partial quotients} or \emph{coefficients}.

Now a cutting sequence through the Farey tessellation can be interpreted as a word in $M_R$ and $M_L$, by simply replacing $R \leftarrow M_R$ and $L \leftarrow M_L$. 
This is justified by the observation that if a triangular region $T$ has endpoints $\a_1/b_1$, $(a_1+a_2)/(b_1+b_2)$, $a_2/b_2$, then the triangular region $T M_L$ (interpreting the matrix as a M\"obius transformation) has triangular endpoints $a_1/b_1$, $(2a_1+a_2)/(2b_1+b_2)$, $(a_1+a_2)/(b_1+b_2)$, i.e., it corresponds to the left sub-interval.

The following proposition now states that the `traditional' expansion \eqref{eqn:ctd-trad} is the same as Series' geodesic expansion interpreted as matrices.

\begin{theorem}[{Series \cite{Series}}]
	\label{thm:cutting}
	Suppose 
	\[
		L^{a_0} R^{a_1} L^{a_2} R^{a_3} \cdots
	\]
	is the L,R-sequence for a geodesic originating on the imaginary axis and ending at $\alpha \in \RR$.  Then a classical continued fraction expansion of $\alpha$ is
	\[
		a_0 + \dfrac{1}{ a_1 + \dfrac{1}{a_2 + \dfrac{1}{a_3 + \cdots }}}.
	\]
	In more detail, if $\alpha$ has cutting sequence $L^{a_0} R^{a_1} \cdots$, define $M_n := M_L^{a_0} M_R^{a_1} \cdots M_*^{a_n} \in \SL_2^+(\ZZ)$.  Then the convergents of the continued fraction expansion of $\alpha$ are given by
	\[
		\begin{pmatrix}
			p_n \\ q_n
		\end{pmatrix}
		= 
		M_n 
		\begin{pmatrix}
			1 \\ 0 
		\end{pmatrix}, 
		\text{ or }
	M_n 
		\begin{pmatrix}
			0 \\ 1 
		\end{pmatrix}, 
	\]
	depending whether the last letter was $L$ or $R$, respectively.
	Furthermore, if $\alpha$ is rational with lowest terms representation $p/q$, then there are exactly two expansions of $\alpha$ (since heading into a cusp is the last step), and we have
	\[
		\begin{pmatrix} p \\ q \end{pmatrix} =
		\left\{
			\begin{array}{ll}
		L^{a_0} R^{a_1} L^{a_2} R^{a_3} \cdots 
		\begin{pmatrix} 1 \\ 0 \end{pmatrix} & \text{if the expansion ends in R} \\ 
		L^{a_0} R^{a_1} L^{a_2} R^{a_3} \cdots 
		\begin{pmatrix} 0 \\ 1 \end{pmatrix} & \text{if the expansion ends in L} \\ 
	\end{array} \right.
	\]
\end{theorem}

Observe that `heading into the mediant' to end a rational approximation would mean evaluating at the vector $(1,1)^t$, so instead we head one step lower to either left or right; this constitutes the ambiguity in the second part of the theorem.

The following exercise should convince you that Theorem~\ref{thm:cutting} is correct.

\begin{exercise}
	Computationally verify the following:
	\[
		L^k \begin{pmatrix} 0 \\ 1 \end{pmatrix} = \begin{pmatrix} k \\ 1 \end{pmatrix}, \quad
		L^kR^\ell \begin{pmatrix} 1 \\ 0 \end{pmatrix} = \begin{pmatrix} 1 + \ell k \\ \ell \end{pmatrix}, \quad
		L^kR^\ell L^t \begin{pmatrix} 0 \\ 1 \end{pmatrix} = \begin{pmatrix} t + k + k\ell t \\ \ell t + 1 \end{pmatrix}.
	\]
	Verify also that
	\[
		\frac{k}{1} = k, \quad
		\frac{1 + \ell k}{\ell} = k + \frac{1}{\ell}, \quad
		\frac{t + k + k\ell t}{\ell t + 1} = k + \frac{1}{\ell + \frac{1}{t}}.
	\]
\end{exercise}

\begin{exercise}
	\label{ex:expansions1}
	Verify Theorem~\ref{thm:cutting} for the beginning of the continued fraction expansion of $e$, and the two full expansions for $17/5$.  
\end{exercise}

%You can actually assign a two-tailed cutting sequence to any geodesic, starting at any point in the middle, meaning that it is infinite in both directions.  Labelled in that way, it turns out two geodesics are $\PSL_2(\ZZ)-$equivalent if and only if they have the same cutting sequence.

We complete our study of continued fractions with one of the principal reasons they are studied:  that the continued fraction convergents capture good approximations.

\begin{theorem}
	\label{thm:approx}
	Any rational approximation $p/q \in \QQ$ to $\alpha \in \RR$ satisfying $\left| \frac{p}{q} - \alpha \right| < \frac{1}{2q^2}$ is a continued fraction convergent of $\alpha$.
\end{theorem}

\begin{exercise}
	To prove Theorem~\ref{thm:approx}, do the following:
	\begin{enumerate}
		\item Let $p/q$ be a good approximation in the sense of Theorem~\ref{thm:approx}.  Suppose without loss of generality that $p/q < \alpha$.  Show that $p/q$ must be the closest rational to $\alpha$, from below, with denominator $\le q$.
		\item Let $p'/q'$ be the closest rational to $\alpha$, from above, with denominator $\le q$.
		\item Conclude that $p/q$ and $p'/q'$ are a unimodal pair and therefore the endpoints of a bubble in the Farey tesselation containing $\alpha$.
		\item Now consider the mediant $M$ between these two rational numbers.  Show that it must lie to the \emph{right} of $\alpha$, by measuring the distance between $p/q$ and $M$ (and using the good approximation property of $p/q$).
		\item Show that if $\alpha$ lies in a bubble of the Farey tesselation, then the bubble endpoint with smallest denominator is a convergent.
	\end{enumerate}
\end{exercise}

\subsection{Indefinite quadratic forms and real quadratic irrationalities}

Recall that \emph{definite} quadratic forms have a reduction theory (Exercise~\ref{ex:reduc} and Figure~\ref{fig:sl2}).  In short, reduction of a quadratic form is a process of navigating through $\PSL_2(\ZZ)$ to obtain a reduced representative of the equivalence class. 
There is also a reduction theory of indefinite quadratic forms, and it can be explained with the Farey tesselation and is closely related to continued fractions.   We turn to this now.

Primitive integral indefinite quadratic forms $ax^2 + bxy + cy^2$, $a,b,c \in \ZZ$, correspond to a conjugate pair of roots $\alpha, \overline\alpha = \frac{-b \pm \sqrt{\Delta}}{2a}$ in the real line.  
Such a pair $\alpha$, $\overline\alpha$ is called \emph{reduced} if (up to swapping the two roots), 
\[
	\overline\alpha < -1 < 0 < \alpha < 1.
\]
It is an elementary exercise to show that this is equivalent to the conditions
\[
	0 \le b < \sqrt{\Delta}, \quad 0 < \sqrt{\Delta} - b \le 2a \le \sqrt{\Delta} + b,
\]
or, to the satisfyingly simple 
\[
	b > |a+c|, \quad ac < 0.
\]
\begin{exercise}
	Show the equivalences just mentioned.
\end{exercise}

In particular, fixing $\Delta$, there are only finitely many possible values for $b$, therefore only finitely many possible $a$, and for each such pair only one possible $c$.  So there are only finitely many reduced forms for each discriminant.

Next, observe that the transformation
\[
 	(\alpha, \overline\alpha) \mapsto
	\left(
	\frac{1}{\alpha} - \left\lfloor \frac{1}{\alpha} \right\rfloor,
	\frac{1}{\overline\alpha} - \left\lfloor \frac{1}{\alpha} \right\rfloor
	\right),
\]
takes $(\alpha, \overline\alpha)$ to another reduced pair.  
Furthermore, under the reduced assumption, this operation on pairs is a bijection, with inverse
\[
	(\alpha, \overline\alpha) \mapsto
		\left(
	\frac{1}{\alpha -1 + \lfloor \overline\alpha \rfloor}, 
	\frac{1}{\overline\alpha - 1 + \lfloor \overline\alpha \rfloor}
	\right).
\]
Thus, the reduced forms fall into cycles, where each step is realized by a bespoke M\"obius map (namely, $z \mapsto 1/z - \lfloor 1/\alpha \rfloor$).

\begin{proposition}
	\label{prop:reduc}
	For each primitive integral indefinite binary quadratic form $ax^2 + bxy + cy^2$ corresponding to pair $(\alpha, \overline\alpha)$, there is a reduced form with pair $(\alpha', \overline\alpha')$ such that $\alpha = M \alpha'$ where $M \in \PSL_2^+(\ZZ)$.
\end{proposition}

\begin{proof}
	We prove it only in the case that $\alpha > \overline\alpha > 0$ (other cases can be reduced to this case).
	There is some triangle of the Farey tesselation with corners $a$, $b$, $c$ such that $0 < a < \overline\alpha < b < \alpha < c$.  Then, by the transitivity of the Farey tesselation, this triangle is an image under some $M_0 \in \PSL_2^+(\ZZ)$ of the triangle with corners $0$, $1$ and $\infty$, in such a way that $0 < \overline\alpha_0 < 1 < \alpha_0 < \infty$, and $M_0(\alpha_0) = \alpha$, $M_0(\overline\alpha_0) = \overline\alpha$.  By composing with a translation, we can assume instead that $\overline\alpha_0 < 0 < \alpha_0 < 1$.

	If $\overline\alpha_0 < -1 < 0 < \alpha_0 < 1$, we are done.  If not, then $-1 < \overline\alpha_0 < 0 < \alpha_0 < 1$, and there is some $\alpha_1, \overline\alpha_1$ with $\overline\alpha_1 < -1 < 0 < \alpha < 1$ and $M_1 \in \PSL_2(\ZZ)$ such that $M_1(\alpha_1) = \alpha_0$, $M(\overline\alpha_1) = \overline\alpha_0$; in fact $M_1 = \begin{pmatrix} 0 & -1 \\ 1 & 0 \end{pmatrix} \begin{pmatrix} 1 & b \\ 0 & 1 \end{pmatrix}$, $b \ge 1$ will suffice.  Note that although $M_1 \notin \PSL_2^+(\ZZ)$, the composition $M_0M_1 \in \PSL_2^+(\ZZ)$.
		The proof is illustrated in Figure~\ref{fig:reduction}.
\end{proof}

	\begin{figure}
	  \includegraphics[width=0.95\textwidth]{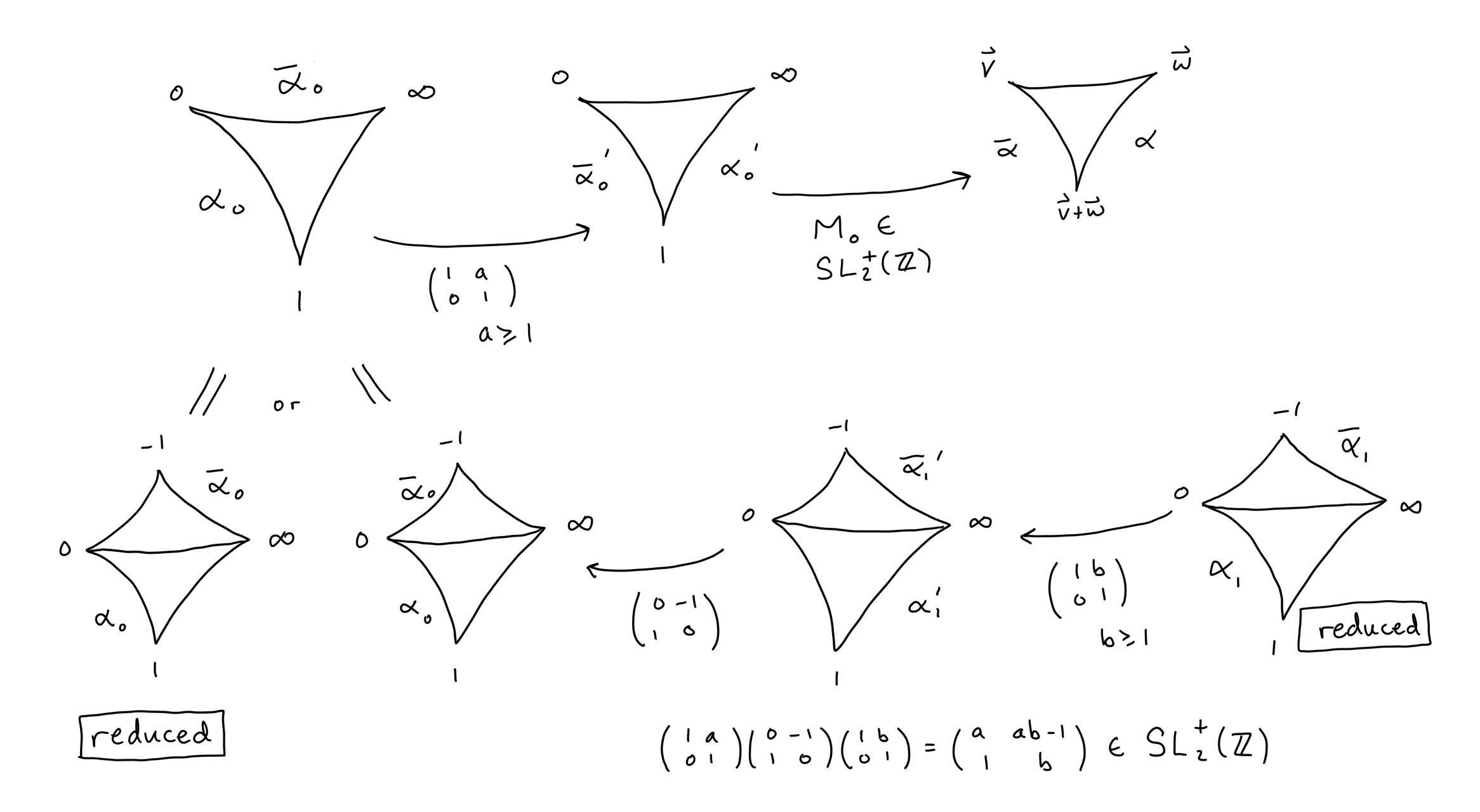}
		\caption{A flow chart that relates an indefinite form to a reduced one via a transformation from $\SL_2^+(\ZZ)$, illustrating the proof of Proposition~\ref{prop:reduc}.}
	  \label{fig:reduction}
  \end{figure}

Combining the proof above, which gives the `pre-periodic' portion of a continued fraction expansion, with the observed cycles of reduced forms, we obtain a corollary.

\begin{corollary}
	Let $\alpha$ be a real quadratic irrational.  Then the continued fraction expansion of $\alpha$ is eventually periodic, and perfectly periodic if $\alpha$ is reduced.
\end{corollary}

\begin{exercise}
	Provide the details of the corollary above.  Amongst these details is the converse that any eventually periodic continued fraction converges to a real quadratic irrational.
\end{exercise}

\begin{exercise}
	Determine the automorphisms of an indefinite form, by using the fact that such an automorphism is a M\"obius transformation fixing the roots.  Show that such automorphisms are given by solutions to the Pell\footnote{As is the case with so much mathematical nomenclature, this equation has at best a tenuous connection to the person it is named after.  Pell revised a translation of a book that discussed it, approximately two millenia after it was first discussed.} equation $X^2 - \Delta Y^2 = 4$.  Next, observe that the M\"obius transformations that circumnavigate once around a cycle of indefinite forms (as described above) fix a reduced form, and therefore correspond to solutions to Pell's equation.  This leads to an algorithm to compute solutions to Pell's equation by continued fractions.  (Such automorphisms correspond to multiplying the associated ideal of $\QQ(\sqrt{\Delta})$ by a unit.)
\end{exercise}

The Farey diagram has a dual graph called the \emph{topograph}, put to lovely use by Conway \cite{ConwayFung} to prove all the basic facts about binary quadratic forms, both definite and indefinite.  A wonderful description of this point of view is Allen Hatcher's \emph{Topology of Numbers} \cite{Hatcher}.

\subsection{Lagrange spectrum}

Dirichlet's Theorem differentiates the approximation properties of rationals and irrationals.  We might ask a more nuanced question, which is, given a constant $C$, whether for an irrational $\alpha$, there are infinitely many $p/q \in \QQ$ such that
\[
	\left| \frac{p}{q} - \alpha \right| < \frac{C}{q^2}.
\]
For $C=1$, there are infinitely many such approximations.  This remains true as $C$ decreases to $1/\sqrt{5}$, where Hurwitz' Theorem states that it still holds.  But after that, there are only finitely many approximations for any $\alpha$ whose continued fraction ends in all $1$'s.  These are called `noble numbers' and include the golden ratio.  More formally, we can define
\[
	v(\alpha) = \inf \{ C : |\alpha - p/q| < C/q^2 \text{ for infinitely many } p/q \in \QQ \}.
\]
Markoff showed that there is a discrete set of values $v(\alpha)$ above $1/3$; the $\alpha$ that realize these discrete values are called \emph{Markoff irrationalities}, and the values themselves are the discrete part of the \emph{Lagrange spectrum} (sometimes called Markoff spectrum, as for example by Series \cite{SeriesIntel}; the terminology is occasionally murky).

It turns out that the discrete elements of the Markoff spectrum are associated to a simple loop on the punctured torus.  See \cite{Series, SeriesIntel} for the details of this story, but in short, one can obtain the punctured torus as a quotient of $\HH^2_U$ by a subgroup of $\PSL_2(\ZZ)$, and the geodesic in question is the geodesic whose endpoint is the irrational to be approximated, and the cutting sequence in this situation is closely related to the continued fraction expansion.

\subsection{Roth's Theorem}

The Diophantine approximation of algebraic numbers has been of particular interest.  It turns out that they are `poorly approximable' in the sense that the exponent 2 in Dirichlet's Theorem is best possible.

\begin{theorem}[{Roth \cite{Roth}}]
	Let $\alpha$ be an algebraic number of degree $d \ge 2$.  Then for every $\epsilon > 0$, there are only finitely many $p/q \in \QQ$ such that
	\[
		\left| \frac{p}{q} - \alpha \right| < \frac{1}{q^{2+\epsilon}}.
	\]
\end{theorem}

In fact, most numbers (almost all, in a measure-theoretic sense) are poorly approximable in the same way (Sprind\u zuk \cite{Sprin}).  So is it possible to find some that are well-approximable?  Liouville constructs such numbers, which can be done cleverly (Exercise~\ref{exercise:liouville}).

	\begin{exercise} \label{exercise:liouville}  Show that $\alpha := \sum_{k=0}^\infty \frac{1}{10^{k!}}$ is very well-approximable in the following sense.  There are infinitely many $p_n/q_n \in \QQ$ such that $\left| p_n/q_n - \alpha \right| < 1/q_n^n$.
	\end{exercise}

	\begin{exercise} This exercise outlines a proof due to Liouville that any algebraic number $\alpha$ of degree $d \ge 2$ has only finitely many approximations $p/q\in \QQ$ satisfying $\left| \frac{p}{q} - \alpha \right| < \frac{1}{q^{d+\epsilon}}$ (in other words, Roth's theorem with exponent $d+\epsilon$ instead of $2 + \epsilon$).  This method of proof contains the seeds of the proof of Roth's theorem.
		\begin{enumerate}
			\item Consider $g(x)$, the minimal polynomial of $\alpha$ but scaled to lie in $\ZZ[x]$ with coefficients having no primitive factor.  Give a lower bound on $|g(p/q)|$.
			\item Use the mean value theorem on the difference $g(\alpha) - g(p/q)$, to give a lower bound on $|\alpha - p/q|$.
			\item Argue that this lower bound is independent of $p$ and $q$.
			\item Comparing the two estimates above, complete the argument.
		\end{enumerate}
	\end{exercise}

\section{Hyperbolic and Minkowski geometry}
\label{sec:hyper}

Having seen the importance of the upper half plane and the M\"obius action in number theory, we will now turn to studying this geometry in earnest, including a higher-dimensional generalization.  
We assume general knowledge of hyperbolic space, but will need the details of several models and the isometries between them.  Our approach to defining a model of hyperbolic space is to emphasize the underlying space and its isometries, in the spirit of of Klein's \emph{Erlanger Programm} \cite{KleinErlanger}.
More specifically, we take an approach based on linear algebra, also in the spirit of Klein; for a more detailed treatment, see for example \cite{ParkerNotes}.  Here we depart from the standard treatment to give the isometries between several models of hyperbolic space in terms of roots and coefficients of polynomials; for this specific treatment, see \cite{HST}.  

\subsection{Minkowski space}

	\begin{figure}
	  \includegraphics[width=3.2in]{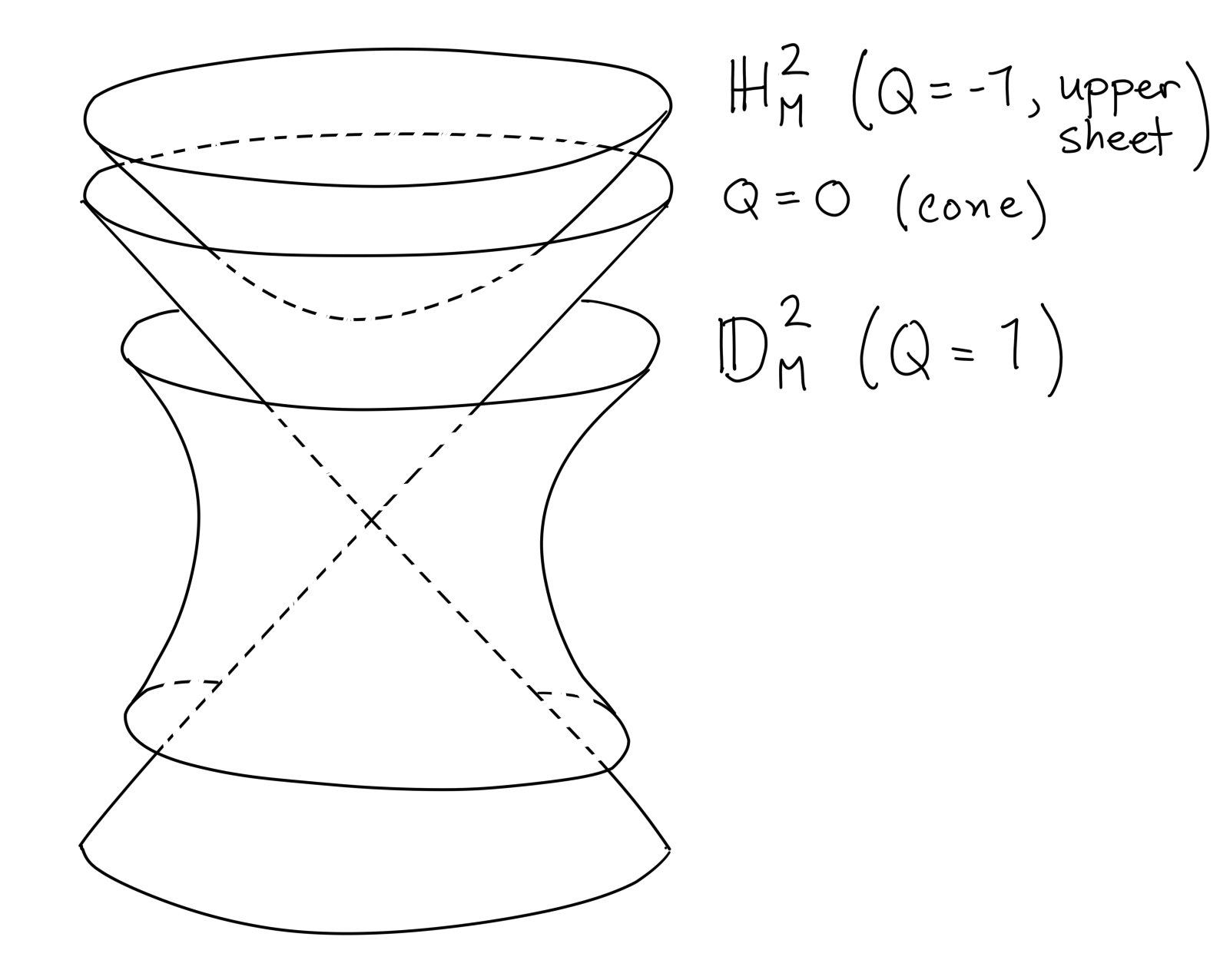}
	  \caption{The Minkowski space $M^3$.  To visualize $M^4$, imagine adding one spatial dimension.}
	  \label{fig:M3}
  \end{figure}

Consider an $n+1$ dimensional real vector space $M^{n,1} := \RR^{n+1}$.  
	Put a quadratic form $Q$ of signature $(n,1)$ and its associated bilinear form $\langle \cdot, \cdot \rangle_Q$ on this space:  we will now call it a \emph{Minkowski space}.  

	\begin{exercise}
		Suppose we are working over a field of characteristic not equal to $2$.  Show that if $Q(\mathbf{x})$ is a quadratic form on a vector space, then $\langle \mathbf{x}, \mathbf{y} \rangle := \frac{1}{2} \left( Q(\mathbf{x}+\mathbf{y}) - Q(\mathbf{x}) - Q(\mathbf{y}) \right)$ is a bilinear form.  Conversely, and inversely, show how to recover a quadratic form from a bilinear form.
	\end{exercise}

	The forms $Q$ and $\langle \cdot, \cdot \rangle_Q$ endow $M^{n,1}$ with geometry.  The zero locus $Q=0$ is a double cone emanating from the origin, called the \emph{light cone}.  Outside the cone, the level sets $Q = c$, $c > 0$ are one-sheeted hyperboloids.  Inside the cone, the level sets $Q = c$, $c < 0$ are two-sheeted hyperboloids.	

	Projectivizing this space to obtain $\PP M^{n,1}$, we can take the subset inside the cone:
	\[
		\HH^n_M := \{ [x_0 : x_1 : \cdots : x_n ] : Q(x_0,\ldots,x_n) < 0 \} \subset \PP M^{n,1};
	\]
	one may think of this as obtained by gluing the two sheets of the two-sheeted hyperboloid.  Then $\HH^n_M$ is a model of hyperbolic $n$-space, called the \emph{hyperboloid model}.  
	See Figure~\ref{fig:M3}.

The metric is given by the differential
\[
	ds^2 = Q(dx_0, dx_1, \ldots, dx_n), 
\]
or the distance function $d_M$ satisfying
\begin{equation}
	\label{eqn:minkowski-cosh}
	\cosh \left(  d_M(\mathbf{u},\mathbf{v})  \right) = \frac{ \langle \mathbf{u}, \mathbf{v} \rangle_Q }{\sqrt{ Q(\mathbf{u}) Q(\mathbf{v}) }}.
\end{equation}

	The quadratic space $M^{n,1}$ is acted upon by $SO_Q(\RR) \cong SO_{n,1}(\RR)$, the special orthogonal transformations preserving the form $Q$.  Their action takes $\HH^n_M$ to itself, acting as hyperbolic isometries. 

	Geodesics in $\HH^n_M$ are obtained as the intersections with lines of $\PP M^{n,1}$ (one may think of this as intersecting planes in $\RR^{n+1}$ with the two-sheeted hyperboloid).
	The light cone itself can be thought of as the `boundary at infinity' of $\HH^n_M$, and the intersection of the aforementioned projective line (or affine plane) with the cone gives the limit points of the geodesic.  In higher dimensions, we obtain geodesic surfaces, spaces, etc. by intersecting higher dimensional subspaces.

%	For the hyperbolic plane (the $n=2$ case), the Klein and Poincar\'e disk models are easy to visualize in this context \KS{add a picture!}.  

\subsection{The upper half plane}

We've seen that the quadratics have an affinity for the M\"obius action on the upper half plane.  The upper half plane $\HH^2_U$ is a model of the hyperbolic plane, whose isometries are $\PSL_2(\ZZ)$.  
The metric is given by the differential
\[
	ds^2 = \frac{ dz\; d\overline z }{\Im(z)^2}
\]
or the distance function $d_U$ satisfying
\begin{equation}
	\label{eqn:h2-dist}
	\cosh \left( d_U(z,w) \right) = 1 + \frac{ |z - w|^2 }{ 2 \Im(z) \Im(w) }.
\end{equation}
The hyperbolic isometries are the M\"obius transformations which we met earlier, $\PSL_2(\RR)$.  
The geodesics in the upper half plane consist of the restrictions to $\HH^2_U$ of lines $\Re(z) = a$ in $\CC$ and circles in $\CC$ centred on $\RR$ (Figure~\ref{fig:upper}).
The boundary of $\HH^2_U$ is
\[
	\partial\HH^2_U := \widehat{\RR} := \RR \cup \{ \infty \}
\]
where the point $\infty$ is an ideal (in other words, mathematically imagined) point `at infinity.'  The point $\infty$ is a limit point for any geodesic arising from a vertical line $\Re(z)=a$, the other limit point being $a$.  Half-circle geodesics have as their two limit points the intersection of the circle with $\RR$.  Observe that the isometries $\PSL_2(\RR)$, interpreted on the extended plane $\widehat{\CC} := \CC \cup \{ \infty \}$, take $\partial \HH^2_U$ to itself.

\subsection{Relating the upper half plane and hyperboloid models}
\label{sec:hyperupper}
	
There is a beautiful dictionary between the hyperboloid model and the upper half plane model of the hyperbolic plane, carried out by a hyperbolic isometry between the two, and it contains some surprises.  The essential observation is that $\Delta(A,B,C) = B^2 - 4AC$ is a signature $2,1$ form, but is also the discriminant form for quadratic polynomials $Ax^2 + Bx + C$, so we can think of the Minkowski space $\RR^{2,1}$ as parametrizing such polynomials by taking $Q = \Delta$.  Then the associated inner product is
\[
	\langle (A_1,B_1,C_1) , (A_2,B_2,C_2) \rangle_Q = B_1B_2 - 2A_1C_2 - 2A_2C_1.
\]
If we are interested in the roots of such polynomials, then it is natural to take a projectivization of this space, as the roots are unaffected by scaling the polynomial by a constant.

\begin{exercise}
	Demonstrate that $\Delta(A,B,C)$ is a signature $2,1$ form.
\end{exercise}

Viewed as a space of polynomials, the light cone cuts out those polynomials with complex conjugate roots as the interior of the light cone ($\Delta < 0$), leaving those with distinct real roots on the exterior ($\Delta > 0$).  The light cone itself corresponds to the quadratic polynomial having a double root ($\Delta = 0$).

%We now return to the image of the quadratics in the upper half plane.  We study the space of coefficients of quadratic polynomials with complex conjugate roots:
%	As the notation suggests, this is a model of the hyperbolic plane, as in the last section.  The quadratic form $b^2 - 4ac$ is of signature $(2,1)$.  %The coefficient space for quadratics with complex conjugate roots encompasses the vectors in the light cone, taken projectively.  If we project them onto the positive sheet of the two-sheeted hyperboloid $b^2 - 4ac = -1$, then we obtain the Poincar\'e model, as in the last section.

	To be more formal, consider now the projectivized interior of the light cone:
		\[
			\mathbb{H}^{2}_M := \{ [A:B:C] : A,B,C \in \mathbb{R}, B^2 < 4AC \} = \{ Ax^2 + Bx + C : A,B,C \in \mathbb{R}, B^2 < 4AC \}.
	\]
	Then, the hyperbolic isometry between $\mathbb{H}^{2}_M$ and $\HH^2_U$ is essentially the quadratic formula!

	\begin{theorem}[{\cite[Theorem 4.9]{HST}}]
		\label{thm:quadratic-formula}
		The map from $\mathbb{H}^{2}_M$ to $\HH^2_U$ given on $[A : B : C]$ by taking the root of positive imaginary part to the polynomial $Ax^2 + Bx + C$, i.e., the quadratic formula, is a hyperbolic isometry.  The inverse to this map takes $\alpha \in \CC \setminus \RR$ to $[1 : - \alpha - \overline{\alpha} : \alpha \overline{\alpha} ]$.
	\end{theorem}

	We can now line up the hyperbolic isometries on either side of the identification between $\HH^2_U$ and $\HH^2_M$.

	\begin{theorem}[{\cite[Observation 4.6]{HST}}] \label{thm:HST2}
		Identify elements $[ A : B : C]$ of $\mathbb{H}^{2}_M$ with matrices $D_{A,B,C} := \begin{pmatrix} C & B/2 \\ B/2 & A \end{pmatrix}$.
			Then the isometry of Theorem~\ref{thm:quadratic-formula} between $\HH^2_U$ and $\mathbb{H}^{2}_M$ is $\PSL_2(\RR)$-equivariant, relating the action of $\PSL_2(\RR)$ via M\"obius transforms on $\HH^2_U$ to the action $M \cdot D_{A,B,C} := M^{-1} D_{A,B,C} M^{-t}$ on $\mathbb{H}^{2}_M$.
\end{theorem}

Writing this out explicitly gives a representation of $\PSL_2(\RR)$ in $O_Q(\RR)$, i.e. the effect of $M$ on $[A : B : C]$ is multiplication by the following $3 \times 3$ matrix:
\[
	\PSL_2(\RR) \rightarrow O_Q(\RR), \quad
	\begin{pmatrix} a & b \\ c & d \end{pmatrix} \mapsto
	\begin{pmatrix} a^2 & -ac & c^2 \\ -2ab & bc+ad & -2cd \\ b^2 & - bd & d^2 \end{pmatrix}.
\]

One advantage of the ambient Minkowski space containing $\HH^2_M$ is that there's a satisfying analogous relationship between the space \emph{outside} the light cone and the \emph{geodesics} of $\HH^2_U$.  For this, we might give a name to the projectivized exterior of the light cone:
		\[
			\mathbb{D}^{2}_M := \{ [A:B:C] : A,B,C \in \mathbb{R}, B^2 > 4AC \} = \{ Ax^2 + Bx + C : A,B,C \in \mathbb{R}, B^2 > 4AC \}.
	\]
	The $\DD$ is an \emph{anti de Sitter space}, which is a term from physics.   See Figure~\ref{fig:M3}.

	\begin{theorem}[{\cite[Observation 4.11]{HST}}]\label{thm:HST3}
		The following two maps from $\mathbb{D}^{2}_M$ to the space of geodesics of $\HH^2_U$ coincide:
		\begin{enumerate}
			\item given on $[A : B : C]$ by returning the geodesic whose endpoints are exactly the two real roots of $Ax^2 + Bx + C$;
			\item given on $[A : B : C ]$ by first taking the plane normal to $[A : B : C]$ in $M^{2,1}$ (with respect to the Minkowski norm), then intersecting it with $\HH^2_M$, and finally composing with the hyperbolic isometry of Theorem~\ref{thm:quadratic-formula}. 
		\end{enumerate}
	\end{theorem}

	\begin{exercise}
		Prove Theorems~\ref{thm:quadratic-formula}, \ref{thm:HST2}, \ref{thm:HST3}.
	\end{exercise}

	\begin{figure}
	  \includegraphics[width=2.5in]{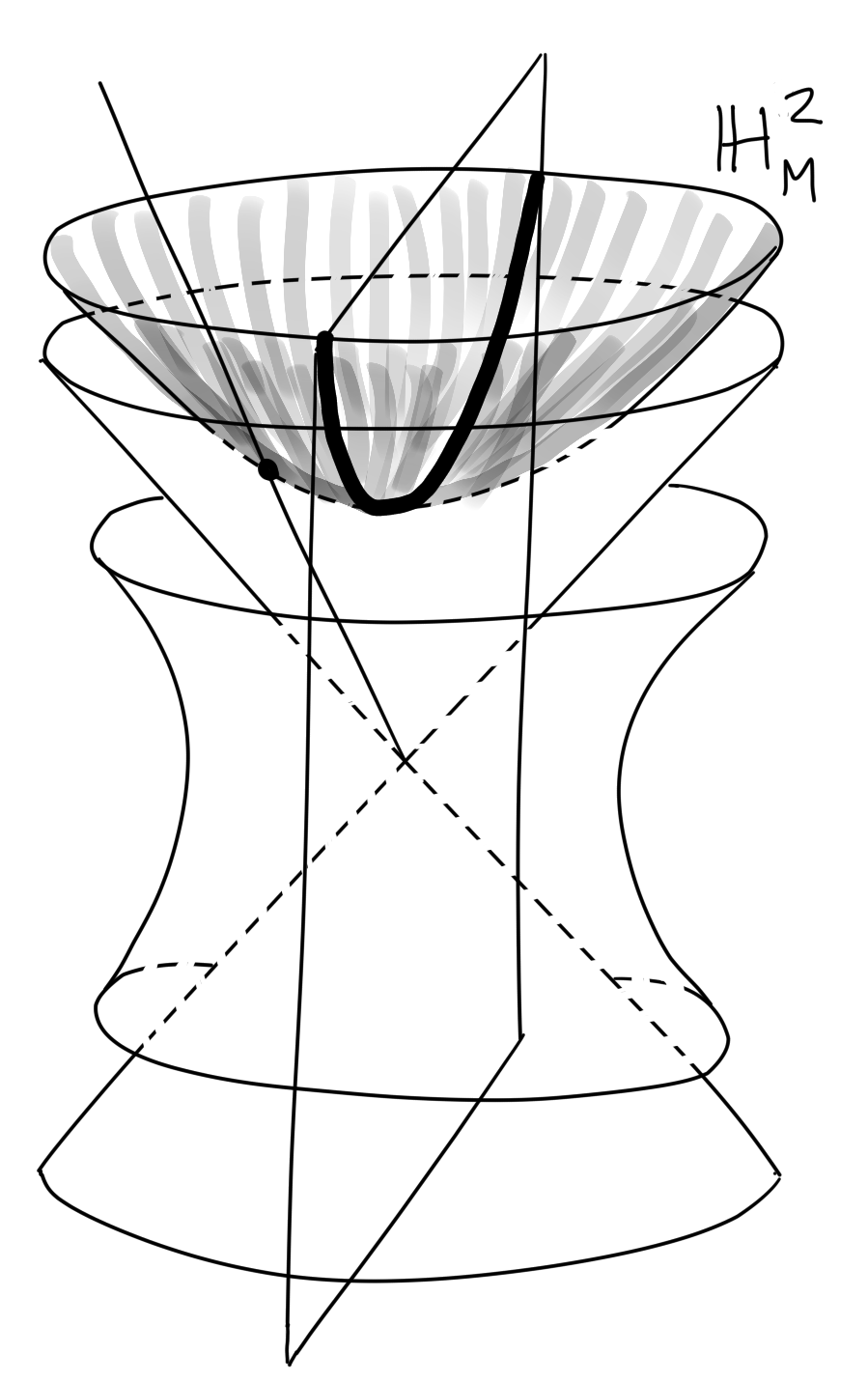}
	  \caption{$M^3$ with example plane cutting $\HH^2_M$ (giving a hyperbolic geodesic, in thick pen) and ray piercing $\HH^2_M$ (giving a hyperbolic point, in thick pen).}
	  \label{fig:cone-cut}
  \end{figure}

  We summarize the situation in a table:
\renewcommand{\arraystretch}{1.5}
\begin{center}
	\begin{tabular}{ m{0.4\textwidth} || m{0.4\textwidth} }
		upper half plane $\HH^2_U$ & Minkowski space $M^{2,1}$  \\
	\hline
	\hline
	points & vectors inside the light cone \\
	\hline
	root of $Ax^2 + Bx + C$, $Q<0$ & vector $[A : B : C]$, $Q<0$,
	or matrix $D_{A,B,C} := \begin{pmatrix} C & B/2 \\ B/2 & A \end{pmatrix}$  \\
	\hline
 geodesics & planes cutting the light cone \\
	\hline
 geodesic joining roots of $Ax^2 + Bx + C$, $Q > 0$ & plane normal to vector $[A : B : C]$, $Q > 0$ \\
	\hline
		M\"obius action of $M \in \PSL_2(\RR)$ & action $D_{A,B,C} \mapsto M^{-1} D_{A,B,C} M^{-t}$, $M \in \PSL_2(\RR)$ \\
\end{tabular}
\end{center}
\renewcommand{\arraystretch}{1}

	It makes sense to think of the one-sheeted hyperboloid as the \emph{space of geodesics} of the upper half plane, since any point in that space is a normal vector normal to a plane cutting out a geodesic.

%	Viewing the hyperboloid from above, we `see' the Klein disc model of the hyperbolic plane.  
%	More precisely, the Klein disc model is obtained by projecting onto a disc perpendicular to the axis of rotational symmetry of the light cone.
%	There is of course a hyperbolic isometry from the Klein model to the upper half plane model.
%
	%The space $\CC\PP^{n+1}$ is acted upon by the symmetrices $\PGL_{n}(\CC)$.
	\subsection{The Hamilton quaternions and upper half space}

	\begin{figure}
	  \includegraphics[width=5.5in]{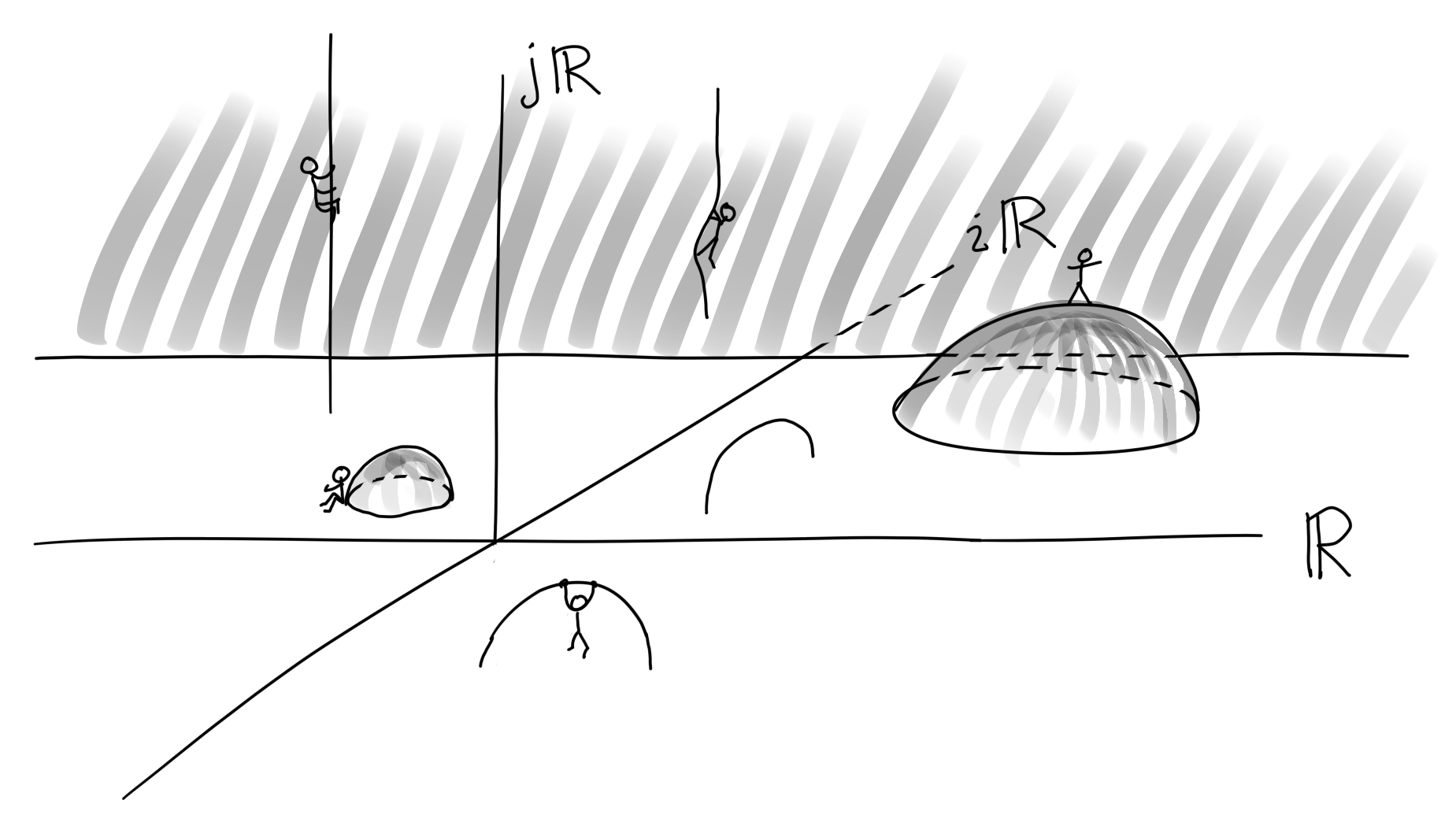}
	  \caption{The upper half space, with example geodesic planes and lines.}
	  \label{fig:upper3}
  \end{figure}

	The ring of Hamilton quaternions is the ring 
	\[
		H = \{ x + yi + zj + wk: x,y,z,w \in \RR \}
	\]
	with relations $i^2 = j^2 = k^2 = -1$ and $k = ij = -ji$.  There is a quaternionic conjugation:
	\[
		\overline{x + yi + zj + wk} = z - yi - zj - wk.
	\]

	Analogously to the upper half plane, we can define the upper half space
	\[
		\HH^3_U := \{ x + yi + zj : z > 0 \} \subseteq H.
	\]
	This is a standard model of hyperbolic $3$-space, thought of as a halfspace of $\RR^3 = \{ x + yi + zj \}$, whose boundary is $\partial \HH^3_U := \widehat{\CC} := \CC \cup \{ \infty \}$, consisting of a copy of $\CC$ (thought of as $z=0$ in $\RR^3$ or $z=w=0$ in H), augmented by a point at $\infty$.  We can also write $\alpha := x+yi \in \CC$ and an element of $\HH^3_U$ as $\alpha + zj$, as in Figure~\ref{fig:upper3}.

%	We can consider it a subset of the Hamilton quaternions, which strengthens the analogy with the upper half plane.  %In particular, the addition in $\HH^3_U$ is that of the quaternions. 
	The differential is
	\[
		ds^2 = \frac{d\alpha d\overline{\alpha} + dz^2 }{z^2},
	\]
	and the distance function $d_U$ satisfies
	\begin{equation}
		\label{eqn:h3dist}
	%	\cosh^2\left( \frac{ \rho( \alpha + zj, \beta + wj ) }{2} \right)
	%		=
	%		\frac{
	%		|\alpha - \beta|^2 + (z +w)^2 }{ 4zw}.
%
		\cosh \left( d_U(\alpha + zj, \beta + wj) \right) = 1 + \frac{ |\alpha - \beta|^2 + (z-w)^2 }{ 2zw}.
		\end{equation}

	There is an action of $\PSL_2(\CC)$ via M\"obius transformations on the boundary of the upper half space, $\widehat{\CC}$.  This action extends to a unique action on $\HH^3_U$ by hyperbolic isometries, namely
	\[
		\begin{pmatrix}
			a & b \\ c & d 
		\end{pmatrix} \cdot
		(\alpha + zj) =
		(a(\alpha + zj) + b) ( c (\alpha + zj) + d)^{-1}, \quad
			\begin{pmatrix}
			a & b \\ c & d 
		\end{pmatrix} \in \PSL_2(\CC), \quad \alpha \in \CC, \quad z \in \RR,
	\]
	where we must keep in mind the non-commutativity of $i$ and $j$, so that order matters; and the notion of inverse takes place in the quaternions, so that
	\begin{equation}
		\label{eqn:quat-inverse}
		(c (\alpha + zj) + d)^{-1} = \frac{(\overline\alpha - zj) \overline{c} + \overline{d}}
		{|c\alpha + d|^2 + |c z|^2}.
	\end{equation}
	Again, order matters. 
	Note that for $z=0$ the given action restricts to the usual M\"obius action on $\partial \HH^3_U$.

	\begin{exercise}
		\begin{enumerate}
			\item Prove that any element of $\PSL_2(\CC)$ can be expressed as a composition of translation ($z \mapsto z + \beta$), scaling ($z \mapsto \beta z$) and circle inversion ($z \mapsto 1/z$).
			\item Prove equation \eqref{eqn:quat-inverse}.
			\item Prove that the given M\"obius action takes $\HH^3_U$ to itself.
			\item Prove that the given M\"obius action on $\HH^3_U$ acts by isometry. 
		\end{enumerate}
	\end{exercise}

	We define the upper half space model of $\HH^3$ to be $\HH^3_U$ with isometries given by the M\"obius action of $\PSL_2(\CC)$ as above.

	The geodesics are given by the restriction to $\HH^3_U$ of any circle contained in a vertical plane with centre on $\partial \HH^3_U$, together with vertical lines $\alpha = \alpha_0$ (Figure~\ref{fig:upper3}).

%	\begin{exercise}
%		Verify that the M\"obius actions take geodesics to geodesics.
%	\end{exercise}

%	\begin{exercise}
%		Recall that the actions of $\PSL_2(\RR)$ extend to hyperbolic isometries on the upper half plane by simply applying the M\"obius transformation to complex points as well as real ones.
%		In analogy to the situation above, write the action of $\PSL_3(\CC)$ on the upper half space in the form of `quaternionic M\"obius transformation.'  You may also want to verify that this results in hyperbolic isometry.
%	\end{exercise}
%

	\subsection{Relating the hyperboloid and upper half-space models for hyperbolic $3$-space}

	The story in Section~\ref{sec:hyperupper} has an analog for hyperbolic $3$-space in place of the hyperbolic plane.  Recall that we viewed $M^{2,1}$ as a space of symmetric matrices $D_{A,B,C} := \begin{pmatrix} C & B/2 \\ B/2 & A \end{pmatrix}$ associated to quadratic forms $Ax^2 + Bxy + Cy^2$ and quadratic polynomials $Ax^2 + Bx + C$.  

		Moving up from $M^{2,1}$ to $M^{3,1}$, we can now consider $M^{3,1}$ to be a space of Hermitian matrices
		\[
		\begin{pmatrix}
			q & -r+si \\ -r-si & p
		\end{pmatrix}, \quad p,q,r,s \in \RR,
	\]
	with determinant $Q(p,q,r,s) = r^2 + s^2 - pq$ playing the role of the signature $3,1$ form.
	The term \emph{Hermitian matrix} means that $M = M^\dagger$ where $\dagger$ represents the complex conjugate transpose.  %Thus we consider the Minkowski space $\RR^{3,1}$ under the negative Hermitian determinant $Q(p,q,r,s) = r^2 + s^2 - pq$ of signature $3,1$.  
	The associated bilinear form is
\[
	\langle (p_1,q_1,r_1,s_1) , (p_2,q_2,r_2,s_2) \rangle_Q = r_1r_2 + s_1s_2 - \frac{p_1q_2}{2} - \frac{p_2q_1}{2}. 
\]
	The form $Q$ breaks the space into the interior and exterior of the light cone $Q=0$, as before.

	If we projectivize, then we obtain Hermitian matrices up to scaling.  Let us define, as before,
	\begin{align*}
		\mathbb{D}^{3}_{M} &:= \{ [p:q:r:s] : p,q,r,s \in \mathbb{R}, r^2 + s^2 > pq \}, \\
		\mathbb{H}^{3}_{M} &:= \{ [p:q:r:s] : p,q,r,s \in \mathbb{R}, r^2 + s^2 < pq \}. 
	\end{align*}
	Revisit Figure~\ref{fig:M3} and imagine adding one spatial dimension.

	Then we have an isometry between our two models of hyperbolic $3$-space, reminiscent of the quadratic formula.

	\begin{theorem}
		\label{thm:3isometry}
		The following map is an isometry from $\mathbb{H}^3_M$ to $\mathbb{H}^3_U$:
		\[
			[ p : q : r : s ] \mapsto \frac{r}{p} + \frac{s}{p} i + \frac{\sqrt{ r^2+s^2 - pq }}{p}j.
		\]
		The inverse is
		\[
			a + bi + cj \mapsto [ 1 : a^2 + b^2 - c^2  : a : b ].
		\]
	\end{theorem}

	\begin{proof}
		That they are inverses is easy.  We will show distances are preserved.  By the transitivity of the hyperbolic isometries of hyperbolic $3$-space on the tangent bundle, we need only compare points $j$ and $aj$ in $\HH^3_U$, which correspond to points $[1,-1,0,0]$ and $[1,-a^2,0,0]$ in $\HH^3_M$.  Specifically, we can move the first point to $j$ by an isometry (by transitivity on points), and then move the tangent direction to the second point to point upward along the $j$ line (by transitivity on the tangent bundle).  For more details, consult the analogous \cite[Theorem 4.9]{HST}.  Then we need only compute the distance before and after the isometry.
		Using \eqref{eqn:minkowski-cosh}, we find
		\[
			\cosh( d_M( [1:-1:0:0], [1:-a^2:0:0] )) = \frac{ 1+a^2}{ 2a}.
		\]
		Using \eqref{eqn:h3dist}, we find
		\[
		\cosh \left( d_U( j, aj ) \right)
			=
			\frac{
		 1 +a^2 }{ 2a}.
		\]
	\end{proof}

	This is somewhat analogous to the `quadratic formula' isometry between `coefficient space' $\HH^2_M$ and `root space' $\HH^2_U$ discussed previously, in that the map takes $[p : q : r : s]$ to a quaternionic root of the polynomial
	\begin{equation}
		\label{eqn:hermitianpoly}
		%p Z \overline{Z} + (r-si) \overline{Z} + Z (r+si) + q = p \left( Z + \frac{r - si}{p} \right) \overline{\left( Z + \frac{r - si}{p} \right)} - \frac{r^2 +s^2 - pq}{p}.
		p \left( Z - \frac{r + si}{p} \right)^2 - \frac{r^2 +s^2 - pq}{p}.
	\end{equation}
%	\KS{should this be $-(r+si)\overline{Z}-{Z}(r-si)$ if centre is $r+si$}
	The only quaternionic root to this polynomial that lies in $\HH^2_U$ is the root mentioned.
%	This is a `quaternionic sphere' centred at $\frac{-r+si}{p}$ with radius $\frac{\sqrt{r^2+s^2-pq}}{p}$. 
%	However, this polynomial has more than one solution.  Observe that in the quadratic form case \KS{bad phrasing}, the polynomial $a z \overline{z} + \frac{b}{2} (z + \overline{z}) + c = 0$ also has more than one solution:  this is the equation of a circle in the complex plane.  However, it has exactly two solutions additionally satisfying $az^2 + bz + c = 0$, namely $\frac{-b \pm \sqrt{b^2-4ac}}{2a}$; these are the uppermost and lowermost point of the solution circle (i.e., extremal imaginary part).
%	\begin{exercise}
%		Show that the quaternionic solutions to the equation \eqref{eqn:hermitianpoly} with zero $k$-coefficient form a sphere \KS{improve phrasing}, whose uppermost (largest coefficient of $j$) point is $\frac{r}{p} + \frac{s}{p} i + \frac{\sqrt{ r^2+s^2 - pq}}{p}j$, and whose center is on $\partial \HH^3_U$ (so the upper hemisphere is a geodesic plane in $\HH^3_U$).
%%and whose lowest point lies on $\partial \HH^3_U$, i.e. a horosphere.
%\end{exercise}
We also obtain an equivariance statement.

	\begin{theorem}
		\label{thm:equiv3}
		Identify elements $[ p : q : r : s ]$ of $\mathbb{H}^{3}_{M}$ with matrices $H_{p,q,r,s} := \begin{pmatrix} q & -r+si \\ -r-si & p \end{pmatrix}$.
			Then the map of Theorem~\ref{thm:3isometry} between upper half space $\HH^3_U$ and $\mathbb{H}^{3}_{M}$ is $\PSL_2(\CC)$-equivariant, relating the action of $\PSL_2(\CC)$ via M\"obius transformations on $\HH^3_U$ to the action $M \cdot H_{p,q,r,s} := M^{-1} H_{p,q,r,s} M^{-\dagger}$ on $\mathbb{H}^{3}_{M}$.
\end{theorem}

\begin{exercise}
	Prove Theorem~\ref{thm:equiv3} and give the corresponding representation of $\PSL_2(\CC)$ in $O_Q(\RR)$.
\end{exercise}

\subsection{The space of circles}

By a circle in $\widehat{\CC}$, we mean any circle in $\CC$ or the union of any straight line in $\CC$ with $\infty$.  These latter we think of as circles through $\infty$, and they have curvature $0$ (infinite radius).

	Associated to a Hermitian matrix such as $M = \begin{pmatrix} q & -r+si \\ -r-si & p \end{pmatrix}$ in the last section, we have a Hermitian form
	\[
		H_M(Z,W) := \begin{pmatrix} W & Z \end{pmatrix} \overline{M} \begin{pmatrix} \overline W \\ \overline{Z} \end{pmatrix} = pZ\overline Z + (-r+si)Z \overline W + (-r-si)\overline{Z} W + q W \overline{W},\quad p,q,r,s \in \RR.
	\]
	When $\Delta > 0$ (i.e. $M \in \mathbb{D}^3_M$), then the locus $H_M(Z,W)$ in $\PP^1(\CC)$ (or $H_M(z,1)$ in $\CC \cup \{ \infty \}$) gives a circle.

	\begin{theorem}
		\label{thm:spaceofcircles}
	%	Let $[p : q : r : s ] \in {M}^4$.  
			 Let $[p : q : r : s] \in \mathbb{D}^3_M$.  
		Let $H_M(Z,W)$ be the associated Hermitian form.  %Then:
%		\begin{enumerate}
			 Then the roots $Z \in \CC$ of $H_M(Z,1)$ form a circle in $\widehat\CC$.  The circle is the boundary of a unique geodesic plane in the upper half space.   Conversely, every circle arises uniquely in this way.  
%			\item Let $[p : q : r : s] \in \mathbb{H}^{3}_{M}$.  Write $\widehat{x+ yi + zj + wk} = x - yi + zj + wk$ for $x,y,z,w \in \RR$.  Then, the quaternionic solutions to the polynomial $h'(Z,1) =  pZ \widehat{Z} + (r+si)Z + \widehat{Z}(r-si) + q$, projected to the upper half space, consists of just one point.  Conversely, every quaternionic point in $\HH^3_U$ arises uniquely in this way.
%		\end{enumerate}
	\end{theorem}

	\begin{proof}
%		We can use the polynomial $h'(Z,1) =  pZ \widehat{Z} + (r+si)Z + \widehat{Z}(r-si) + q$
%		for both statements.
		The equation $H_M(Z,1) = 0$ can be expressed as
		\[
\left(Z - \frac{r + si}{p}\right)\left(\overline{Z} - \frac{r - si}{p}\right) = \frac{r^2 + s^2-pq}{p^2},
		\]
	%	is equivalent to the following system in $(x,y,z,w) \in \RR^3$:
	%	\[
	%		z(xp+r) = pyw, \quad pyz = -w(xp+r), \quad h(x+yi,1) - 2p z^2 - 2p w^2 = 0.
	%	\]
%
%		If $z=w=0$ then the values of $h'(x+yi +zj+wk,1) = h(x+yi,1)$ are real and 
		so represents a circle with centre $(r+si)/p$ and radius $\sqrt{r^2+s^2-pq}/p$ whenever $r^2 + s^2 > pq$.  This last inequality follows from $[p: q : r : s] \in \mathbb{D}^3_M$.
		%.
		%If at least one of $z$ or $w$ is non-zero, then $x^2 + y^2 = r^2/p^2$, which entails $z^2 = -(r/p)^2 + q/p$.  This has real solutions if and only if $pq > r^2$.
%
%		The condition that $[p : q : r : s] \in \mathbb{H}^3_M$ implies $r^2 + s^2 < pq$, so there are no $z=0$ solutions and two other solutions only one of which is in $\HH^3_U$, namely $(-p/r,0,\sqrt{qp-r^2}/p)$.
%		The condition that $[p : q : r : s] \in \mathbb{D}^3_M$ implies $r^2 + s^2 > pq$, so there is a circle of solutions in $\partial \HH^3_U$ and no solution in $\HH^3_U$.
%
		For the converse, observe that the circle determines $r/p, s/p, q/p$, hence $[p : q : r : s] \in \PP M^4$. 
	\end{proof}

	Thus, it makes sense to consider the projectivization of the one-sheeted hyperboloid $\mathbb{D}^{3}_{M}$ to be a parameter space for the circles in $\widehat{\CC}$, i.e. \emph{the space of circles}.  Normalizing so that $r^2 + s^2 - pq = 1$, the coordinates have the following somewhat standard names:
	\begin{enumerate}
		\item $p$ = \emph{curvature}, which is the inverse of radius;
		\item $r$ = real part of curvature times center (the curvature times centre is sometimes more simply called the \emph{curvature-centre});
		\item $s$ = imaginary part of curvature times center;
		\item $q$ = \emph{co-curvature}, which is the curvature of the inversion of the circle through the unit circle.
	\end{enumerate}
%	Then $(b,c,x,y) \in \RR^4$ on the surface $\DD_{\operatorname{circ}}: bc = x^2 + y^2 - 1$, which is a one-sheeted hyperboloid.  The quadratic form $bc - x^2 - y^2$ is of signature $(3,1)$, so that we are in a Minkowski space, and the one-sheeted hyperboloid surrounds the light cone (see Figure in written lecture notes for now).%~\ref{fig:mink-circles}).  

	\begin{exercise}
		Show that the circle associated to each point of $\mathbb{D}^3_M$ via Theorem~\ref{thm:spaceofcircles} is the same circle obtained by the following process.  For a vector $\mathbf{v} \in \mathbb{D}^3_M$, take the hyperplane $P$ perpendicular to $\mathbf{v}$ in the Minkowski geometry (i.e., with respect to $\langle \cdot , \cdot \rangle_Q$).  Take the intersection of $P$ with $\mathbb{H}^3_M$ and call the intersection $I$.  Then use the hyperbolic isometry of Theorem~\ref{thm:3isometry} to map $I$ into $\HH^3_U$.  This should provide a geodesic plane $G$.  Take the circle $\mathcal{C}$ at the boundary of $G$.  Revisit Figure~\ref{fig:cone-cut}, imagining an extra spatial dimension.
	\end{exercise}

	Recall that $\PSL_2(\CC)$ acts on $\widehat{\CC}$.  
It will be helpful later to understand the image of $\widehat{\RR}$ under a M\"obius transformation; this is a straightforward computation.

	\begin{proposition}[{\cite[Proposition 3.7]{StangeVis}}]
		\label{prop:mobR}
		Let $z \mapsto \frac{\alpha z + \beta}{ \gamma z + \delta}$ be a M\"obius transformation.  Then the image of $\widehat{\RR}$ under the transformation is a circle with curvature equal to $2\Im(\overline\gamma {\delta}) = i(\gamma\overline{\delta} - \overline{\gamma}\delta)$, and curvature times center equal to $i(\alpha \overline{\delta} - \overline{\gamma}\beta)$.  Finally, the co-curvature is $2\Im(\overline\alpha \beta) = i(\alpha \overline\beta - \overline\alpha \beta)$.
	\end{proposition}

	\begin{exercise}
		\label{ex:mobR}
		Prove Proposition~\ref{prop:mobR}.  Furthermore, assume the M\"obius transformation is in $\PSL_2(\ZZ[i])$, and show that for such a circle, the resulting vector $(p,q,r,s) \in \ZZ^4$ in the space of circles is congruent to $(0,0,0,1)$ modulo $2$.
	\end{exercise}

	\begin{exercise}
		What are the correct definitions of curvature and curvature-center for straight lines?
	\end{exercise}

%	\begin{exercise}
%		Given a M\"obius transformation in $\PSL_2(\ZZ[i])$, consider the circle $C$ that is the image of $\widehat{\RR}$.  Show that the image of $\widehat{\QQ}$ lies in $\QQ(i)$.
%	\end{exercise}

\section{Diophantine approximation in the complex plane}

It is natural to ask some of our fundamental Diophantine approximation questions for complex numbers.  This is where we can reap the benefits of the geometric perspective of the last section.  Since the rationals cannot approximate complex numbers off the real line, we must begin by asking what we are hoping to approximate by.  One natural answer is to ask to approximate by algebraic numbers of a fixed degree.  Another is to approximate by elements of a fixed number field.  Here, we will consider the first (later, we will consider the second).

The next question is how to measure the size of an approximation.  In the real/rational case, we used the denominator of $p/q \in \QQ$ as a natural measure of size.  For a general algebraic number $\alpha$, having minimal polynomial $a_d x^d + \cdots + a_0 \in \ZZ[x]$, where $\gcd(a_i) = 1$, the \emph{na\"ive height} is defined as
\[
	H(\alpha) := \max_i |a_i|.
\]
Then, in analogy to Dirichlet's and Roth's theorems, Koksma defines \cite{Koksma} 
\[
	k_d (\alpha) :=  \sup \{ k : \text{ there exist infinitely many algebraic $\beta$ of degree $\le d$ such that } | \alpha - \beta | < 1/H(\beta)^k \}.
\]
In this language, Dirichlet's Theorem states that $k_1(\alpha) \le 2$ for rational $\alpha$ and $\ge 2$ for $\alpha \in \RR \setminus \QQ$.  Roth's theorem says that $k_1(\alpha) = 2$ for algebraic $\alpha$, and Sprind\u zuk says $k_1(\alpha) = 2$ for almost all real $\alpha$.

\begin{theorem}[{Sprind\u zuk, \cite{Sprindzuk}}] 
	$k_d(\alpha) = d + 1$ for almost all $\alpha \in \RR$ and $k_d(\alpha) = \frac{d+1}{2}$ for almost all $\alpha \in \CC \setminus \RR$.
\end{theorem}

This theorem shows that approximation by algebraic numbers is fundamentally different on the real line than off of it.  Next, here is a tantalyzingly precise description of $k_d(\alpha)$ for $\alpha$ algebraic.

\begin{theorem}[Bugeaud-Evertse \cite{BugeaudEvertse}]
	\label{thm:bugeaud-evertse}
	Let $\alpha \in \CC \setminus \RR$ be algebraic.  Then
	\[
		k_d(\alpha) = \left\{
			\begin{array}{ll}
				\frac{d+1}{2} \text{ or } \frac{d+2}{2}  & \text{ if } \deg(\alpha) \ge d+2, d \text{ even} \\
				\min\left\{ \frac{deg(\alpha)}{2}, \frac{d+1}{2} \right\} & \text{ otherwise }
			\end{array}
			\right.
		\]
		For $d=2$, $\deg(\alpha) > 2$:
		\[
			k_d(\alpha) = \left\{
				\begin{array}{ll}
					2 & \text{ if }1,\alpha \overline{\alpha}, \alpha + \overline{\alpha} \text{ are $\QQ$-independent} \\
					3/2 & \text{otherwise} 
				\end{array}
				\right.
			\]
\end{theorem}

The underlying methods for this are algebraic.  This theorem, like Roth's Theorem, follows from a far-reaching theorem called Schmidt's subspace theorem.

\subsection{Quadratics}

We are interested in finding geometric connections or explanations for Theorem~\ref{thm:bugeaud-evertse}.  
In particular, the result shows that approximation by quadratics is quite different for certain types of complex numbers compared to others.  Our starting point is an analog to Figure~\ref{fig:ratsDots} showing the rational numbers.  In Figure~\ref{fig:quadratics}, we see the quadratic algebraic numbers in the upper half plane, sized by their discriminants.
	\begin{figure}
                \includegraphics[width=5.7in]{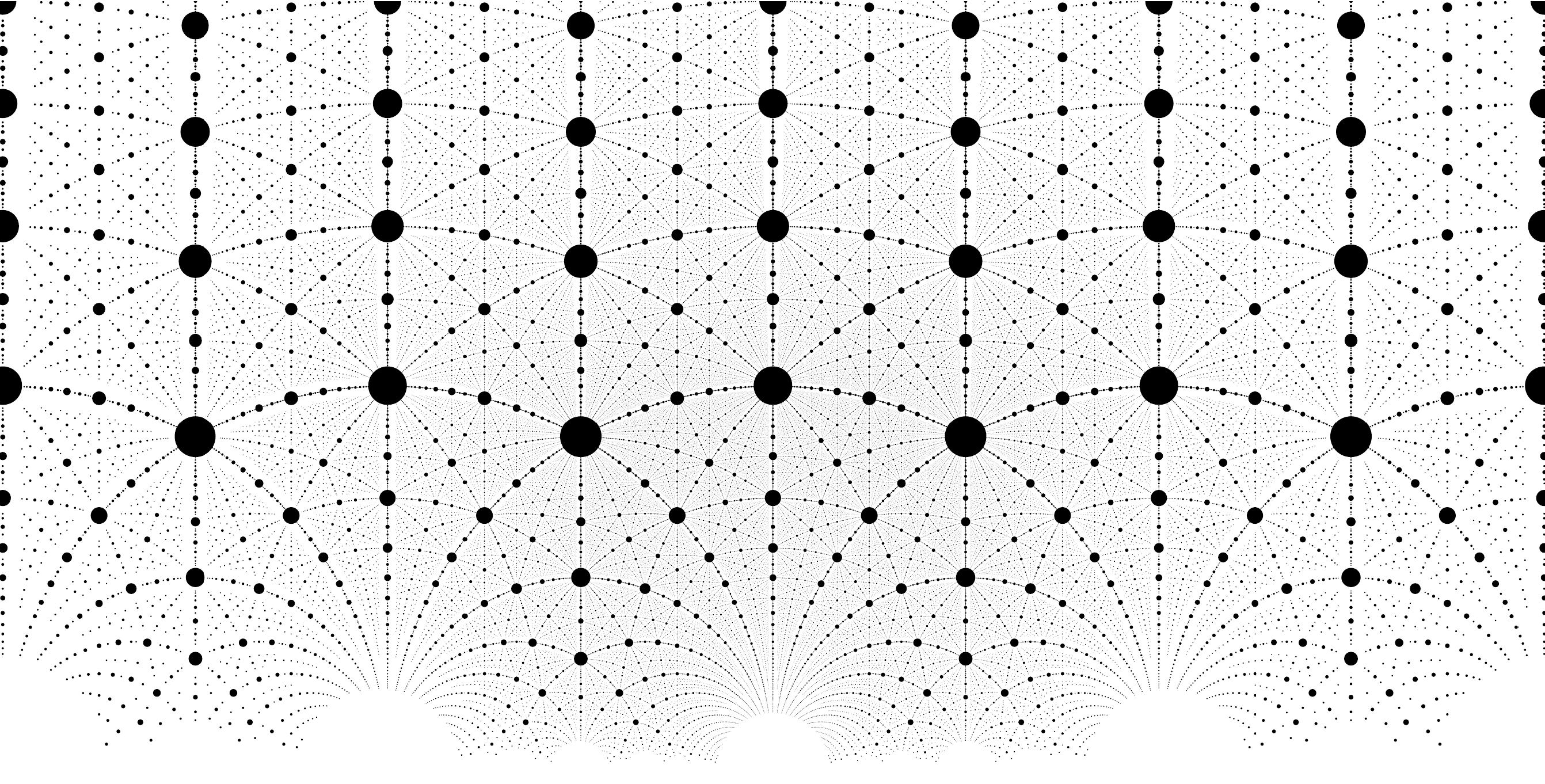}
		\caption{The quadratic algebraic numbers in the upper half plane, sized by discriminant.  Image:  Edmund Harriss, Steve Trettel, Katherine E. Stange.}
		\label{fig:quadratics}
	\end{figure}
	In this picture, our eyes infer the $\SL_2(\ZZ)-$tessellation from Figure~\ref{fig:sl2}, but with the geodesic boundaries -- and many more geodesics besides -- filled in with the pearly necklaces of Figure~\ref{fig:ratsDots}.  In light of the $\ZZ^2$ hiding in Figure~\ref{fig:ratsDots}, this image asks us a question:  are we viewing a higher dimensional lattice under some transformation?

	To answer this question, we return to the coefficient and roots spaces $\HH^2_M$ and $\HH^2_U$ of Section~\ref{sec:hyper}.
	Within the coefficient space $\HH^2_M$, we have the quadratic algebraics: $\HH^2_{M}(\ZZ) := \{ [a:b:c] : a,b,c \in \mathbb{Z}, b^2 < 4ac \}$.  This is the projectivization of the portion of the 3D lattice $\ZZ^3$ which lies inside the light cone.  The na\"ive height is, roughly speaking, a measure of how far the representative $[a:b:c]$ with $\gcd(a,b,c)=1$ is from the origin.  So we are in a situation not unlike our description of continued fractions in the real line as the projectivization of the lattice $\ZZ^2$.
	That is, the relationship between $\HH^2_M(\ZZ)$ and Figure~\ref{fig:quadratics} is described by the transformation of Theorem ~\ref{thm:quadratic-formula}, which describes Figure~\ref{fig:quadratics} as an image of a 3D lattice under a certain geometric transformation.
		See Figure~\ref{fig:lattice-to-alg}.

	\begin{figure}
	  \includegraphics[width=5.5in]{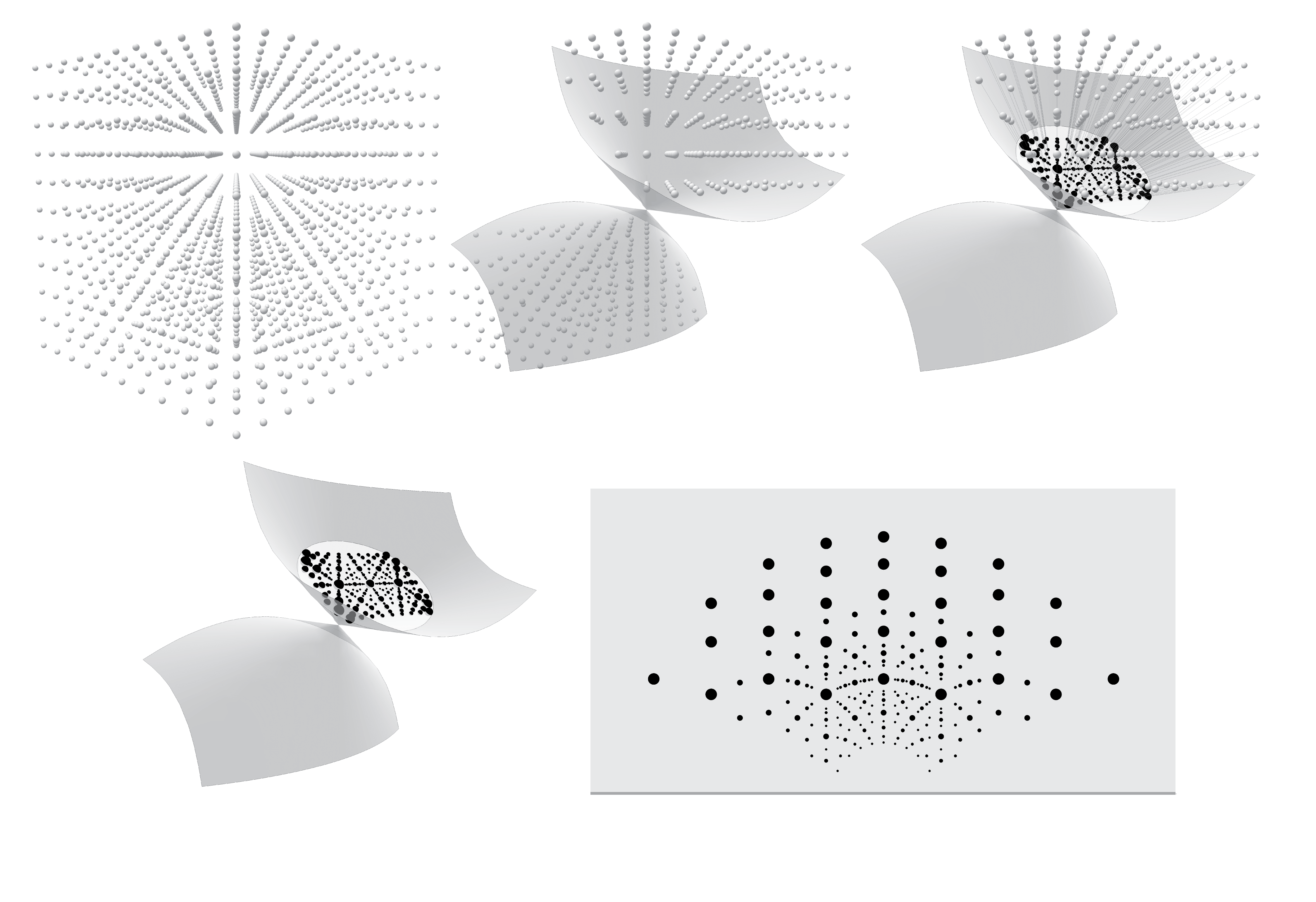}
		\caption{The transformation of polynomials with integer coefficients from `coefficient space' $\mathbb{H}^2_M$ to `root space' $\HH^2_U$.   From top left, left-to-right and top-to-bottom:  The lattice $\ZZ^3 \subseteq M^{2,1}$ representing polynomials with integer coefficients (i.e., having algebraic roots); the cone $B^2-4AC = 0$ and the portion of the lattice within it, representing those polynomials with complex conjugate algebraic roots; the projection onto a plane perpendicular to the axis of symmetry of the cone (the Klein disc model); the image of the lattice under the hyperbolic isometry of Theorem~\ref{thm:quadratic-formula} in the upper half plane. Image: \cite[Figure 24]{HST}. }
	  \label{fig:lattice-to-alg}
  \end{figure}

	Under this transformation, the planes inside the lattice $\ZZ^3$ projectivize to geodesics on the Poincar\'e disc model and hence in the upper half plane; these are the visually striking necklaces in Figure~\ref{fig:quadratics}.  We have the following consequence:

	\begin{proposition}[{\cite[Observation 4.12]{HST}}]
		\label{prop:rational-geodesic}
		Let $\alpha \in \CC \setminus \RR$.  Then $1, \alpha \overline{\alpha}, \alpha + \overline{\alpha}$ are $\QQ$-dependent if and only if $\alpha$ lies on a geodesic whose limit points form a conjugate pair of points in a real quadratic field, or a pair of rational points.
	\end{proposition}
	We call such geodesics \emph{rational geodesics}.  
The main reference for the remainder of this section is \cite{HST}.
\begin{exercise}
	\label{exercise:rational-geodesic} Prove Proposition~\ref{prop:rational-geodesic}.
\end{exercise}

	\subsection{The Diophantine approximation of quadratics}

	Having this geometric interpretation of the way the quadratic algebraic numbers fill out the complex upper half plane, we make some geometrically motivated choices for the notions of \emph{distance}, \emph{complexity} and \emph{goodness} in Diophantine approximation, that differ somewhat from the algebraically motivated ones made classically:
	\begin{enumerate}
		\item We size quadratics by their discriminant.  Unlike the na\"ive height and other classical measures of arithmetic complexity, the discriminant is invariant under $\PSL_2(\ZZ)$, the natural symmetries of the space.  In the fundamental region, it tracks with the na\"ive height, so this is not a violent change (see \cite[\S 5.2.4]{HST}). 
		\item We use the hyperbolic metric in the upper half plane, not the Euclidean one.  Again, this respects the action of $\PSL_2(\ZZ)$, without doing great violence in any bounded region.  Let $\alpha$ and $\beta$ correspond to points $f_\alpha$ and $f_\beta$ in $\HH^2_{M}$.  The distance in $\HH^2_{M}$ dictated by the Minkowski geometry is given by
			\[
				d_{M}(f_\alpha,f_\beta) = \acosh \left( \frac{ -\langle f_\alpha, f_\beta \rangle }{\sqrt{\Delta_\alpha \Delta_\beta}} \right).
			\]
	\end{enumerate}

	\begin{figure}
	  \includegraphics[width=4.5in]{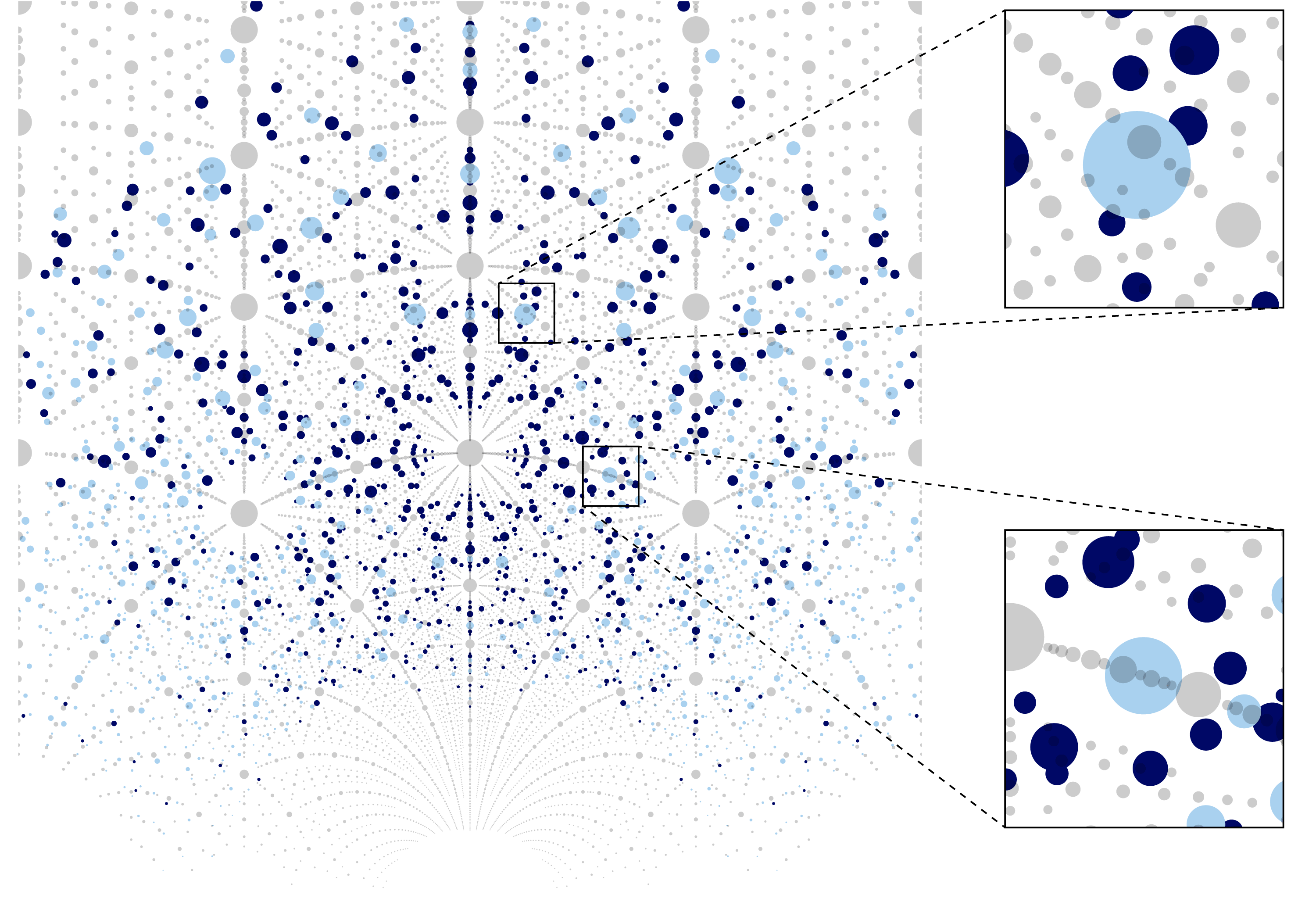}
		\caption{Quadratics are shown in grey.  Quartics are shown in two shades of blue.  Dots are sized by discriminant. Image: \cite[Figure 33]{HST}.}
	  \label{fig:quad-quar}
  \end{figure}

	As a consequence, we can re-interpret complex approximation results under these choices.

  \begin{theorem}[{\cite[Theorem 6.3]{HST}}]
		\label{thm:quad-dirichlet}
		Let $\alpha \in \CC \setminus \RR$ not be quadratic algebraic, but lying on a rational geodesic.  Then there exists $K_\alpha > 0$, depending on $\PSL_2(\ZZ)$ orbit of $\alpha$, so that there are infinitely many quadratic $\beta$ on that geodesic with
		\[
			d_{U}(\alpha,\beta) \le
				\acosh\left( 1 + \frac{K_\alpha}{|\Delta_\beta|^2} \right).
		\]
	\end{theorem}

	\begin{proof} We provide only a sketch, but the proof is an adaptation of the proof of Dirichlet's Theorem.  We begin by writing the element of coefficient space corresponding to $\alpha$ as
		\[
			[1/(\alpha + \overline{\alpha}) : 1 : \alpha \overline{\alpha}/(\alpha + \overline{\alpha})] = [\alpha_1 : 1 : \alpha_2].
		\]
		Then, we consider the multiples of this vector modulo $[\ZZ : \ZZ : \ZZ]$.  Dirichlet's box principle\footnote{More commonly known as the pigeonhole principle, although I've heard it claimed that this very proof is the first recorded use of the idea in the mathematical literature.} tells us two multiples are close to one another, so their difference, which we write as $n[\alpha_1 : 1 : \alpha_2]$ is close to a lattice element, which we denote $[p_1:n:p_2]$; this is a candidate good approximation $f_\beta$.
		This tells us that certain linear forms are small:
		\[
		|n\alpha_1 - p_1|, \quad 
		|n\alpha_2 - p_2|, \quad 
		|\alpha_1 p_2 - \alpha_2 p_1|.
	\]
	This gives a small discriminant pairing 
		$\langle f_\alpha , f_\beta \rangle$, which in turn gives a small hyperbolic distance.
\end{proof}

\begin{figure}
	  \includegraphics[width=5.5in]{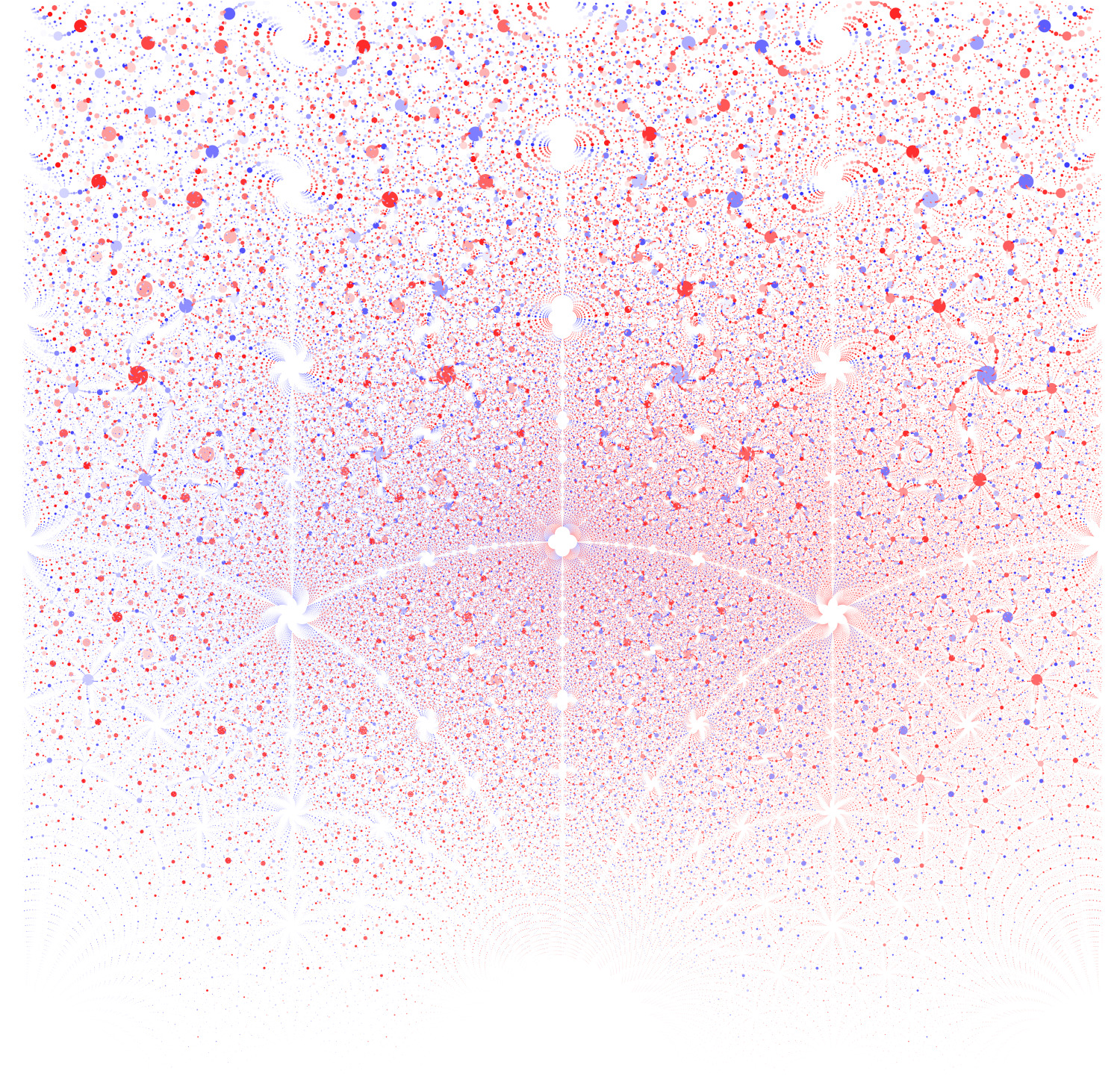} 
	  \caption{The cubics in the upper half plane \cite[Figure 1]{HST}.}
	  \label{fig:cubicsUHP}
  \end{figure}

Next, we re-interpret the result of Bugeaud and Evertse (see Figure~\ref{fig:quad-quar}).
\begin{theorem}[{\cite[Theorem 6.6]{HST}}]
		Let $\alpha \in \CC \setminus \RR$ be algebraic but not quadratic.  Let $\epsilon > 0$.  {If $\alpha$ lies on a rational geodesic}, then there are only finitely many quadratic $\beta$ such that
		\[
			d_{U}(\alpha,\beta) \le
			\acosh\left( 1 + \frac{1}{|\Delta_\beta|^{{2+\epsilon}}} \right).
		\]
		{Whether $\alpha$ is on a rational geodesic or not}, amongst $\beta$ not sharing a rational geodesic with $\alpha$, there are only finitely many such that
		\[
			d_{U}(\alpha,\beta) \le
			\acosh\left( 1 + \frac{1}{|\Delta_\beta|^{{3/2+\epsilon}}} \right).
		\]
	\end{theorem}

	The proof proceeds in a similar way, by moving to the same linear forms as in Theorem~\ref{thm:quad-dirichlet}, and then one applies Schmidt's subspace theorem.

\subsection{Cubics and beyond}

The cubics having two complex conjugate roots can be treated in a similar way, although the discriminant locus is more complicated.  The lattice of coefficients is 4-dimensional, so the image is much more complicated and layered when drawn as roots in the upper half plane.  It can help to embed it in a 3-dimensional space made up of the upper half plane root together with the real root.  It is natural to use a disc model for the upper half plane, and a circle for the real line, resulting in a torus as their product.  Some images are shown in Figures~\ref{fig:cubicsUHP}--\ref{fig:cubicsSLView}.\\

	\begin{minipage}{0.5\textwidth}
		\begin{center}
	  \includegraphics[width=2.5in]{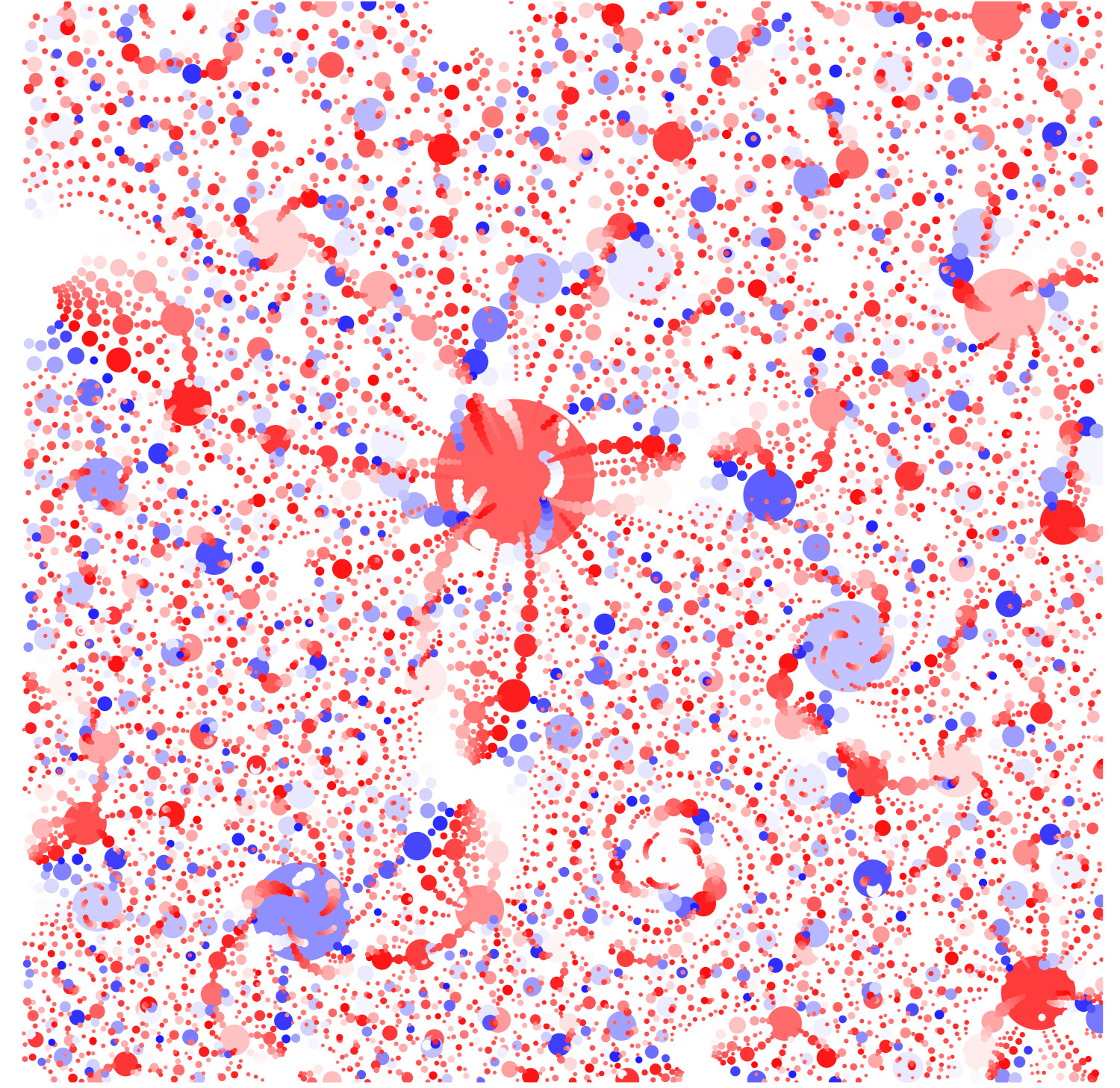}
			\captionof{figure}{A detail of the cubics of Figure~\ref{fig:cubicsUHP}.}
		\label{fig:cubicsDetail}
		\end{center}
	\end{minipage}\hfill
	\begin{minipage}{0.5\textwidth}
		\begin{center}
	  \includegraphics[width=2.5in]{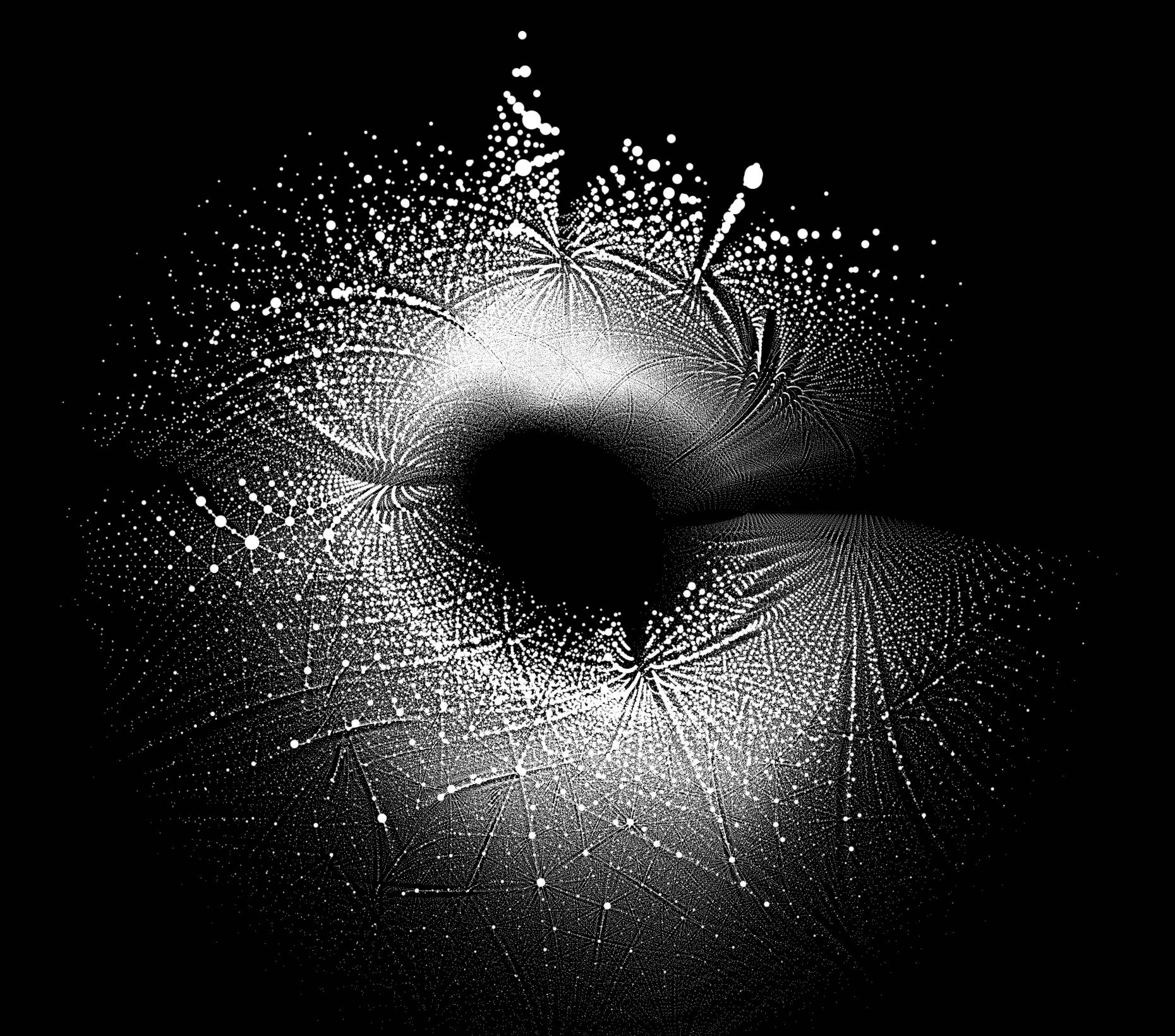}
		\captionof{figure}{A 3d view of the cubic roots embedded in a torus (this image was created using Emily Dumas' SL(View) software \url{https://www.dumas.io/}) \cite[Figure 10]{HST}.}
		\label{fig:cubicsSLView}
		\end{center}
		\end{minipage}

\subsection{Open problems}

  There are a great many open problems motivated by this perspective.
  \begin{enumerate}
	  \item Is there a geodesic flow / continued fraction theory for good approximations by quadratics?
	  \item Is there a natural analog to a Lagrange spectrum?
	  \item Is there an analog to the Farey subdivision?
	  \item And many more; see \cite{HST}.
  \end{enumerate}

\section{Apollonian circle packings: geometric aspects}

\subsection{Schmidt subdivision}

Asmus Schmidt developed a beautiful complex analogue to the Farey subdivision \cite{SchmidtComplex}.  Recall that the Farey subdivision on $\widehat{\RR} = \partial \HH^2_U$ is the boundary of a nested set of geodesics in $\HH^2_U$.   Similarly, Schmidt's subdivision views $\widehat{\CC}$ as the boundary of the upper half space, and the subdivision is formed as the boundary of geodesic planes in the upper half space.   The initial subdivision divides the plane by the use of two lines and two circles into eight regions (see Figure~\ref{fig:schmidt-farey}, top).  Each triangle is subdivided as in the lower portion of Figure~\ref{fig:schmidt-farey}, by inserting a new circle tangent to all three sides.  Repeating this, we obtain nested regions like in Figure~\ref{fig:schmidt-farey-gauss}.  (Schmidt's subdivision also incorporates dual circles orthogonal to these, but we will ignore them for now.)

	\begin{figure}
		\caption{A screenshot of Asmus Schmidt's paper, \emph{Diophantine Approximation of Complex Numbers} \cite[Figure 1, 1*]{SchmidtComplex}.  \textcolor{red}{Currently suppressed because permission has not yet been requested.}}
	  \label{fig:schmidt-farey}
  \end{figure}

  \begin{figure}
	  \includegraphics[width=4.5in]{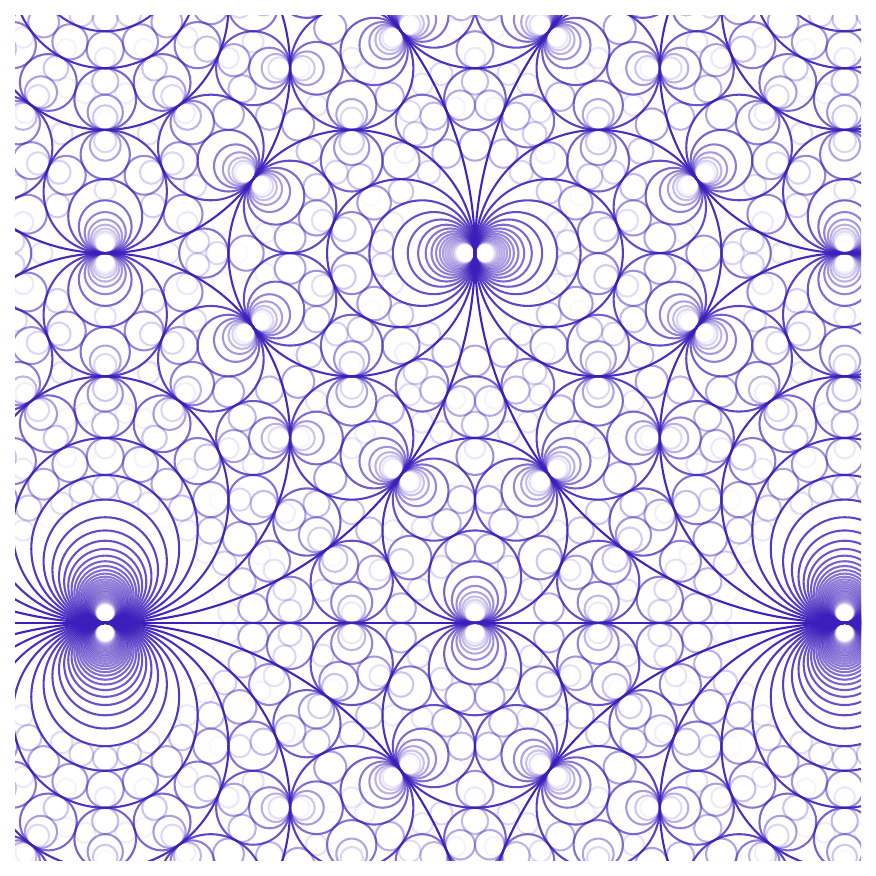}
	  \caption{Schmidt's subdivisions, iterated (without the dual circles).}
	  \label{fig:schmidt-farey-gauss}
  \end{figure}

  The analogy is as follows:
  \vspace{.2em}

\begin{center}
{\renewcommand{\arraystretch}{1.5}%
\begin{tabular}{ c | c | c }
	Continued fractions & Farey subdivision & Schmidt subdivision \\ 
	\hline
			    & $\widehat{\RR}$ & $\widehat{\CC}$ \\
			    & $\PSL_2(\ZZ)$ & $\PSL_2(\ZZ[i])$ \\
			    &  intervals & circles and triangles \\
	convergents & endpoints & tangency points \\
	coefficients & series of nested geodesics & series of nested geodesic planes \\
			  & --- & Apollonian circle packings \\
\end{tabular}}
\end{center}
  \vspace{.2em}

The last line of the table is a new and very rich phenomenon that only exists in the $\widehat{\CC}$ case.  Unlike the Farey subdivision, the Schmidt subdivisions are built out of intermediate `pieces,' called Apollonian circle packings.  It is this new object that is so fascinating.

\subsection{Descartes quadruples and Apollonian circle packings}

A \emph{Descartes quadruple} is a collection of four circles, all pairwise mutually tangent, of disjoint interiors.  Given three mutually tangent circles, there are exactly two circles, called \emph{Soddy circles}, which can complete the triple to a Descartes quadruple.  %\KS{add a note on history here}  
To construct an Apollonian circle packing, begin with three mutually tangent circles.  At each stage, add in any missing Soddy circles for any mutually tangent triple in the packing, and repeat ad infinitum (Figure \ref{fig:build}).  

Famously proven by Descartes and Princess Elisabeth of Bohemia\footnote{What we know is as follows.  Descartes proposed the problem (of finding the radii of the circles tangent to a given triple of circles) to Elisabeth in 1643, and she provided a solution, which is lost, but left Descartes, in his own words, `filled with joy.'  (Keep in mind he was writing to a Princess!)  Descartes' side of the correspondence having survived, we have two of his solutions, the second of which assumes the three starting circles are mutually tangent, and most closely matches what is now often called \emph{Descartes' Theorem} in this context.} in correspondence \cite[pp. 73--81; AT 4:37, 4:44, 4:45]{Elisabeth}, four circles form a Descartes quadruple if and only if their curvatures $a,b,c,d$ satisfy $Q(a,b,c,d)=0$ for the quadratic form
\begin{equation}
	\label{eqn:desc}
	Q(a,b,c,d) = (a + b + c + d)^2 - 2 (a^2 + b^2 + c^2 + d^2).
\end{equation}
In particular, given a mutually tangent triple of curvatures $a,b,c$, the Soddy circles have curvatures $d_1$ and $d_2$ satisfying $Q(a,b,c,d_i)=0$; these are related by
\[
        d_1 + d_2 = 2(a + b + c).
\]
There are various generalizations of this theorem, including to spheres in higher dimension \cite{DescartesMonthly, Rasskin}.  We will sometimes refer to the quadruple of curvatures as a Descartes quadruple also, at the risk of some minor confusion.

As we will exploit later, this has the beautiful consequence that if one begins with a Descartes quadruple of integer curvatures $a,b,c,d$, then the entire packing will consist of integer curvatures.  Such an \emph{integral} packing is called \emph{primitive} if it has no common factor among its curvatures.
Some examples are shown in Figure~\ref{fig:examples}.

\begin{figure}
        \raisebox{-.49\height}{
                \includegraphics[width=1.5in]{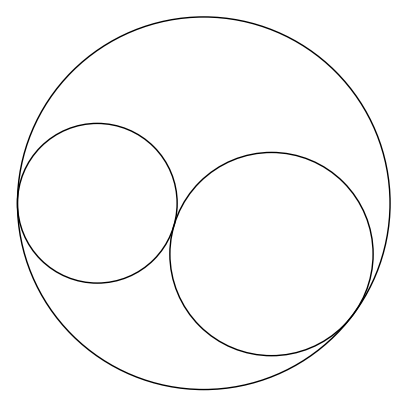}
\includegraphics[width=1.5in]{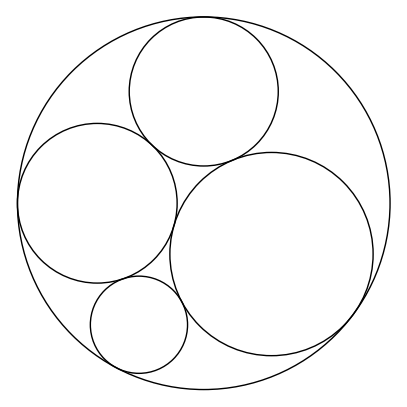}
\includegraphics[width=1.5in]{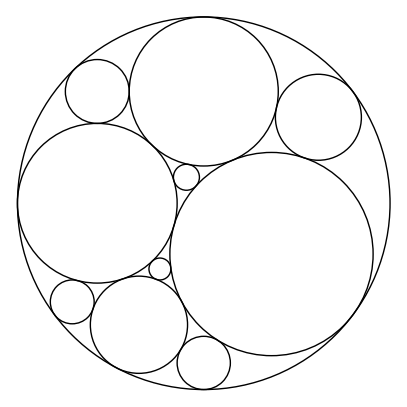}
}
$\;\ldots\;$
\raisebox{-.49\height}{
\includegraphics[width=3.5in]{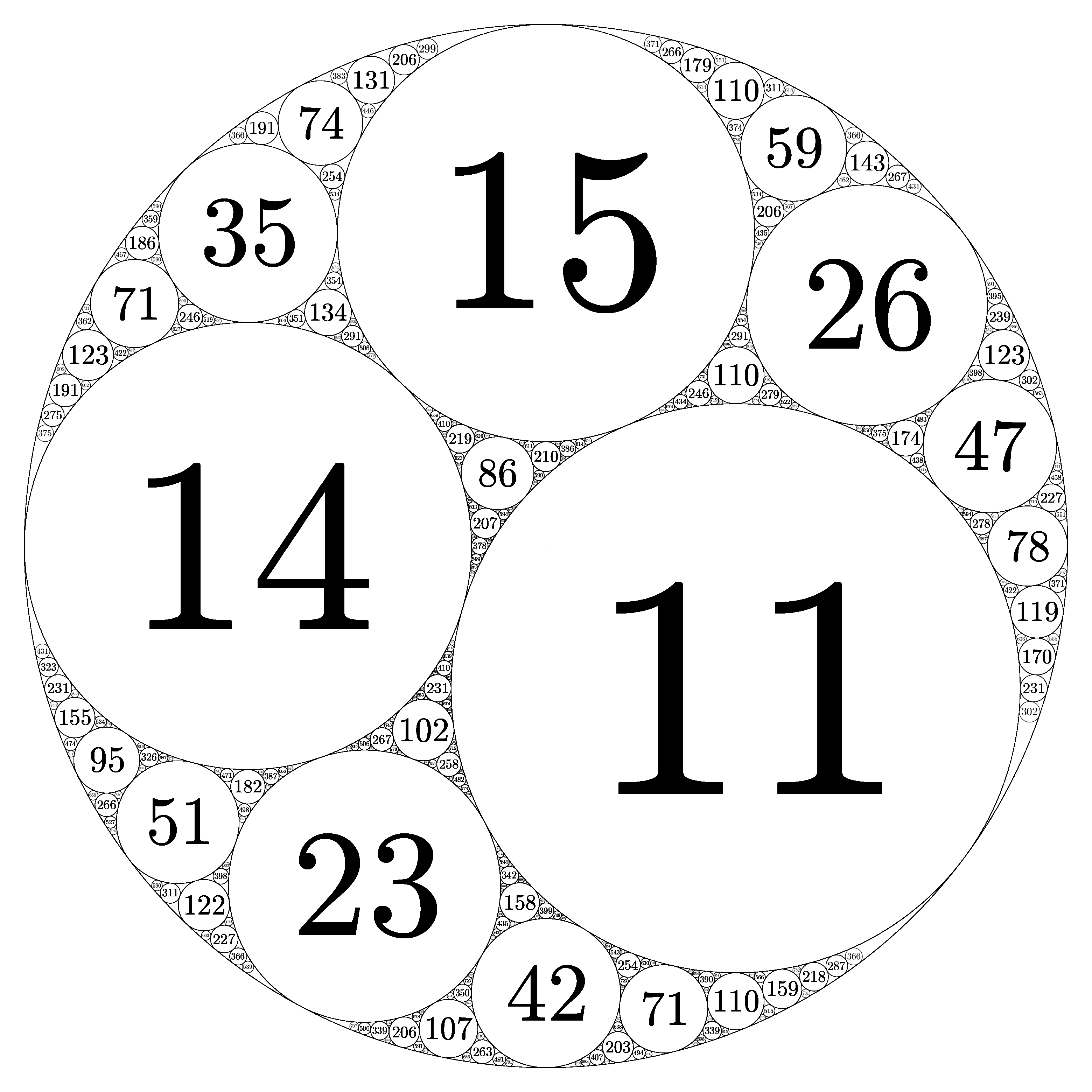}
}
\caption{\small Generating an Apollonian circle packing; at each stage, curvilinear triangles are filled with tangent circles.  In the final packing, curvatures are indicated.} 
\label{fig:build}
\end{figure}

\begin{figure}
\includegraphics[width=2.8in]{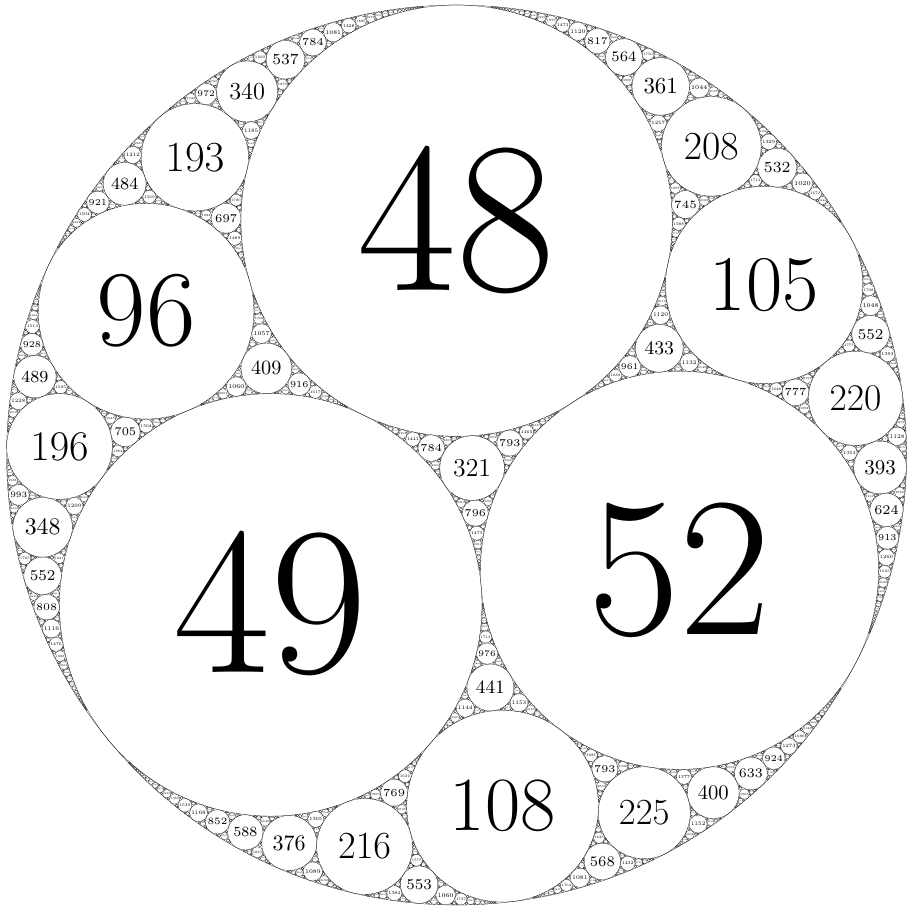}
\includegraphics[width=2.8in]{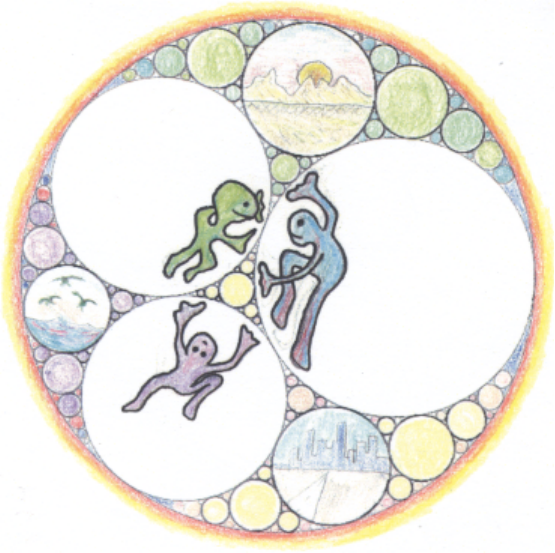}
	\caption{Several examples of Apollonian circle packings.  The packing at left can be zoomed in quite amazingly far if you are on a computer screen, and was produced with user-friendly public software created by James Rickards \cite{RickardsSoftware}.  The packing with people in it was drawn by Katherine Sanden; a few other fanciful packings can be enjoyed in \cite{FuchsSanden}.}
	\label{fig:examples}
\end{figure}

\subsection{Apollonian group}

The structure of the Apollonian packing is governed by the \emph{Apollonian group} $\mathcal{A}$, which acts freely and transitively on the collection of its unordered Descartes quadruples \cite[Theorem 4.3]{GLMWY-geometry}.  In particular, it acts on quadruples of curvatures $(a,b,c,d)$ satisfying $Q(a,b,c,d)=0$, and therefore is considered a subgroup of the orthogonal group $O_Q(\ZZ)$ preserving the form $Q$.  The Descartes form $Q$ has signature $(3,1)$.
Specifically, $\mathcal{A}$ is generated by the four matrices
\begin{equation}
        \tiny
        \label{eqn:iappgens}
	S_1 :=
        \begin{pmatrix}
                -1 & 0 & 0 & 0 \\
                2 & 1 & 0 & 0 \\
               2 & 0 & 1 & 0 \\
                2 & 0 & 0 & 1 \\
        \end{pmatrix},
	S_2 :=
        \begin{pmatrix}
                1 & 2 & 0 & 0 \\
                0 & -1 & 0 & 0 \\
                0 & 2 & 1 & 0 \\
                0 & 2 & 0 & 1 \\
        \end{pmatrix},
	S_3 :=
        \begin{pmatrix}
                1 & 0 & 2 & 0 \\
                0 & 1 & 2 & 0 \\
                0 & 0 & -1 & 0 \\
                0 & 0 & 2 & 1 \\
        \end{pmatrix},
	S_4 :=
        \begin{pmatrix}
                1 & 0 & 0 & 2 \\
                0 & 1 & 0 & 2 \\
                0 & 0 & 1 & 2 \\
                0 & 0 & 0 & -1 \\
        \end{pmatrix},
\end{equation}
acting on row vectors of curvatures from the right.   
Each generator corresponds to fixing three of the circles in a Descartes quadruple and `swapping' out one Soddy circle for its alternative (see Figure~\ref{fig:swaps}).  The Descartes quadruples in any one packing constitute one orbit of the Apollonian group.

	\begin{minipage}{0.40\textwidth}
		\begin{center}
                \includegraphics[height=1.3in]{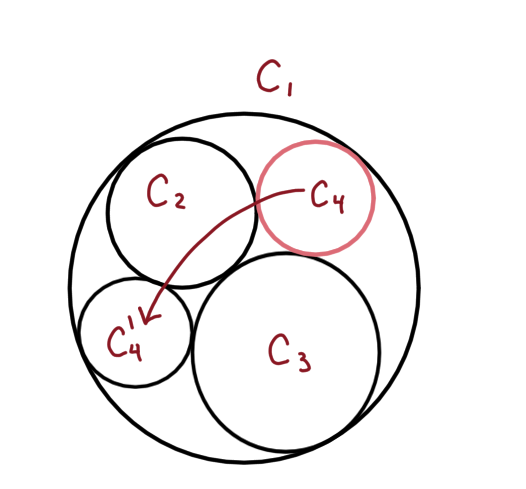}
		\captionof{figure}{An Apollonian swap, replacing the fourth circle $C_4$ with its alternate $C_4'$.}
		\label{fig:swaps}
		\end{center}
	\end{minipage}\hfill
	\begin{minipage}{0.55\textwidth}
		\begin{center}\includegraphics[height=2.0in]{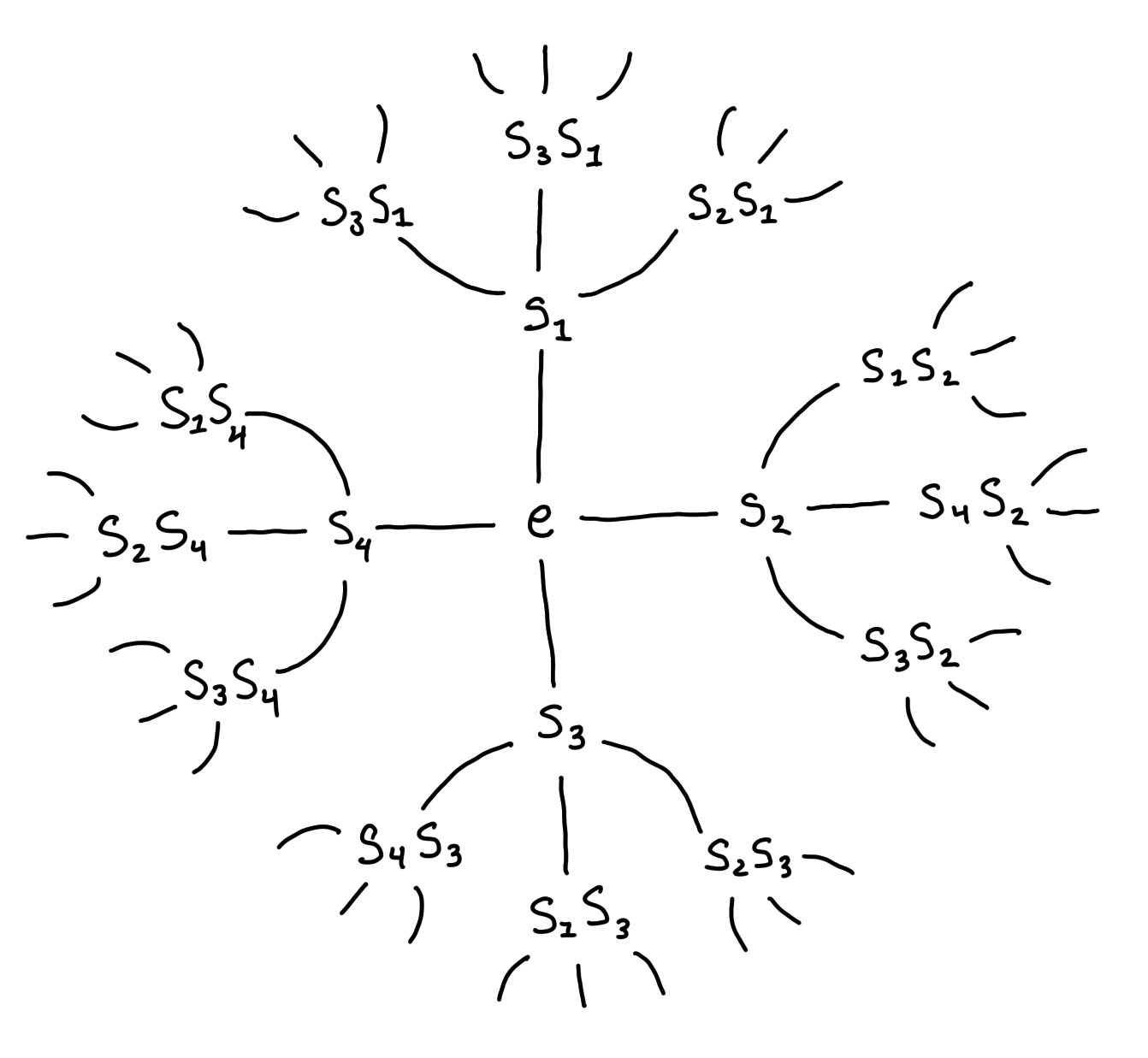}
		\captionof{figure}{Cayley graph of $\mathcal{A}$, shown to two levels.}
		\label{fig:Cayley-graph}
		\end{center}
		\end{minipage}

	This group, whose definition goes back to Hirst \cite{MR0209981}, is our point of access to the rich arithmetic structure of the curvatures.  Fuchs showed that it was a \emph{thin group} \cite{FuchsStrong}, meaning that it has infinite index in $O_Q(\ZZ)$ and yet is Zariski dense in the algebraic group $O_Q$.  Zariski density means that if a polynomial $f(\ldots,x_{ij},\ldots)$ vanishes on elements of $\mathcal{A}$ (the $x_{ij}$ representing the entries of the matrix), then it vanishes for all matrices in $O_Q(\RR)$.  In other words, elements of $\mathcal{A}$ cannot be detected by any polynomial condition on its matrix entries.

These matrices satisfy the relations $S_i^2 = I$, and in fact there are no other relations \cite[Proof of Theorem 4.3]{GLMWY-geometry}, so that
 \[
	  \mathcal{A} = \left\langle S_1, S_2, S_3, S_4 : S_i^2 = 1 \right\rangle < \operatorname{O}_{Q}(\ZZ).
        \]
	This means that the Cayley graph is particularly nice.  The \emph{Cayley graph} of a group $G$ with respect to a generating set $S$ is the graph whose vertices are the elements of $G$ and which has a directed edge from $g$ to $sg$ for all $s \in S$ and $g \in G$.  In the case of the Apollonian group, we take $S$ to be the set of generators $S_1, S_2, S_3, S_4$.  Since these are involutions, we can consider the Cayley graph to be an undirected graph of degree $4$.  It will be a tree, as in Figure~\ref{fig:Cayley-graph}.

	\begin{minipage}{0.45\textwidth}
				\begin{center}
                \includegraphics[height=2.0in]{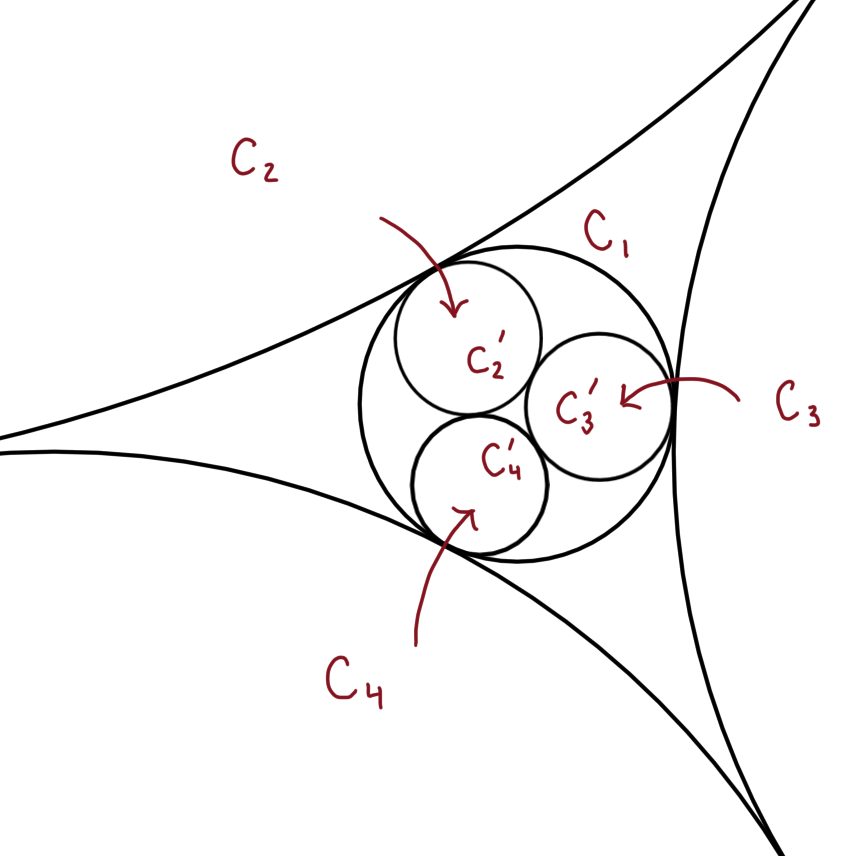}
		\captionof{figure}{Inversion $S_1^\perp$.}
		\label{fig:inversion}
		\end{center}

	\end{minipage}\hfill
	\begin{minipage}{0.5\textwidth}
		\begin{center}
         \includegraphics[height=1.0in]{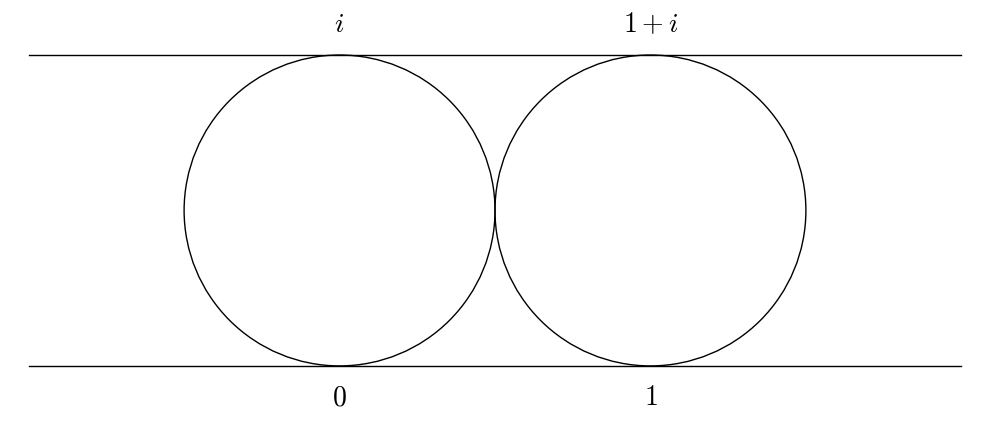}
			\captionof{figure}{The base quadruple.}
		\label{fig:basequad}
		\end{center}
		\end{minipage}

\subsection{Super-Apollonian group}

Graham, Lagarias, Mallows, Wilks and Yan defined the \emph{super-Apollonian group} by adding four circle inversions $S_i^\perp$ to the Apollonian group, one for inverting into each of the four circles of the Descartes quadruple (Figure~\ref{fig:inversion}).  The matrix for $S_i^\perp$ is the transpose of the matrix $S_i$, which is a consequence of a type of duality \cite{GLMWY-geometry}.  The presentation is
\[
	\langle S_1, S_2, S_3, S_4, S_1^\perp, S_2^\perp, S_3^\perp, S_4^\perp : S_i^2 = (S_i^\perp)^2, S_j S_k^{\perp} = S_k^\perp S_j, j \neq k \rangle.
\]
This has finite index in $O_Q(\ZZ)$, so it is no longer thin.  The words of length $5$ taken in normal form (eliminating $S_i^2$, $(S_i^\perp)^2$, $S_i^\perp S_j$) are shown in Figure~\ref{fig:5word}.  In fact, the full orbit coincides with Schmidt's subdivision shown in Figure~\ref{fig:schmidt-farey-gauss}.
This gives a little bit of perspective on the way in which the Apollonian circle packings form an essential `piece' of Schmidt's vision of $\widehat{\CC}$.

	\begin{figure}
                \includegraphics[height=5.0in]{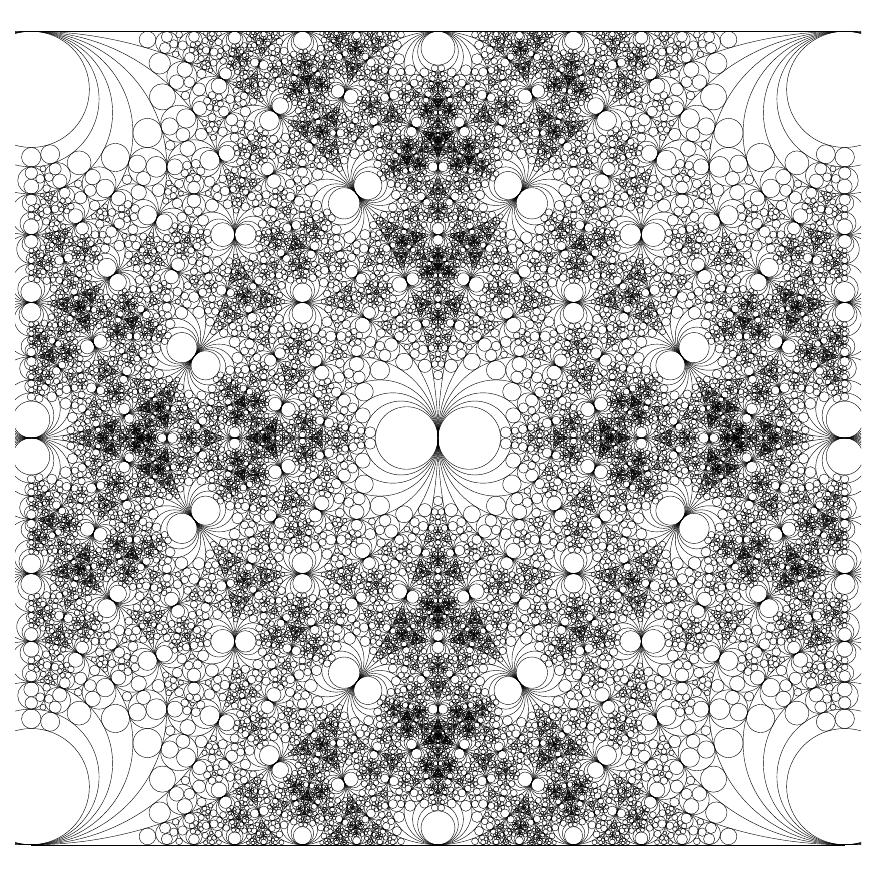}
		\caption{Circles obtained from the base quadruple of Figure~\ref{fig:basequad} using words of length six in normal form from the super-Apollonian group.  This illustrates the invariant measure associated to super-Apollonian continued fractions.  Image:  Robert Hines; see also \cite[Figure 7]{dynamics}.}
		\label{fig:5word}
	\end{figure}

	Using this super-Apollonian perspective, one can define continued fractions for the complex plane, for approximating elements of $\CC$ by Gaussian rationals \cite{dynamics}; this is distinct from, but related to, Schmidt's method \cite{SchmidtComplex}.

\subsection{Geometric Apollonian group}

	The Apollonian group has another incarnation, sometimes called the \emph{geometric} Apollonian group (as distinct from its \emph{algebraic} version above).  

	From this perspective, we view the strip packing as the orbit of the base quadruple shown in Figure~\ref{fig:basequad}, under a group of M\"obius transformations called the \emph{geometric Apollonian group}.  In fact, these transformations can be taken to have entries from $\ZZ[i]$, the Gaussian integers, so that we obtain a group $\mathcal{A}^{geo} < \operatorname{PSL}_2(\ZZ[i]) \rtimes \langle \tau \rangle$, where $\tau$ is complex conjugation.  This larger group of conformal maps on $\widehat{\CC}$ is sometimes called the \emph{generalized M\"obius transformations} \cite{GLMWY-geometry}.  One set of generators that suffices is \cite[equation (8)]{StangeBianchi}
\begin{equation*}
        \label{eqn:appmob}
        z \mapsto \frac{(2i+1)\overline{z}-2}{2\overline{z}+2i-1}, \quad
        z \mapsto -\overline{z}+2, \quad
        z \mapsto \frac{\overline{z}}{2\overline{z}-1}, \quad
        z \mapsto -\overline{z}.
\end{equation*}
This generates the strip packing (Figure~\ref{fig:bounded-types}), although one must then scale by a factor of $2$ to obtain a primitive integral packing (note that the base quadruple in Figure~\ref{fig:basequad} has curvatures $0,0,2,2$).  All other Apollonian packings are images of the strip packing under some M\"obius transformation.  

\begin{exercise}
	Using the fact that the M\"obius transformations act triply transitively on $\widehat{\CC}$, show that any two Apollonian packings are related by a M\"obius transformation.
\end{exercise}

	To relate the algebraic and geometric Apollonian groups, one can use the \emph{exceptional isomorphism}
	\begin{equation}
		\label{eqn:exceptional}
		\PGL_2(\CC) \rightarrow \operatorname{SO}^+_{1,3}(\RR).
	\end{equation}
	For now, we leave aside the details (which are finicky).  This is sometimes discussed in the language of the \emph{spin homomorphism} $\SL_2(\CC) \rightarrow \operatorname{SO}_{1,3}(\RR)$ or \emph{spin double cover}.  See \cite{FuchsStrong, GLMWY-geometry}.

	The group $\PGL_2(\CC)$ can be identified with $\operatorname{Isom}(\HH^3_U)$, the isometries of the hyperbolic 3-space realized as the upper half space whose boundary is $\widehat{\CC}$.  Under this interpretation, $\mathcal{A}^{geo}$ is a \emph{Kleinian group} (that is, a discrete subgroup of $\PGL_2(\CC)$), with an infinite volume quotient hyperbolic $3$-manifold $\mathcal{A}^{geo} \backslash \HH^3_U$.  It is geometrically finite (in this context, this means the fundamental domain can be taken to have finitely many sides).

	\subsection{Limit set and circle growth}

	Viewing an Apollonian packing $\mathcal{P}$ as a subset of $\widehat{\CC}$, we define the \emph{residual set} $\Lambda(\mathcal{P})$ as the closure of the union of the circles of the packing.  There are countably many circles and countably many tangency points, but there are uncountably many points in $\Lambda(\mathcal{P})$ which are added in the closure process.  The complement of $\Lambda(\mathcal{P})$ consists of the interiors of all the circles.

	For $\mathcal{P}$ the strip packing, $\Lambda(\mathcal{P})$ is equal to the limit set of $\mathcal{A}^{geo}$ (i.e., the set of accumulation points for the orbit of a point under $\mathcal{A}^{geo}$) (a similar statement holds for other packings, taking an appropriate conjugate of $\mathcal{A}^{geo}$).

	\begin{theorem}[{\cite{McMullen}}]
		\label{thm:Hausdorff-dim}
		The Hausdorff dimension of $\Lambda(\mathcal{P})$ is $\alpha := 1.30568\ldots$. 
	\end{theorem}
	This constant, for which no closed form is known, was recently rigorously computed to an impressive 128 digits by Vytnova and Wormell \cite{Polina}.

	\begin{exercise}
		Let $\Lambda^*(\mathcal{P})$ temporarily denote the complement of the interiors of $\mathcal{P}$.  Hirst \cite{MR0209981} showed this has Hausdorff dimension less than $2$, hence measure zero.  Use this to show that $\Lambda^*(\mathcal{P}) = \Lambda(\mathcal{P})$.
	\end{exercise}

	The Hausdorff dimension controls the growth of the number of circles in the packing in terms of size:
	\begin{theorem}[{\cite{KontorovichOh}}]
		Let $N_\mathcal{P}(X) = \# \{ C \in \mathcal{P} : \operatorname{curv}(C) < X \}$.  Then
\[
		N_\mathcal{P}(X) \sim c_\mathcal{P} X^{\alpha}.
	\]
\end{theorem}
An error term and other refinements came afterward \cite{LeeOh, OhShah}.  The constant $c_\mathcal{P}$ is called the Apollonian constant for the packing $\mathcal{P}$ and a formula is given in \cite[Remark 2.9]{Vinogradov}.

	For further details in this direction, consult \cite[Theorem 4.1]{GLMWY-geometry}, \cite{Oh} and \cite{OhICM}.

\subsection{Quadratic forms}

The geometric Apollonian group has as a subgroup the \emph{$2$-congruence subgroup}, the kernel in $\PSL_2(\ZZ)$ under reduction modulo $2$:
\[
	\Gamma(2) := \{ M \in \PSL_2(\ZZ) : M \equiv I \pmod{2} \} \subseteq \PSL_2(\ZZ).
\]
This is an arithmetic group (much more is known for arithmetic groups than for thin groups), and is a \emph{Fuchsian group}, that is, a discrete subgroup of $\PSL_2(\RR)$.  The group $\Gamma(2)$ is the subgroup of $\mathcal{A}^{geo}$ stabilizing $\widehat{\RR}$.

%Let us consider a family of circles tangent to $\widehat{\RR}$.
Take the three circles of the base quadruple tangent to $\widehat{\RR}$ (Figure~\ref{fig:basequad}); those of radius $1/2$ centred on $i/2$ and $1+i/2$, and the circle $i+\widehat{\RR}$, i.e. the horizontal line through $i$ (this is indeed tangent to $\widehat{\RR}$ at $\infty$; exercise).
The group $\Gamma(2)$ preserves tangencies, and so the orbit of the base quadruple under $\Gamma(2)$ gives a family of circles tangent to $\widehat{\RR}$.
This is the family of
the Ford circles:  the circles tangent to $\widehat{\RR}$ at each rational number $p/q \in \QQ$ (in lowest form) of radius $1/2q^2$ (plus the horizontal line through $i$).  See Figure~\ref{fig:ford}.  
Their curvatures, as a family, are exactly twice the perfect squares.

\begin{exercise}
	Prove this statement (that the Ford circles are the indicated orbit).
\end{exercise}

	\begin{figure}
		\includegraphics[height=2.0in]{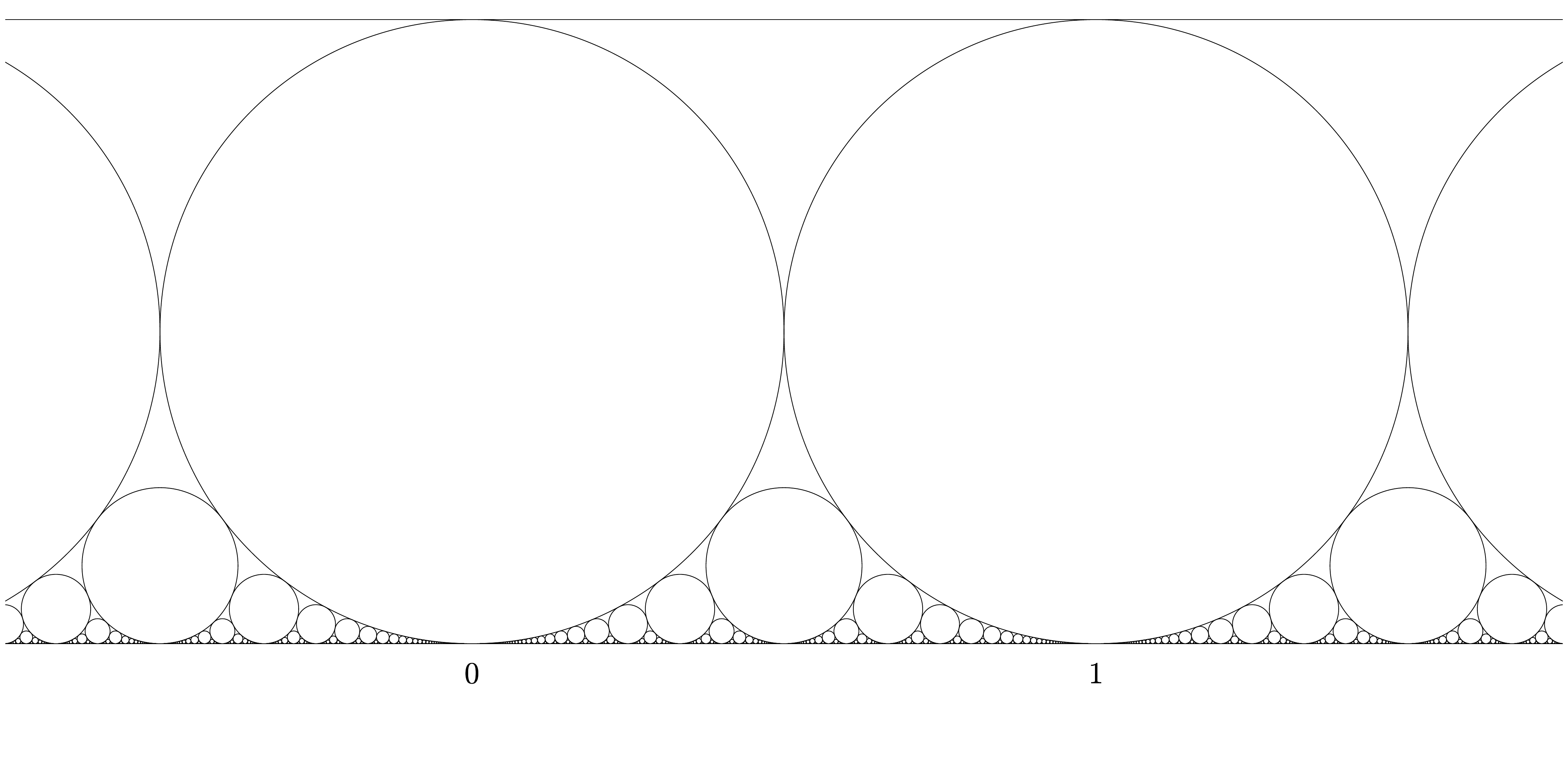}
		\caption{The Ford circles.}
		\label{fig:ford}
	\end{figure}

	\begin{figure}
		\includegraphics[height=4.0in]{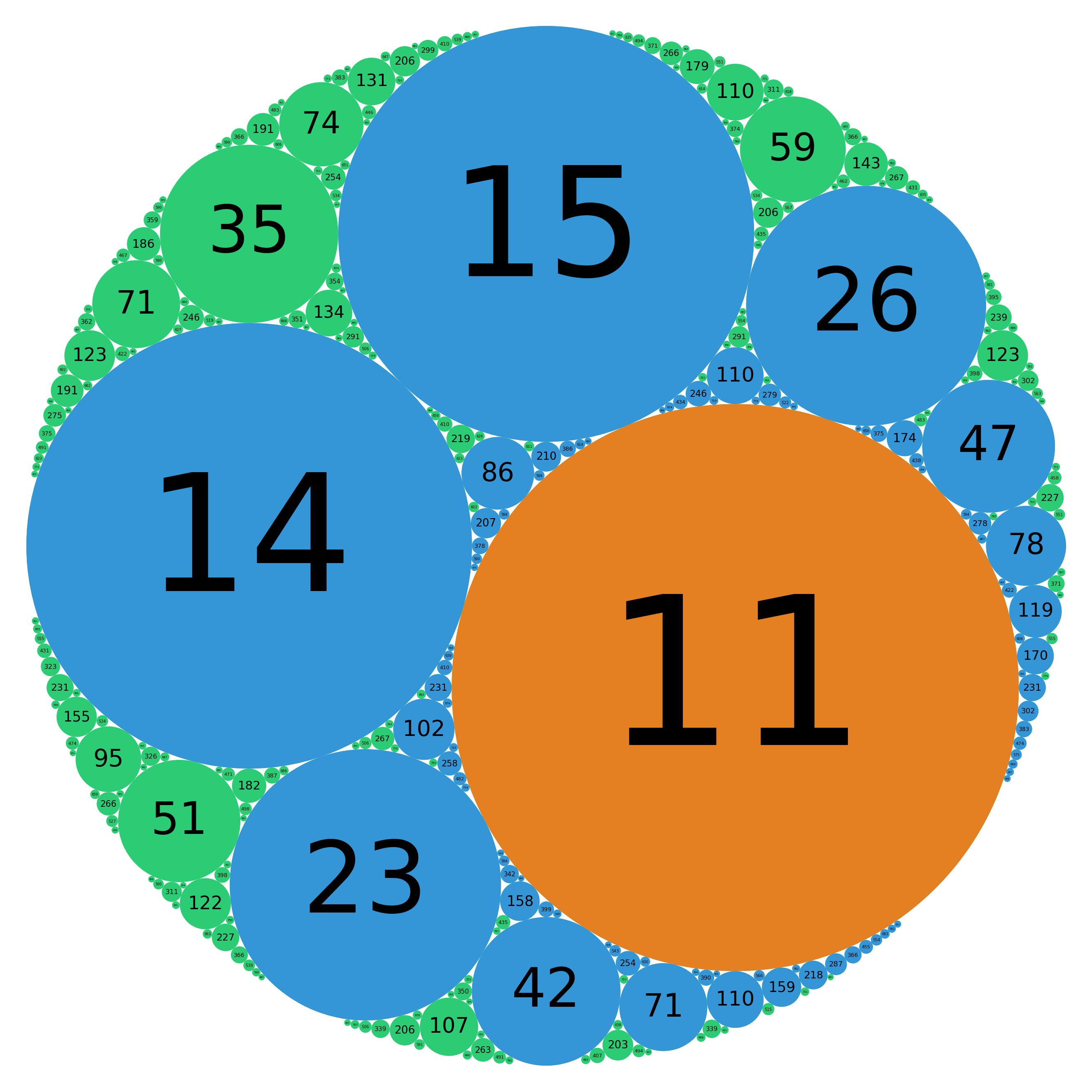}
		\caption{A mother circle (orange) and the family of circle tangent to it (blue).  Other circles in the packing are green.}
		\label{fig:motherfamily}
	\end{figure}

	More generally, fix a circle $\mathcal{C}$, which we will call the \emph{mother circle}.  Letting $M \cdot \widehat{\RR} = \mathcal{C}$, the subgroup $M \Gamma(2) M^{-1}$ is the stabilizer of $\mathcal{C}$ in $M \mathcal{A}^{geo} M^{-1}$ and gives rise to a collection of circles tangent to $\mathcal{C}$ (see Figure~\ref{fig:motherfamily}) as the image of the Ford circles under $M$.  The curvatures of this family of circles is exactly the set of primitively represented values of a translated quadratic form.  Filling out the details of this relationship proves the following theorem, first observed in this form by Sarnak \cite{Sarnakletter}, but present in another form in \cite{GLMWY-number}.

\begin{theorem}
	\label{thm:curv-form}
	Let $\mathcal{C}$ be a circle of curvature $c$ within an Apollonian circle packing $\mathcal{P} \subseteq \widehat{\CC}$.  Then there is a real binary quadratic form $f_\mathcal{C}(x,y)$ of discriminant $-4c^2$ such that the set of curvatures of circles tangent to $\mathcal{C}$ within $\mathcal{P}$ is the set
	\[
		\{ f_\mathcal{C}(x,y) - c : x,y \in \ZZ, (x,y) = 1 \}.
	\]
\end{theorem}

\begin{proof}
	The original observation was derived as a consequence of Descartes' relation.  Here we give a proof using $\mathcal{A}^{geo}$, beginning with the strip packing.  Recall that $\mathcal{A}^{geo}$ generates the strip packing from the base quadruple $(0,0,2,2)$ of Figure~\ref{fig:basequad}. %, after a scaling by a factor of two.  We leave the scaling to the very end. 
	Consider the circle which is the horizontal line through $i$.  The corresponding M\"obius transformation from $\widehat{\RR}$ is $\begin{pmatrix} 1 & i \\ 0 & 1 \end{pmatrix}$.  If we post-compose by $\Gamma(2)$ (which, we recall, acts to permute the circles tangent to $\widehat{\RR}$), we have a coset
	\[
		 \Gamma(2) \begin{pmatrix} 1 & i \\ 0 & 1 \end{pmatrix}
	\]
	of M\"obius transformations representing the circles tangent to $\widehat{\RR}$.  However, these are all oriented opposite to how they should be (having negative curvatures), so we must post-compose by, say, $\begin{pmatrix} 0 & 1 \\ 1 & 0 \end{pmatrix}$.  

		Now more generally, consider a packing that is the image of the strip packing under some $M \in \PSL_2(\CC)$, where we wish to parametrize the family of circles tangent to $\mathcal{C} = M \cdot \widehat{\RR}$.  Then this family is given by:
	\[
		M \begin{pmatrix} 0 & 1 \\ 1 & 0 \end{pmatrix} \Gamma(2) \begin{pmatrix} 1 & i \\ 0 & 1 \end{pmatrix}.
	\]
	Now we apply Proposition~\ref{prop:mobR} to compute the curvatures of this family.

	Taking an element $\begin{pmatrix} x & r \\ y & s \end{pmatrix} \in \Gamma(2)$, we obtain
	\begin{align*}
		2\Im( \overline {(x\delta + y \gamma )} {\left( (r\delta + s\gamma) + i(x\delta + y\gamma) \right)} ) 
		&= 2\Im( \overline{(x\delta + y \gamma)} { (r\delta + s\gamma) } )
		 +2\Im( \overline{(x\delta + y \gamma)} { i(x\delta + y\gamma) } ) \\
		& = -2\Im( \overline\gamma {\delta} ) + 2N( x\delta + y\gamma).
	\end{align*}
	The first term is the curvature of $\mathcal{C}$.  The second is a quadratic form in integral variables $x,y$.  
\end{proof}

%For the strip packing with base quadruple $(0,0,1,1)$, we must scale the base quadruple by two, obtaining the form $-\Im(\overline\gamma \delta) + N(x\gamma+y\delta)$.

\begin{exercise}
	Using the the proof above, recover the fact that the Ford circles have curvatures $2x^2$.
\end{exercise}

\begin{exercise}
	Prove the theorem from the Descartes relation.
\end{exercise}

The proof hints at how to compute the form $f_\mathcal{C}$.  Fix a Descartes quadruple $\mathcal{C}_1, \mathcal{C}_2, \mathcal{C}_3, \mathcal{C}_4$ containing $\mathcal{C}_1 = \mathcal{C}$, the mother circle.  Write $[n,a,b,c]$ for the quadruple of curvatures in the same order.  Choose $M$ to take $\infty$, $0$, $1$ to $a,b,c$.  Then, the form is
\[
	f_\mathcal{C}(x,y) = (n+a)x^2 + (n+a+b-c) xy + (n+b)y^2;
\]
because then we recover the curvatures $a,b,c$ from $f_\mathcal{C}(x,y)-n$ where $(x,y) = (1,0), (0,1), (1,1)$ respectively.  Notice that a different choice of quadruple including $\mathcal{C}_1$ corresponds to a different $M$ and a change of variables on $f_\mathcal{C}$ within the $\PGL_2(\ZZ)-$equivalence class.  The following strong statement encompasses this.

\begin{proposition}[{ \cite[Theorem 4.2]{GLMWY-number}, \cite[Proposition 3.1.2]{Staircase} }]
	\label{prop:quad-space}
	Let $n \in \RR$.  The quadruples $[n,a,b,c] \in \RR^4$ satisfying the Descartes quadratic form \eqref{eqn:desc} biject with the set of positive semi-definite (this means not taking negative values) binary quadratic forms $Ax^2 + Bxy + Cy^2$ of discriminant $-4n^2$.  The map is
\[
	\phi: [n,a,b,c] \mapsto (n+a)x^2 + (n+a+b-c) xy + (n+b)y^2.
\]
Furthermore, if we identify quadruples $[n,a,b,c]$ under the action of $S_2$, $S_3$, $S_4$ and under permutation of the last three entries, then equivalence classes of Descartes quadruples $[n,a,b,c]$ are in bijection with $\GL_2(\ZZ)-$equivalence classes of positive semi-definite binary quadratic forms of discriminant $-4n^2$.
\end{proposition}

We have seen that these forms are a reflection of the Fuchsian group $\Gamma(2)$ and its conjugates inside the Apollonian group.  These form one of the main tools used to study Apollonian circle packings (more generally, these are a special feature of certain thin Kleinian groups that cause them to behave similarly to the Apollonian case \cite{FuchsStangeZhang}).

\subsubsection{The space of Descartes quadruples over $\RR$}

Proposition~\ref{prop:quad-space} allows us to map Descartes quadruples into the upper half plane $\HH^2_U$, using the correspondence $Ax^2 + Bxy + Cy^2 \mapsto \frac{-B \pm \sqrt{B^2 - 4AC}}{2A}$ as studied in Section~\ref{sec:hyperupper}.  Thus, $\HH^2_U$ is the space of pairs $(\mathcal{C}, A)$ where $\mathcal{C}$ is a circle and $A$ an Apollonian circle packing containing it, where these pairs are identified under affine transformations (i.e. we only think of the geometry of the packing, not its position in space).

	\begin{figure}
		\includegraphics[height=1in]{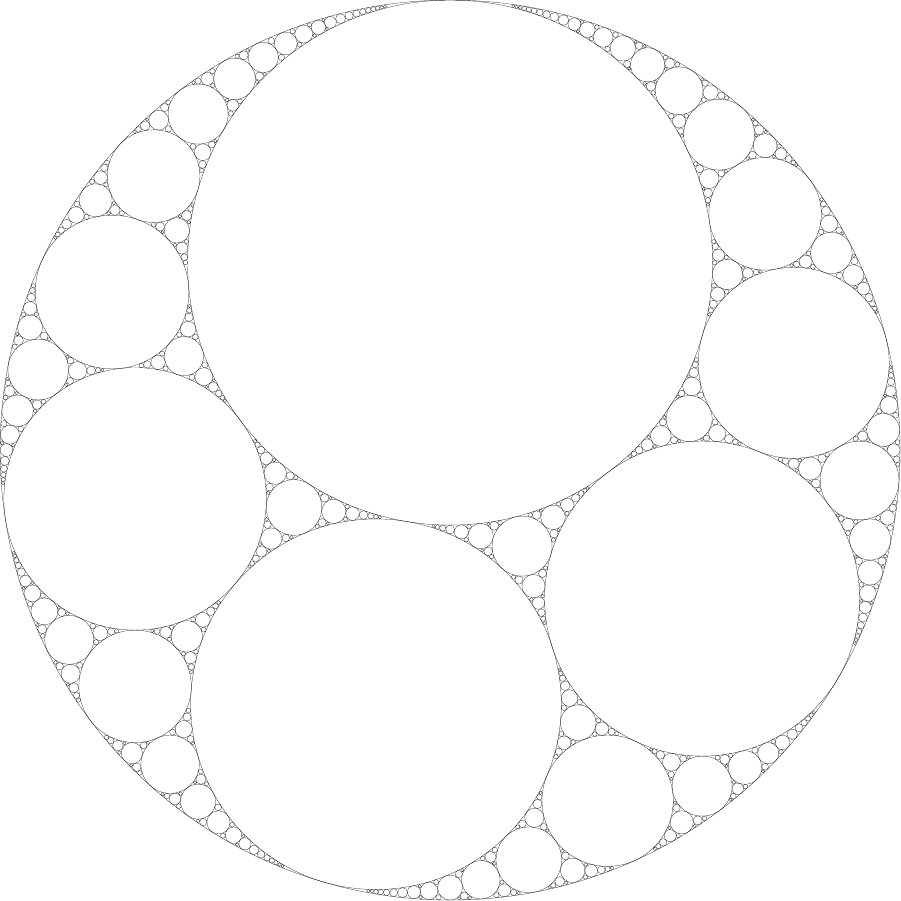}
		\includegraphics[height=1in]{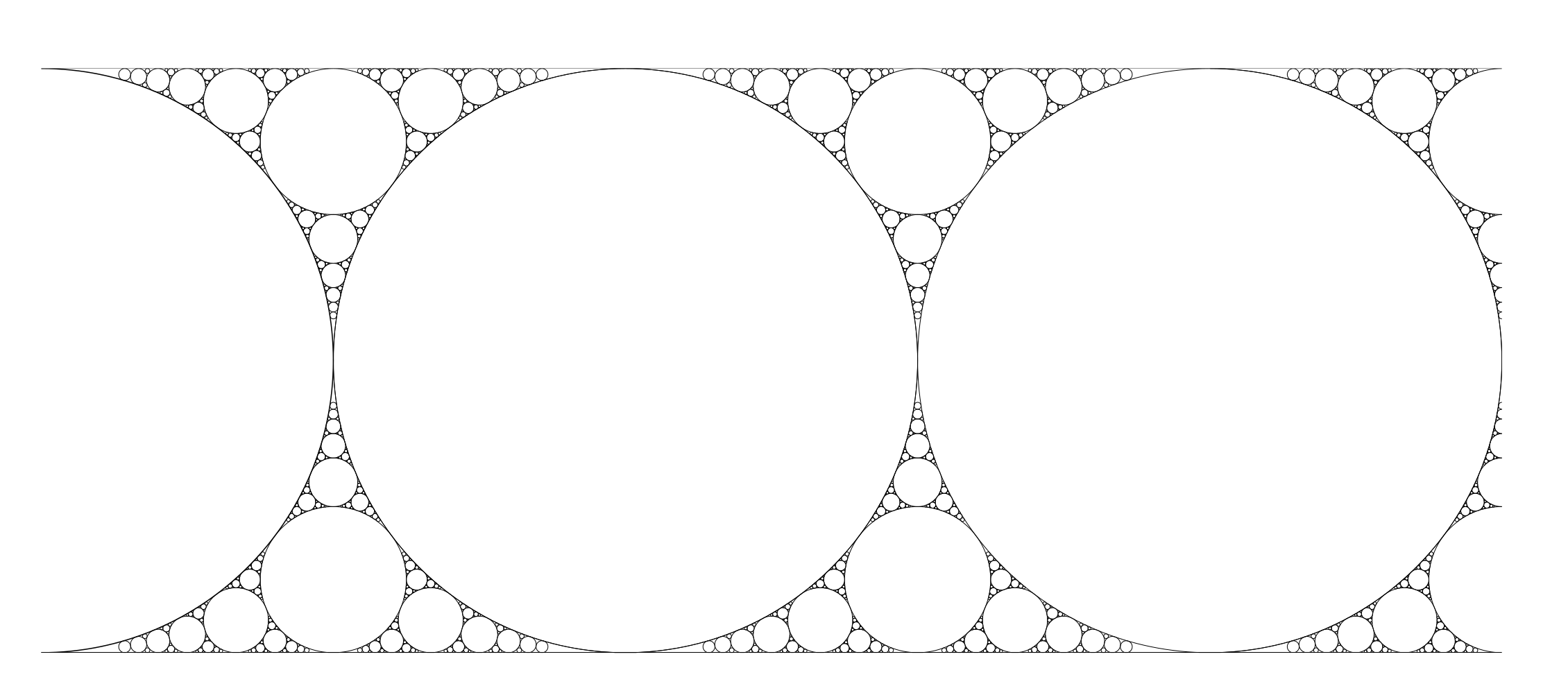} \\
		\includegraphics[height=1.0in]{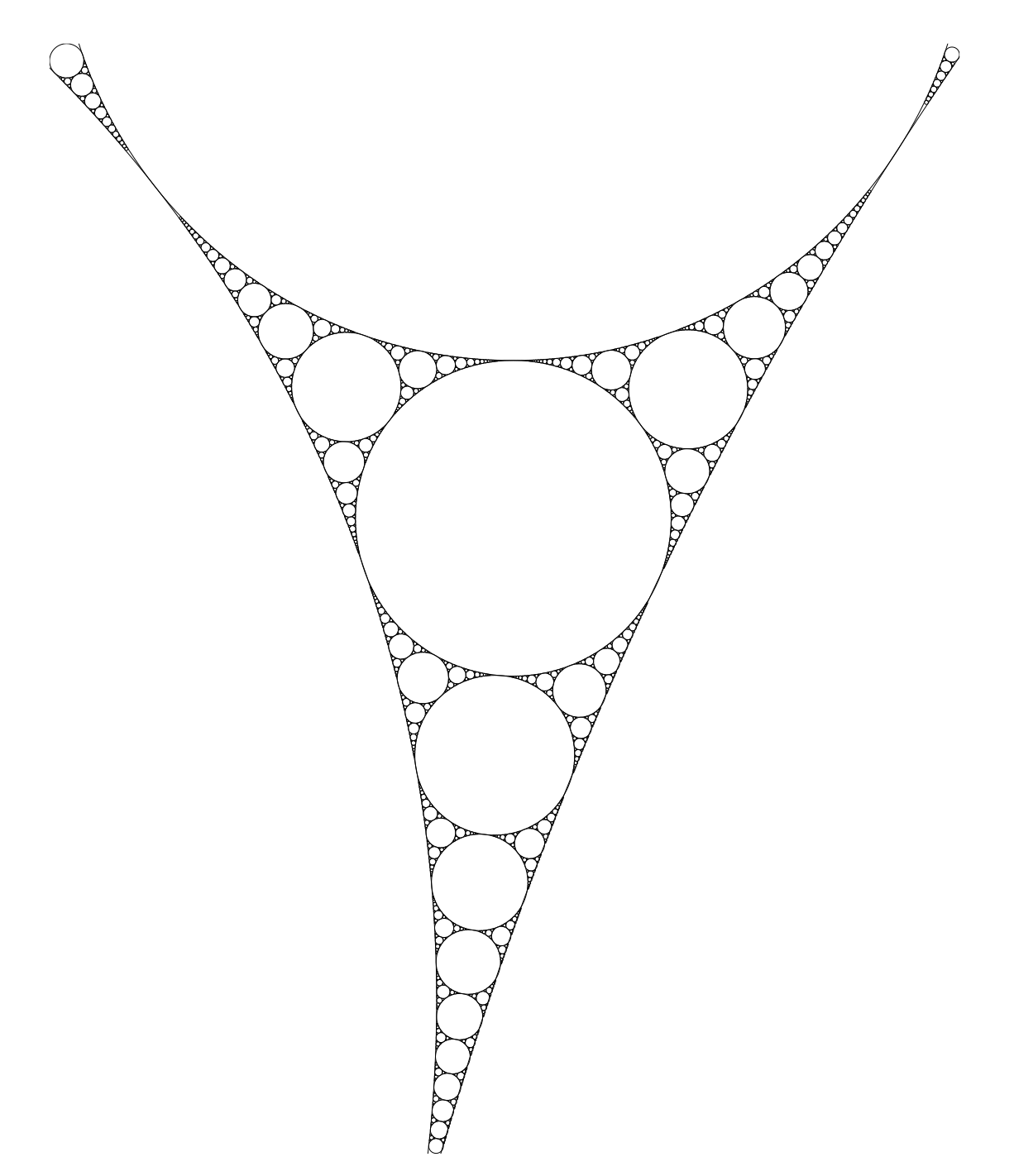}
		\includegraphics[height=1in]{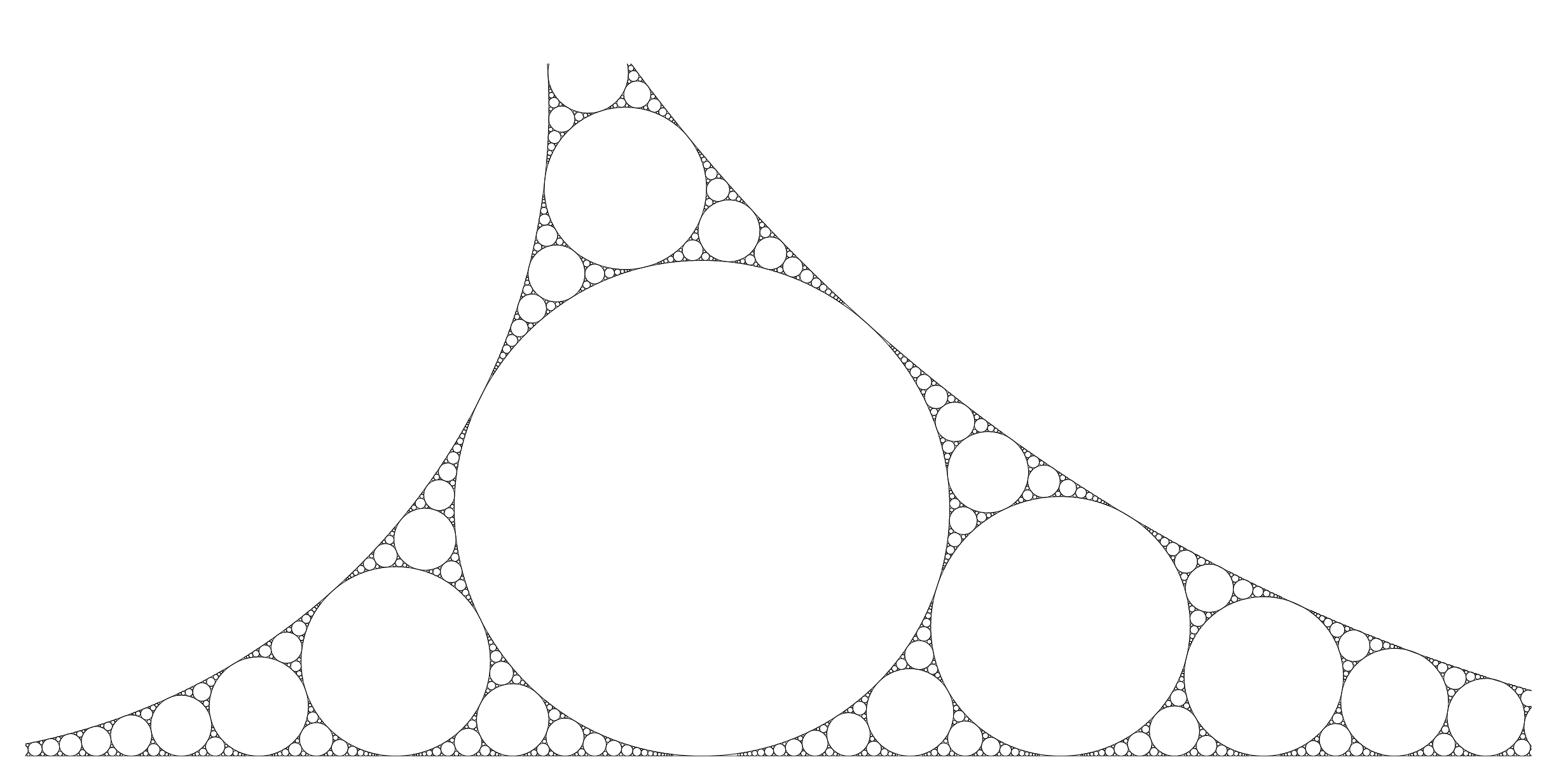}
		\caption{The various types of packings, clockwise from upper left:  bounded, strip, half-plane, full-plane.  Images:  James Rickards.}
		\label{fig:bounded-types}
	\end{figure}

	\begin{figure}
		\includegraphics[width=4.0in]{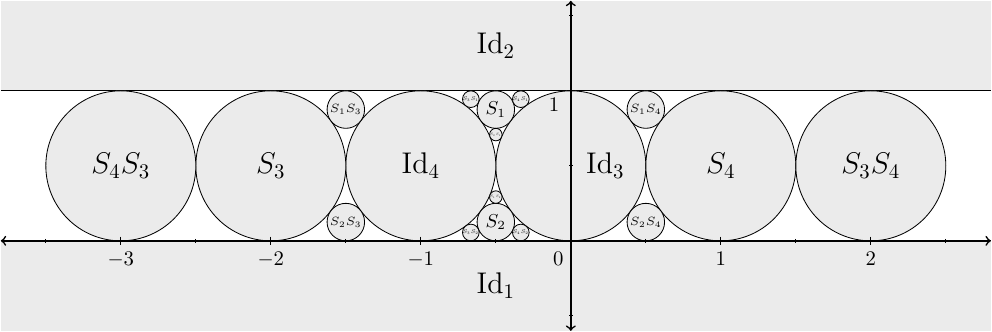}
		\caption{The upper half plane as a parameter space of Descartes quadruples.  That's right, the \emph{parameter space} of Apollonian circle packings, marked by type, is itself an Apollonian circle packing!  In this image, the interior of each circle contains bounded packings, and they are labelled by their \emph{depth element}, i.e., the elements of the Apollonian group which must be applied to obtain the `nearest' quadruple involving the outer bounding circle.  Image: \cite[Figure 6]{Staircase}.}
		\label{fig:quadruple-space}
	\end{figure}

This raises an interesting question:  can we explore this parameter space?  Since we are in the context of real curvatures (not necessarily integral or rational), the packing generated by a quadruple has more variety than we have so far discussed:  there are \emph{bounded} packings, \emph{half-plane} packings, and \emph{full-plane} packings (see Figure~\ref{fig:bounded-types}).  

It is natural to ask which type of packing one obtains -- for example, what is the locus, in the upper half plane, of the bounded packings?  The answer is striking:  colouring the parameter space according to the type of packing results in the strip packing itself in the upper half plane!  Let $\HH^2_U$ denote the upper half plane considered as a parameter space of quadruples, and let $\mathcal{P}$ be the strip packing inside $\HH^2_U$.  Then:
\begin{enumerate}
	\item the bounded packings occur exactly in the interiors of the circles of $\mathcal{P}$;
	\item the strip packing occurs at each tangency point of $\mathcal{P}$;
	\item the half-plane packings occur on the circles of $\mathcal{P}$, away from tangency points;
	\item the full-plane packings occur at all remaining points, which are the points in $\Lambda(\mathcal{P}) \setminus \mathcal{P}$. 
\end{enumerate}
In Figure~\ref{fig:quadruple-space}, we label the circles of $\mathcal{P}$ by their \emph{depth}:  the number of Apollonian swaps required to reach the outer circle.
This is explained in \cite{Staircase}; see also variations on this parameter space in \cite{Holly} and \cite{Kocik}.

\section{Apollonian circle packings:  number theory aspects}

\subsection{What is known about curvatures?}

Which integer curvatures appear as curvatures in a primitive integral Apollonian circle packing?  It is evident that not all curvatures can appear, since the packing doesn't have enough geometric `room' to fit all the small curvatures.  However, once we start collating the larger curvatures, we begin to see many repeats.  So it may be reasonable to believe that we eventually begin to see every sufficiently large curvature.

\begin{figure}
    \includegraphics[height=2.1in]{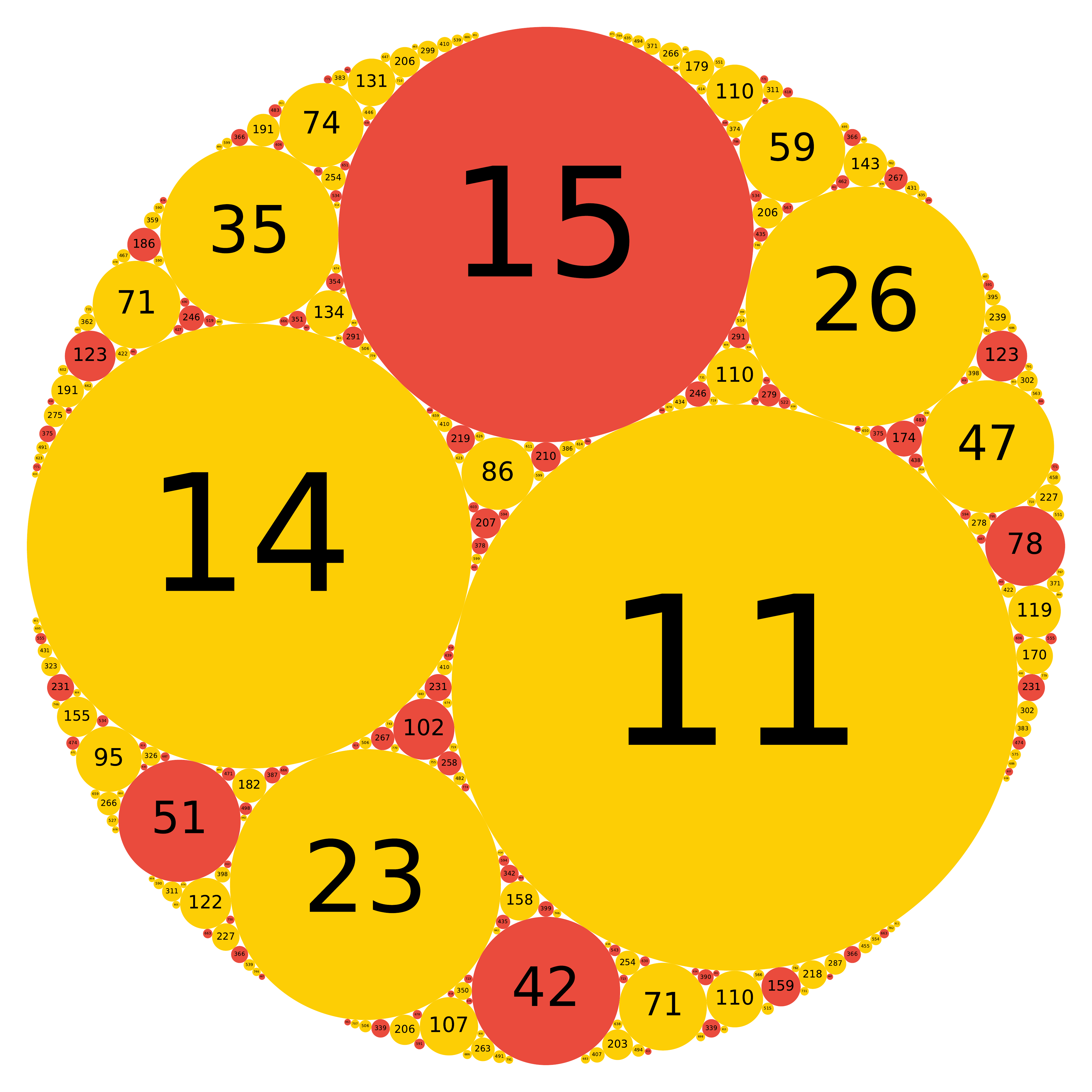}
    \includegraphics[height=2.1in]{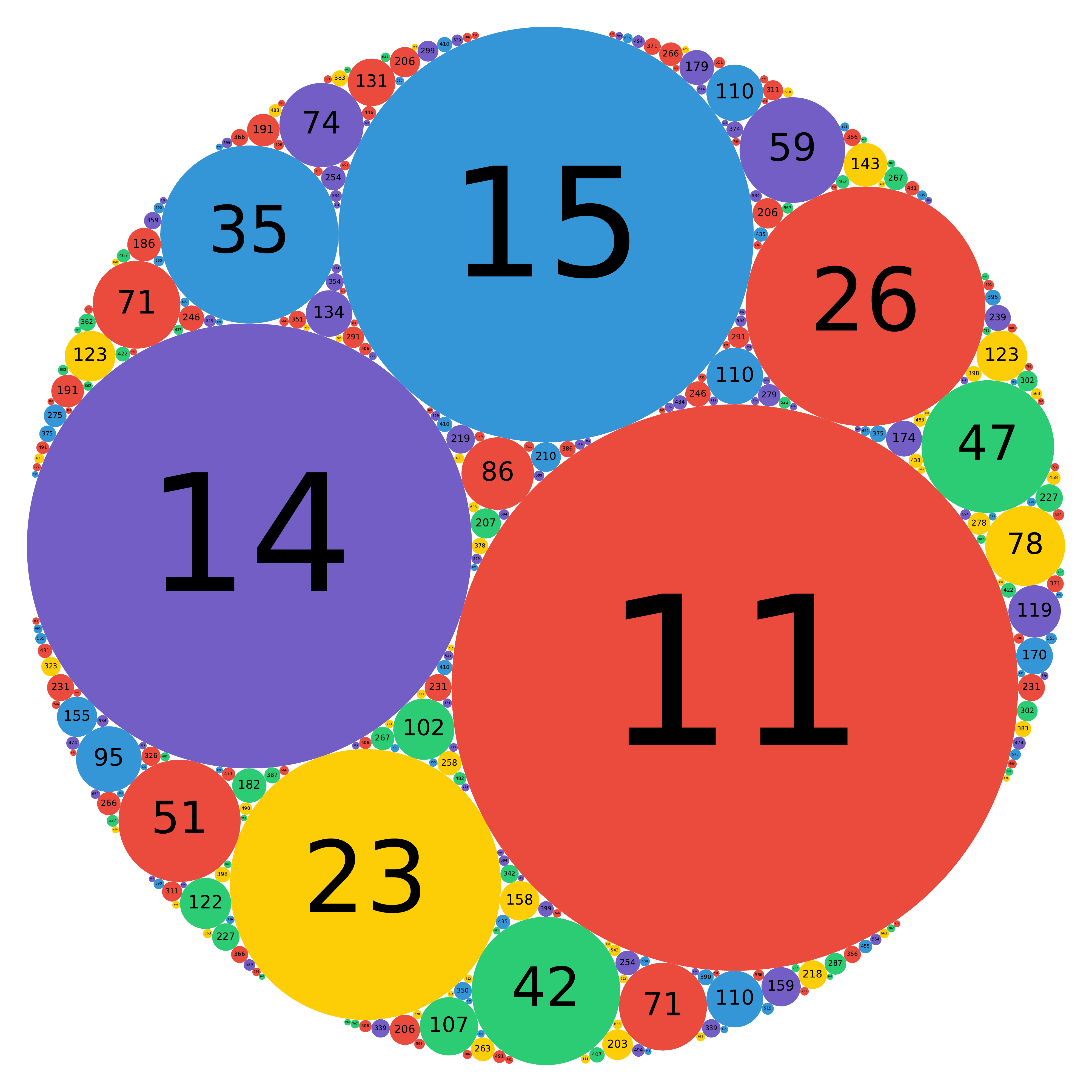}
	\caption{An Apollonian packing coloured by residue modulo $3$ and $5$, respectively.}
	\label{fig:reduced}
\end{figure}

This is not true however.  See Figure~\ref{fig:reduced}.  Graham, Lagarias, Mallows, Wilks and Yan observed a congruence restriction on curvatures \cite{GLMWY-number}.  Precisely, they observed that certain residue classes modulo $24$ were sometimes entirely avoided by the set of curvatures in a primitive integral Apollonian circle packing.  Specifically, when reducing the curvatures of an Apollonian packing modulo $24$, one obtains one of the following six possible sets of \emph{admissible curvatures} (for a complete proof, see \cite[Proposition 2.1]{HKRS}):
	\begin{center}
		{\upshape
		\begin{tabular}{|c|c|} 
			\hline
			type      & residues \\ \hline
			$(6, 1)$  & $0, 1, 4, 9, 12, 16$\\ \hline
			$(6, 5)$  & $0, 5, 8, 12, 20, 21$\\ \hline
			$(6, 13)$ & $0, 4, 12, 13, 16, 21$\\ \hline
			$(6, 17)$ & $0, 8, 9, 12, 17, 20$\\ \hline
			$(8, 7)$  & $3, 6, 7, 10, 15, 18, 19, 22$\\ \hline
			$(8, 11)$ & $2, 3, 6, 11, 14, 15, 18, 23$\\ \hline 
		\end{tabular}
		}
	\end{center}
	Each admissible set is assigned a \emph{type} for reference, which is $(n,k)$ where $n$ is the number of residues and $k$ is the smallest residue coprime to $24$.

	To understand the local (i.e., congruence) obstructions that may appear for a modulus $n$, one can label the Cayley graph by Descartes quadruples (beginning from the root labelled with any fixed quadruple of the packing), and reduce modulo $n$, identifying vertices with identical quadruples modulo $n$.  We obtain a picture such as that in Figure~\ref{fig:reduced-cayley}.  For some moduli, the reduced Cayley graph is small, and some residue classes are missed amongst the curvatures (as in the figure).  This proves that the corresponding residue class can never occur as the residue of a curvature in the Apollonian circle packing.

\begin{figure}
    \includegraphics[height=2.1in]{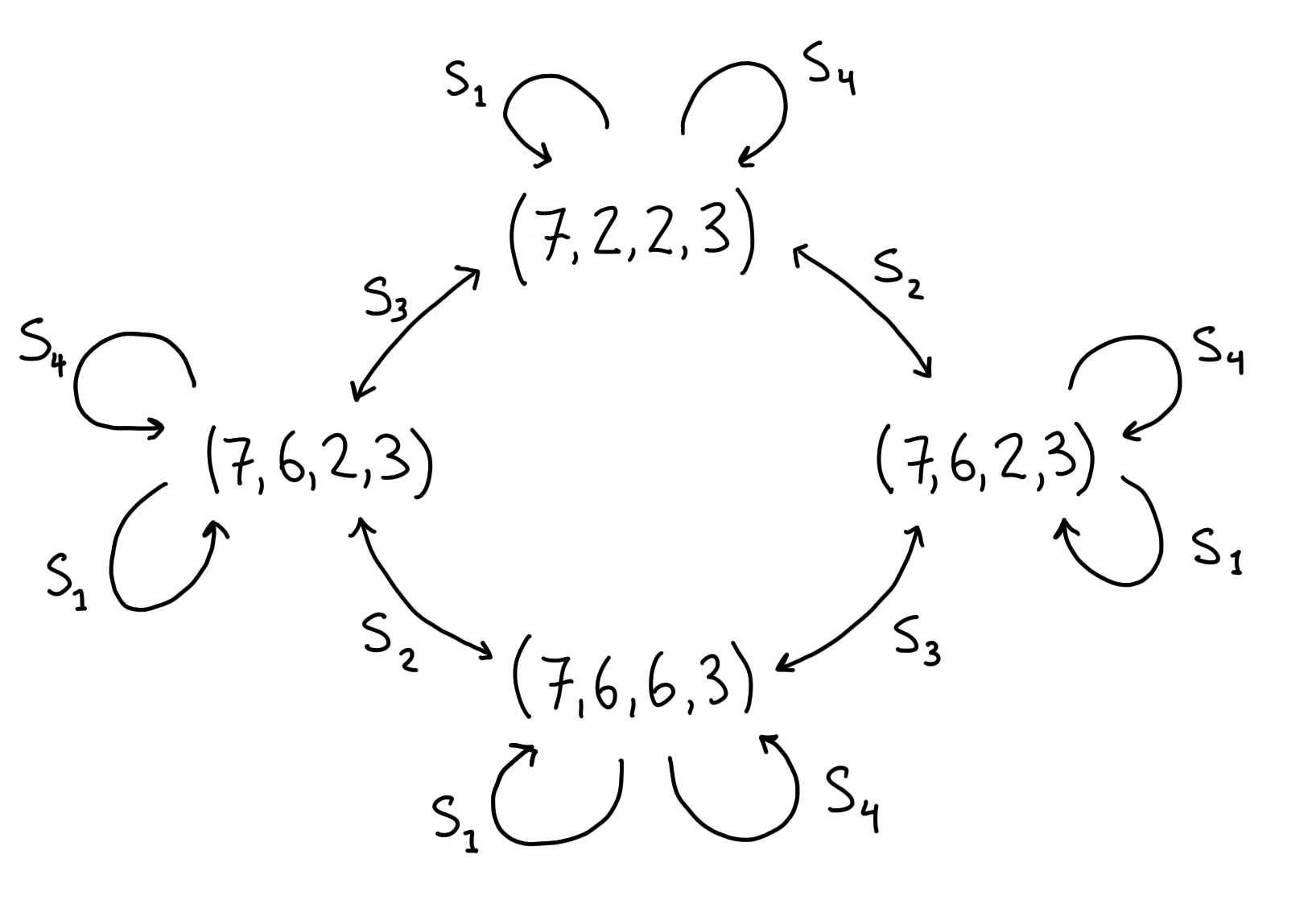}
	\caption{An Apollonian packing reduced modulo $8$.}
	\label{fig:reduced-cayley}
\end{figure}

Graham, Lagarias, Mallows, Wilks and Yan conjectured that this obstruction is essentially all we expect, giving the set of curvatures a positive density in the natural numbers \cite[Positive Density Conjecture]{GLMWY-number}; and that maybe even only finitely many exceptions beyond the congruence obstructions occur \cite[Strong Density Conjecture]{GLMWY-number}.  This came to be known as the \textbf{local-to-global conjecture} for Apollonian circle packings, later refined by Fuchs-Sanden \cite{FuchsSanden}, who collected data on several packings, computing curvatures up to size $5 \times 10^8$.  In particular, they computed the multiplicity of appearance for curvatures.  An example of the results is shown in Figure~\ref{fig:histo}, and this indicates that, for most admissible curvatures, the expected multiplicity grows with the curvature size.  

\begin{figure}
        \begin{center}
                \includegraphics[height=1.7in]{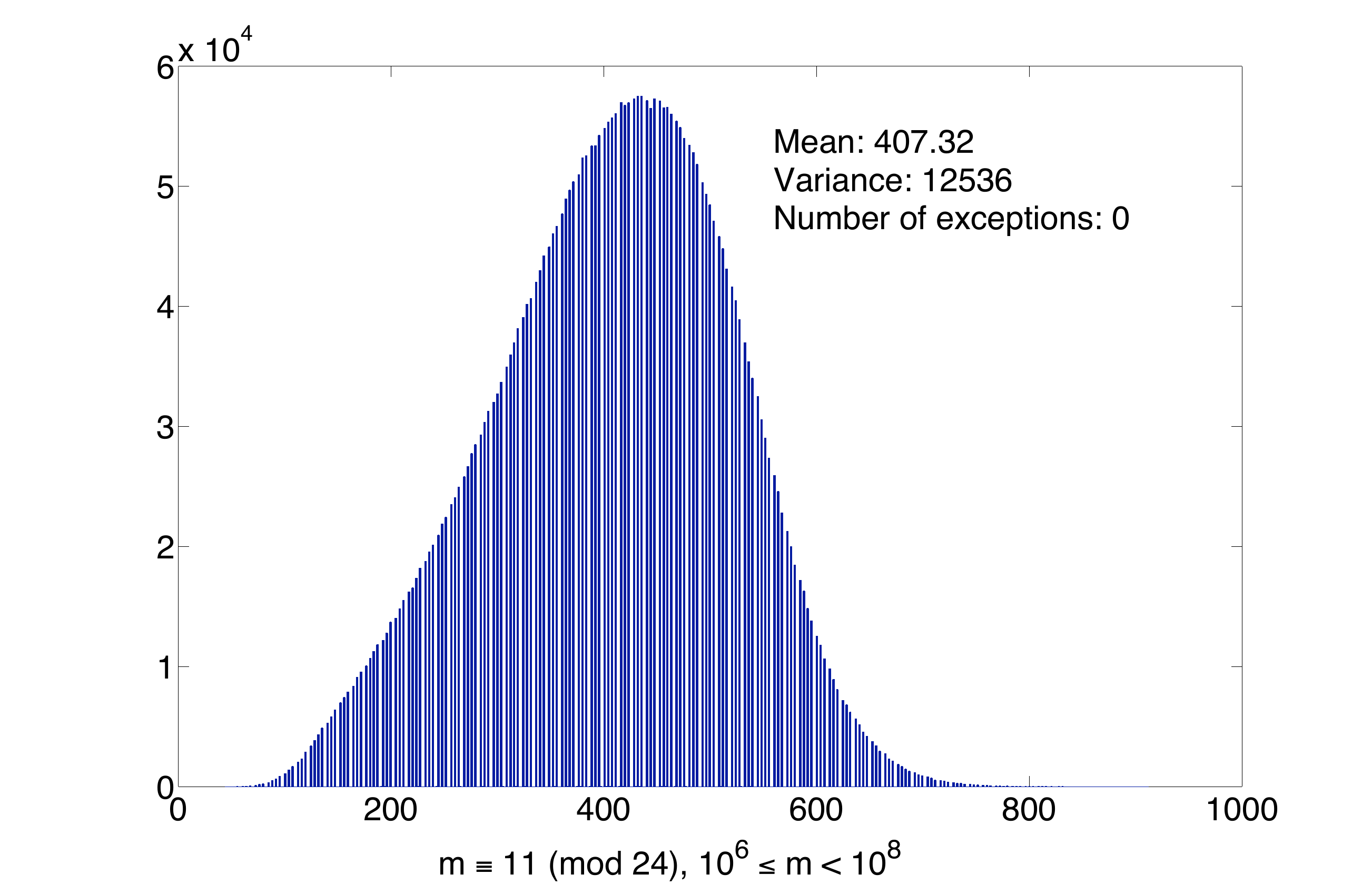}
        \end{center}
	\caption{On the $x$ axis, the number of times a curvature appears.  On the $y$-axis, the number of curvatures with that multiplicity of appearance.  The graph concerns curvatures which are $11 \pmod{24}$ in the packing generated by quadruple $(-1,2,2,3)$, in the range $10^6$ to $10^8$.  Image: \cite[Figure 10]{FuchsSanden}.} 
	\label{fig:histo}
\end{figure}

If we write $\mathcal{K}(N) := \# \{ n < N : n\text{ is a curvature in $\Pcal$} \}$, then the conjecture implies that
\begin{equation}
	\label{eqn:local-global}
		\mathcal{K}(N) = cN + O(1),
	\end{equation}
  where $c = \frac{\mbox{\# admissible curvatures modulo $24$}}{24}$.

Over time, lower bounds on the number of distinct integer curvatures which appear in a packing have gradually improved.
In the original paper of Graham, Lagarias, Mallows, Wilks and Yan, it was shown that 
			$\mathcal{K}(N) \gg \sqrt{N}$.
			Sarnak was able to show that
			$\mathcal{K}(N) \gg \frac{N}{\sqrt{\log N}}$ \cite{Sarnakletter}.  Positive density (meaning a positive proportion of integers) was shown by Bourgain and Fuchs: 
			$\mathcal{K}(N) \gg N$ \cite{BourgainFuchs}.
			This was improved to density one using the Hardy-Littlewood circle method by Bourgain and Kontorovich:
			$\exists \eta>0, \; \mathcal{K}(N) = cN + O(N^{1-\eta})$, where $\eta$ is effectively computable \cite{BourgainKontorovich}.
			Finally, this result was extended to other packings by Fuchs, Stange and Zhang \cite{FuchsStangeZhang}.

However, it was recently discovered \cite{HKRS} that most Apollonian circle packings are subject to some additional more subtle restrictions on their curvatures: the local-to-global conjecture is false!  The source of these new obstructions is quadratic reciprocity.  Recall Definition~\ref{def:legendre} of the Legendre symbol for an integer $a$ and prime $p$ is defined as:
\[
	\left( \frac{a}{p} \right) = 
	\left\{
		\begin{array}{ll}
			1 & a \text{ is a non-zero square modulo } p \\
			-1 & a \text{ is not a square modulo } p \\
			0 & a \text{ is zero modulo } p \\
		\end{array}
		\right.
	\]
	This is multiplicative in the numerator:
	\[
	\left( \frac{ab}{p} \right) = 
	\left( \frac{a}{p} \right)
	\left( \frac{b}{p} \right).
\]
	More generally, the \emph{Jacobi symbol} extends the Legendre symbol multiplicatively for odd positive denominators:
	\[
	\left( \frac{a}{p_1p_2} \right) = 
	\left( \frac{a}{p_1} \right)
	\left( \frac{a}{p_2} \right).
\]
One can extend even further to the \emph{Kronecker symbol} defined for all integers.
Then, \emph{quadratic reciprocity} says that for two odd primes $p$ and $q$, there is a symmetry between the behaviour of $p$ modulo $q$ and $q$ modulo $p$:
\[
	\left( \frac{p}{q} \right)
	=
	(-1)^{\frac{p-1}{2}\cdot \frac{q-1}{2}}
	\left( \frac{q}{p} \right).
\]
There are special rules for $2$ and $-1$.

To demonstrate where these new reciprocity obstructions arise from, we prove a single example. 

\begin{theorem}[{\cite[Theorem 1.9]{HKRS}}]
	The Apollonian circle packing generated by the quadruple $(-3,5,8,8)$ has no square curvatures.
\end{theorem}

\begin{proof}  
	This packing has the property that all curvaures are $0$ or $1 \pmod{4}$.  Let $\mathcal{C}$ be a circle of curvature $n$.  By Theorem~\ref{thm:curv-form}, the circles tangent to $\mathcal{C}$ have curvatures arising as the primitively represented values of a translated quadratic form $f_{\mathcal{C}}(x,y)-n$ of discriminant $-4n^2$.  Therefore the form $f_{\mathcal{C}}(x,y)$ becomes degenerate modulo $n$, being equivalent to $Ax^2$ for some coefficient $A$ after a change of variables.  We see then that the invertible values of $f_\mathcal{C}$ lie in one multiplicative coset of the squares (see \cite[Proposition 4.1]{HKRS}).  Thus we define $\chi_2(\mathcal{C})$ to be the unique non-zero value of the Kronecker symbol $\kron{c}{n}$, where $c$ ranges over the curvatures of circles tangent to $\mathcal{C}$.

	Now suppose two circles $\mathcal{C}_1$ and $\mathcal{C}_2$ are tangent in the packing, having coprime curvatures $a$ and $b$, respectively.  Then, by quadratic reciprocity,
\[\chi_2(\mathcal{C}_1)\chi_2(\mathcal{C}_2)=\dkron{a}{b}\dkron{b}{a}=1.\]
So $\chi_2(\mathcal{C}_1) = \chi_2(\mathcal{C}_2)$.

	By \cite[Corollary 4.7]{HKRS}, any two circles are connected by a path of consecutively pairwise coprime curvatures, and so $\chi_2(\mathcal{C})$ is constant across the entire packing.  It remains to compute this value using one pair of circles, say in the root quadruple:
\[\chi_2(\mathcal{P})=\dkron{8}{5} = -1.\]
If there did exist a circle $\mathcal{C}$ of square curvature in $\mathcal{P}$, it would give $\chi_2(\mathcal{P}) = 1$, a contradiction. 
\end{proof}

\begin{exercise}
	Prove that any two circles in an integral Apollonian circle packing are connected by a path of consecutively coprime curvatures.
\end{exercise}

In general, we define two functions $\chi_2$ and $\chi_4$ on the set of Descartes quadruples, or, if you prefer, on the set pairs $(\mathcal{C}, \mathcal{P})$ of a circle in a packing.  These should be thought of in the vein of characters or Legendre symbols; in particular, they take values $\chi_2 \in \{\pm 1 \}$ and $\chi_4 \in \{ \pm 1, \pm i \}$.  To define them in general is a little complicated, but we can explain $\chi_2(\mathcal{C})$ fairly simply in the case of a packing of type $(6,*)$ as the Legendre symbol $\left( \frac{a}{b} \right)$ where $b$ is the curvature of $\mathcal{C}$, and $a$ is a coprime curvature tangent to $\mathcal{C}$, and $a$ and $b$ are both coprime to $6$.  The symbol $\chi_4$ is only relevant, hence only defined, for types $(6,1)$ and $(6,17)$.

Then, it is a consequence of quadratic reciprocity that the values of $\chi_2$ are well-defined for a packing $\mathcal{P}$ independent of the circle $\mathcal{C}$.  Hence we can write $\chi_2(\mathcal{P})$.  Similarly, $\chi_4(\mathcal{P})$ is well-defined because of quartic recprocity.

%Then, the fundamental observation is that if $\chi_2(\mathcal{C}) = -1$, then the circle $\mathcal{C}$ \emph{cannot be tangent to a circle of square curvature modulo $b$}.

\begin{theorem}[{Haag-Kertzer-Rickards-Stange \cite{HKRS}}]
	Let $\mathcal{P}$ be a primitive integral Apollonian circle packing.  There is an explicit list of obstructions (families of missing curvatures) of the form $\{ ux^2 : x \in \ZZ \}$ for $u \mid 6$ or $\{ ux^4 : x \in \ZZ\}$ for $u \mid 36$ that are missed by the list of curvatures in $\mathcal{P}$.  The list is determined entirely by the set of admissible curvatures, and by the value(s) $\chi_2(\mathcal{P})$ and $\chi_4(\mathcal{P})$, if defined.
\end{theorem}

The full list is given here, where the type is now extended to be of the form $(n,k,\chi_2)$ or $(n,k,\chi_2,\chi_4)$:
\vspace{0.2em}
		\begin{center}
		{\upshape
			\begin{tabular}{|l|c|c|c|c|} 
				\hline
				type              & quadratic obstructions  & quartic obstructions    \\ \hline
				$(6, 1, 1, 1)$    &                         &                         \\ \hline
				$(6, 1, 1, -1)$   &                         & $n^4, 4n^4, 9n^4, 36n^4$\\ \hline
				$(6, 1, -1)$      & $n^2, 2n^2, 3n^2, 6n^2$ &                         \\ \hline
				$(6, 5, 1)$       & $2n^2, 3n^2$            &                         \\ \hline
				$(6, 5, -1)$      & $n^2, 6n^2$             &                         \\ \hline
				$(6, 13, 1)$      & $2n^2, 6n^2$            &                         \\ \hline
				$(6, 13, -1)$     & $n^2, 3n^2$             &                         \\ \hline
				$(6, 17, 1, 1)$   & $3n^2, 6n^2$            & $9n^4, 36n^4$           \\ \hline
				$(6, 17, 1, -1)$  & $3n^2, 6n^2$            & $n^4, 4n^4$             \\ \hline
				$(6, 17, -1)$     & $n^2, 2n^2$             &                         \\ \hline
				$(8, 7, 1)$       & $3n^2, 6n^2$            &                         \\ \hline
				$(8, 7, -1)$      & $2n^2$                  &                         \\ \hline
				$(8, 11, 1)$      &                         &                         \\ \hline
				$(8, 11, -1)$     & $2n^2, 3n^2, 6n^2$      &                         \\ \hline
			\end{tabular}
		}
		\end{center}
		\vspace{0.2em}

These families of powers are called \emph{reciprocity obstructions} and more specifically \emph{quadratic obstructions} and \emph{quartic obstructions}.
Certain packings have no reciprocity obstructions at all, but many (most, in a suitable sense) do.

We use the term \emph{missing} for the curvatures which do not appear in a packing but are allowed by the congruence obstructions.  Curvatures which are allowed by the known linear (congruence), quadratic, and quartic obstructions but are \emph{still} nevertheless missing, are called \emph{sporadic}.

\begin{conjecture}[Haag-Kertzer-Rickards-Stange \cite{HKRS}]
        \label{conj:main}
	Let $\Pcal$ be a primitive integral Apollonian circle packing.  Then the sporadic set $S(\Pcal)$ is finite.
\end{conjecture}
This actually says that, instead of \eqref{eqn:local-global}, we are in most cases asserting at best that 
\begin{equation}
	\label{eqn:newconj}
		\mathcal{K}(N) = cN + O(\sqrt{N}).
	\end{equation}

  Haag, Kertzer, Rickards and Stange collected data on the missing curvatures up to bounds between $10^{10}$ and $10^{12}$ in a few dozen small packings, and the data supports the conjecture that the sporadic set is finite, as the set peters out in that range and appears to end.

For a nice exposition on Apollonian circle packings and their integer curvatures (before the newest obstructions were found, however), see \cite{FuchsBulletin}.  The rest of this section is devoted to some of the key tools that appear in proofs concerning integral packings, particularly lower bounds on the number of distinct curvatures.

\subsection{Integral quadratic forms}

We have seen in Theorem~\ref{thm:curv-form}, that the circles tangent to a fixed `mother' circle represent the values of a translated quadratic form.  It is furthermore the case that when $\mathcal{P}$ is a primitive integral packing, then the form is a primitive integral positive semi-definite form, and the multiplicity with which a curvature $k \in \ZZ$ appears is exactly the number of primitive solutions $(x,y)$ to $f_\mathcal{C}(x,y) - c = k$.  Since quadratic forms are\footnote{arguably} well-understood, many of the analytic lower bound results so far mentioned involve collecting ensembles of such forms and restricting their overlap to grow large collections of curvatures.

In the section on Schmidt arrangements that follows, we will see that the images of the strip packing $(0,0,2,2)$ under $\PSL_2(\ZZ[i])$ include all primitive integral packings (all scaled by two); see Theorem~\ref{thm:every-app}.  We will see that up to similarity, they are in bijection with ideal classes of orders of the Gaussian integers.

\subsection{Expander graphs}

Consider the Cayley graphs of the Apollonian group modulo $p^m$.  To see how these are useful, we consider the spectral theory of graphs, which is to say, studying the eigenvalues of the adjacency matrix.  The graph we are interested in, denoted $\mathcal{A}_{p^m}$, is the Cayley graph of $\mathcal{A}^{geo}/p^m$ (coefficients of the matrices taken modulo $p^m$) using the images of the standard swaps as generators.  This is still $4$-regular, just as for the Cayley graph of $\mathcal{A}$.  These Cayley graphs, taken as a family, are an \emph{expander family}, meaning that they are well-connected in a precise sense as $p^m$ grows.

We will develop the basics of the spectral theory slightly more generally.  Let $\mathcal{G}$ be a $d$-regular graph with vertex set $V$ of size $n$ and edge set $E$.  (The spectral theory of non-regular graphs is significantly more complex in various ways.)  Let $A$ be the $n \times n$ adjacency matrix of $\mathcal{G}$, whose $ij$-th entry is $1$ when the $i$-th vertex is connected to the $j$-th vertex, and $0$ otherwise.  The normalized adjacency matrix $\frac{1}{d} A$ can be thought of as an operator controlling the flow of mass between vertices.  If $\mathbf{x}$ is a vector whose entries are indexed by the vertices of $\mathcal{G}$, then $\frac{1}{d}A\mathbf{x}$ has entries
\[
	\frac{1}{d} \sum_{w \sim v} x_w,
\]
where $v \sim w$ denotes adjacency, so the sum is over the neighbours.  Imagine the vector $\mathbf{x}$ denotes a distribution of mass amongst the vertices of $\mathcal{G}$.  Then $\frac{1}{d}A\mathbf{x}$ is the mass distribution after each vertex `gives away' its mass uniformly to its neighbours (that is, it sends $1/d$ of its mass to each neighbour), and, consequently, receives $1/d$ of the mass of each of its neighbours.  This is a type of Markov chain.

Under this perspective, an eigenvector of this matrix is a distribution of mass which is scaled under such a flow.  One such eigenvector is the uniform distribution (the same mass at all vertices), which has eigenvalue $1$.  In fact, the eigenvalues $\lambda_i$ of this matrix are real, and satisfy 
\[
	1 = \lambda_0 \ge \lambda_1 \ge \cdots \ge -1.
\]
That they are real is a property of symmetric matrices.  Moreover, since the adjacency matrix is real, symmetric, non-negative and irreducible, there is an orthonormal basis of eigenvectors.

If $\mathcal{G}$ is connected, then $\lambda_0 > \lambda_1$ (see Exercise~\ref{exercise:components}).  The size of this \emph{spectral gap} measures the connectedness of the graph in some sense.  One intuition for this is to consider the convergence of the Markov chain toward the uniform distribution.  Let $\mathbf{v}_0, \ldots, \mathbf{v}_{n-1}$ be an orthonormal basis of eigenvectors such that $\frac{1}{d}A \mathbf{v}_i = \lambda_i \mathbf{v}_i$.  Let $\mathbf{w}$ be any mass distribution on the graph.  Then $\mathbf{w} = \sum \alpha_i \mathbf{v}_i$.  Then we have
\[
	\left( \frac{1}{d} A \right)^k \mathbf{w} = \alpha_0 \mathbf{v}_0 + \sum_{i > 0} \lambda_i^k \alpha_i \mathbf{v}_i.
\]
From this, using that the basis of eigenvectors is orthonormal,
\[
	\left|\left| \left(\frac{1}{d}A\right)^k \mathbf{w} - \alpha_0 \mathbf{v}_0 \right|\right|_2^2 = \sum_{i > 0} |\lambda_i|^{2k} |\alpha_i|^2 \le |\lambda_1|^{2k} || \mathbf{w} ||^2.
\]
Thus we see that the rate of convergence to a uniform distribution is controlled by the spectral gap (i.e., $|\lambda_1|$).

\begin{definition}
	\label{defn:expander}
	A family of graphs, $\mathcal{G}_i$ for $i \ge 1$, is called an \emph{expander family of degree $d$} if the following hold:
	\begin{enumerate}
		\item the $\mathcal{G}_i$ are finite $d$-regular connected graphs with $|\mathcal{G}_i| \rightarrow \infty$;
		\item for each $i$, let $\lambda_i$ be the largest eigenvalue in absolute value besides $\pm 1$ of the normalized adjacency matrix of $\mathcal{G}_i$; then
			\[
				\epsilon := \limsup_{i \rightarrow \infty} \lambda_i < 1.
			\]
	\end{enumerate}
\end{definition}

If the graph is `easily cut' in the sense that removing a small number of edges can disconnect it, then these edges form a bottleneck to rapid convergence.  Another measure of the same `connectedness' is given in this language.
For any subset $S \subseteq V$, write $\partial S \subseteq E$ for the `boundary' of $S$, i.e. the edges connecting a vertex of $S$ to a vertex of the complement.  The \emph{Cheeger constant} of $\mathcal{G}$ is
\[
	h(\mathcal{G}) = \min_{S \subseteq V, |S| \le |V|/2} \frac{ |\partial S| }{ |S| }.
\]
This measures how easy it is to disconnect the graph by removing edges.  Then we can replace condition (2) above with the existence of an $\epsilon < 0$ such that $h(\mathcal{G}_i) > \epsilon$ for all $i$.

It is possible to prove a spectral gap for the Apollonian Cayley graphs $\mathcal{A}_{p^m} := \mathcal{A}^{geo}/p^m$ in a combinatorial way, by showing that every vertex can be reached in a bounded number of steps, where the bound is independent of the growing parameter $p^m$.  The existence of short paths to all points in the graph implies `good mixing' and `connectedness'; this is the `combinatorial spectral gap' used in \cite[Section 8]{FuchsStangeZhang}.  

In the next section we will prove strong approximation, but a byproduct of this proof is the spectral gap.

\begin{theorem}\label{thm:expander}
	The Cayley graphs $\mathcal{A}^{geo}/p^m$ form an expander family.
\end{theorem}

We will discuss the use of this tool at the end of the next section. 
A good reference for expander graphs for those interested in Cayley graphs is \cite{Kowalski}.

\begin{exercise}
	\label{exercise:components} Let $G$ be a $d$-regular graph with $k$ connected components.  Show that $\lambda_0 = \lambda_1 = \cdots = \lambda_{k-1}$ by finding independent eigenvectors.  Conversely, show that when $k=1$, $\lambda_0 > \lambda_1$. 
\end{exercise}

\begin{exercise}
	Let $G$ be a connected $d$-regular graph.  Show that the eigenvectors for $\lambda_i < 1$ have the property that the sum of their entries is $0$.  This shows that they are orthogonal to the eigenvector for $\lambda_0 = 1$.
\end{exercise}

\subsection{Strong approximation}

An algebraic group $G$ has \emph{strong approximation} if the maps $G(\ZZ) \mapsto G(\ZZ/p^m\ZZ)$ are surjective.  For example, $\GL_2$ fails this property since matrices in $\GL_2(\ZZ/p\ZZ)$ with invertible determinants other than $\pm 1$ are not in the image.  However, it does hold for $\SL_2$ \cite{DavidoffSarnakValette}.  %add stuff about dense open image 

In particular, the reduction map modulo $\mathfrak{a}$ for any ideal $\mathfrak{a}$ of $\ZZ[i]$, $\SL_2(\ZZ[i]) \rightarrow \SL_2(\ZZ[i]/\mathfrak{a})$ is always surjective.
One way to measure of the `size' of a subgroup like $\mathcal{A}^{geo}$ is to ask whether the reduction maps $\mathcal{A}^{geo} \rightarrow \SL_2(\ZZ/p^m\ZZ)$ are surjective.  The Apollonian group \emph{almost} has this property:  it holds for sufficiently large prime powers.  Thus we say that $\mathcal{A}^{geo}$ itself has strong approximation.

For primes besides $2$ and $3$, the proof is by construction, using the fact that $z \mapsto z + i$ is in the Apollonian group.

	\begin{theorem}
		The Apollonian group $\mathcal{A}^{geo}$ satisfies $\mathcal{A}^{geo}/p^m \cong \SL_2(\ZZ/p^m\ZZ)$ for all $p \ge 5$. 
	\end{theorem}

	The following proof is an adaptation of that of Varj\'u \cite[Appendix]{BourgainKontorovich} and \cite{FuchsStangeZhang}.
	\begin{proof}
		Let $p \ge 5$ be prime and $m \ge 1$ be an integer.  
		Then there exists a pair $x,y \in \ZZ$ such that $x^2 + y^2 \equiv 1 \pmod{p^m}$ and $(xy, p) = 1$ \cite[Exercise 13(v)]{Cassels}. % check this 
		Consider reduction modulo $p^m$ on $\ZZ[i]$ (the following will work whether $p$ is split or inert; we write $i$ for the image of $i$ under reduction).
		The image matrix $T_0 := \begin{pmatrix} x & y \\ -y & x \end{pmatrix}$ lies in $\SL_2(\ZZ/p^m\ZZ)$ by construction.
		Note that $T_0$ has fixed points $\pm i$.
		Then there exists a lift $T_1 := \begin{pmatrix} x_0 & -y_0 \\ y_0 & x_0 \end{pmatrix} \in \SL_2(\ZZ)$ of $T_0$ by the strong approximation of $\SL_2$.

		Let $T := \begin{pmatrix} 1 & i \\ 0 & 1 \end{pmatrix} \in \mathcal{A}$.  
		Then $T T_1 T^{-1}$ has fixed points $T(\pm i) = \{0, 2i\}$, the first of which we can conjugate to $\infty$ using $S \in \SL_2(\ZZ)$.  Call the result $T_2 := S T T_1 T^{-1} S^{-1} \in \mathcal{A}$, which fixes $\infty$ modulo $p^m$.  That is, it must have the form
		\[
			T_2 = \begin{pmatrix} a_0 & b \\ 0 & a_1 \end{pmatrix}.
		\]
		Here, $a_0$ and $a_1$ are the eigenvalues of $T_2$, and hence also of $T_1$, which are $x \pm yi$.  In particular, $a_0^2 \notin \ZZ/p^m\ZZ$ (observe that $y \not\equiv 0 \pmod{p}$ by construction, so the reduction of $x + iy$ modulo $p^m$ is in $\ZZ[i]/p^m\ZZ[i] \setminus \ZZ/p^m\ZZ$), and it is invertible (also by construction).
		Now let
		\[
			T_{3,n} := T_1 \begin{pmatrix} 1 & n \\ 0 & 1 \end{pmatrix} T_1^{-1} \equiv \begin{pmatrix} 1 & na_0^2 \\ 0 & 1 \end{pmatrix}
		\]
		where the upper right corner of $T_{3,1}$ is invertible, but not in $\ZZ/p^m \ZZ$.  Hence $a_0^2$ and $1$ generate $\ZZ[i]/p^m\ZZ[i]$.  This implies that all upper triangular matrices are in $\mathcal{A}$.
		Similarly we obtain all lower triangular matrices.  By combining these, we have everything except those things whose lower left entry is divisible by $p$.  That is, we have more than half of $\SL_2(\ZZ/p^m\ZZ)$.  Therefore we must generate it all.
	\end{proof}

	Although the theorem fails for $(p,m) \in \{ (2,1), (2,2), (2,3), (3,1) \}$, for higher powers of $2$ and $3$, we do recover predictable behaviour, in the following sense.  Although $\mathcal{A}_3$ is not all of $\SL_2(\ZZ/3\ZZ)$, when lifting from $\mathcal{A}_3$ to $\mathcal{A}_{3^2}$, we do obtain `all' of the valid lifts, meaning that although of course we cannot recover $\SL_2(\ZZ/3^2\ZZ)$, we do not `lose even more.'  More precisely, for any $m > m_p$ (where $m_2 = 2$ and $m_3 = 1$), if $M \in \SL_2(\ZZ/p^m\ZZ)$ has a reduction to $\mathcal{A}_{p^{m_p}}$, then it has a reduction to $\mathcal{A}_{p^m}$.

	In fact, the proof above for $p \ge 5$ shows that all of $\SL_2(\ZZ/p^m\ZZ)$ is generated by words of a bounded finite length independent of $p$ and $m$; a similar result for $2$ and $3$ combines with this to give a \emph{spectral gap}, showing that the Cayley graphs form an expander family (Theorem~\ref{thm:expander}).

Strong approximation and the spectral gap turn out to be an important tool in the proofs that many curvatures appear.  The rough idea is that for an expander graph, there is rapid mixing, so that as $n$ and $m$ grow, all curvatures modulo $n$ and modulo $m$ will be appearing regularly, so that all residues modulo $nm$ are also likely to occur regularly.  Analytic methods allow one to extrapolate that most integers will eventually occur.  One might think of this as a sort of explicit Sunzi's Theorem for Apollonian packings.

\begin{exercise}
 \label{exercise:reduction} Prove that the natural reduction map $\PSL_2(\ZZ) \rightarrow \PSL_2(\ZZ/n\ZZ)$ is surjective.  Show that this is a group homomorphism and find the kernel.
%	\item \label{exercise:triply} Given points $z_1, z_2, z_3$ and $w_1, w_2, w_3$, find a M\"obius transformation in $\PSL_2(\CC)$ taking $z_i \mapsto w_i$.
\end{exercise}

\subsection{Orbits of thin groups more generally}

These can be viewed as statements about orbits of the Apollonian group.  To be precise, in this perspective, explored in more depth in \cite{ApolloniusZaremba}, the curvatures of a fixed packing form a set
	$\{ \pi_i(\mathbf{v}) : \mathbf{v} \in \mathbf{v}_0 \mathcal{A}, 1 \le i \le 4 \}$,
	where $\pi_i$ is projection on the $i$-th coordinate, $\mathcal{A}$ is the Apollonian group, and $\mathbf{v}_0$ is a row vector of curvatures of some fixed Descartes quadruple.  We consider the Apollonian group as a subgroup $\mathcal{A}$ of $O_Q(\ZZ)$.

	Another famous conjecture is part of the same general type.

	\begin{conjecture}[{Zaremba's Conjecture, \cite{Zaremba72}}]
		There exists a positive constant $Z$ such that 
		every natural number is the denominator of some rational number (in reduced form) whose continued fraction partial quotients are $\le Z$.
	\end{conjecture}

For example, the continued fraction expansions of all rationals with denominator $7$ are:
\[
	\frac{1}{7} = [0; 7], \quad
	\frac{2}{7} = [0; 3,2], \quad
	\frac{3}{7} = [0; 2,3], \quad
	\frac{4}{7} = [0; 1,1,3], \quad
	\frac{5}{7} = [0; 1,2,2], \quad
	\frac{6}{7} = [0; 1,6], \quad
\]
A reasonable guess with current data is that Zaremba's conjecture holds for $Z = 5$.  In the example above, things are `so far so good' for $Z=5$ because at least one of the expansions (in fact, 4 of them) involves only convergents $\le 5$.  We know that denominators $6$, $54$, and $150$ %and $1014$ 
fail for $Z=4$, but we do not know of any further failures.  Niederreiter \cite{Niederreiter78} conjectured that even for $Z=3$ there are only finitely many exceptions; Hensley \cite{Hensley} conjectured this for $Z=2$.  Of course, it certainly fails for $Z=1$, whose continued fraction expansions have only Fibonacci sequence denominators.

One might consider the Cantor-like set $C_Z$ of real numbers whose continued fraction expansions contain only convergents $a_i \le Z$.  Or even more generally, a cantor set $C_S$ for a set $S$ of allowable convergents.  Hensley conjectured that a Zaremba-like statement will hold for convergents in $S$ if and only if the Hausdorff dimension of $C_Z$ exceeds $1/2$ \cite{Hensley}.  However, Bourgain and Kontorovich found a counterexample of $S=\{2,4,6,8,10\}$, which has a congruence-like obstruction; denominators of $3 \pmod{4}$ cannot appear \cite{BourgainKontorovichZaremba}.

Zaremba's conjecture and its variants can be phrased as a thin orbit question.  As we saw before, we can generate continued fraction convergents as the columns of elements $M \in \SL_2^+(\ZZ)$.  A slight reformution is that every continued fraction convergent can be found as a column of a matrix of the form
\[
	\begin{pmatrix}
		0 & 1 \\ 1 & a_0
	\end{pmatrix}
	\begin{pmatrix}
		0 & 1 \\ 1 & a_1
	\end{pmatrix}
	\begin{pmatrix}
		0 & 1 \\ 1 & a_2
	\end{pmatrix}
	\cdots
	\begin{pmatrix}
		0 & 1 \\ 1 & a_n
	\end{pmatrix}
\]
where the $a_n$ are the partial quotients.  To restrict the partial quotients, we simply generate a semi-group (i.e., no inverses) by the matrices
$	\begin{pmatrix}
		0 & 1 \\ 1 & a
	\end{pmatrix}
	$
where $a$ is in our set $S$ of allowable convergents.
In recent work, Rickards and Stange find reciprocity obstructions in this context \cite{RickardsStange}.

%	\item (Uses previous two exercises.)  Gah I need to do this with Laplacian,  The behaviour of the eigenvalues is closely related to certain optimization problems of the graph.  Let $A$ be the adjacency matrix of a $d$-regular graph.
%		\begin{enumerate}
%			\item Show that
%\[
%	\lambda_0 = \min_{\mathbf{x} \neq 0} \frac{ \mathbf{x}^t A \mathbf{x} }{\mathbf{x}^t \mathbf{x}},
%\]
%and furthermore, the optimal solution (the $x_0$ that realizes this minimum) is the eigenvector associated to the eigenvalue $\lambda_0$.  
%\item Similarly, show that for the adjacency matrix of a $d$-regular connected graph,
%\[
%	\lambda_1= \min_{\mathbf{x} \neq 0, \sum_i x_i = 0} \frac{ \mathbf{x}^t A \mathbf{x} }{\mathbf{x}^t \mathbf{x}},
%\]
%and furthermore, the optimal solution is the eigenvector associated to $\lambda_1$.
%(These are sometimes called \emph{Rayleigh quotients}.)  
%\item Can you interpet these as statements about the graph?
%\end{enumerate}

\section{Schmidt arrangements}

\newcommand{\SK}{\mathcal{S}_K}

\subsection{$\PSL_2(\mathcal{O}_K)-$orbit}

As we saw before, Schmidt defined a subdivision of the complex plane that involved Apollonian circle packings.  See Figure~\ref{fig:gauss-schmidt}.  The easiest way to generate this image is to take the image of $\widehat{\RR}$ under $\PSL_2(\ZZ[i])$.  Suitable references for this section are \cite{StangeVis, StangeBianchi, MartinStudy}.

%	Which M\"obius images are primitive and integral?  It turns out one can combine all the primitive integral packings in a beautiful way, as an \emph{Apollonian superpacking} or \emph{Schmidt arrangement}, shown in Figure~\ref{fig:gauss-schmidt}.  This was first defined by Asmus Schmidt, quite apart from any consideration of Apollonian superpackings, and then re-discovered by Graham-Lagarias-Mallows-Wilks-Yan.  It is most simply described as the orbit of $\widehat{\RR}$ under all of $\PSL_2(\ZZ[i])$, and it is, in a sense made precise below, the union of all primitive integral Apollonian circle packings.  I find this one of the most compelling arguments that the study of Apollonian packings is natural or fundamental.  

	\begin{figure}
                \includegraphics[height=4.0in]{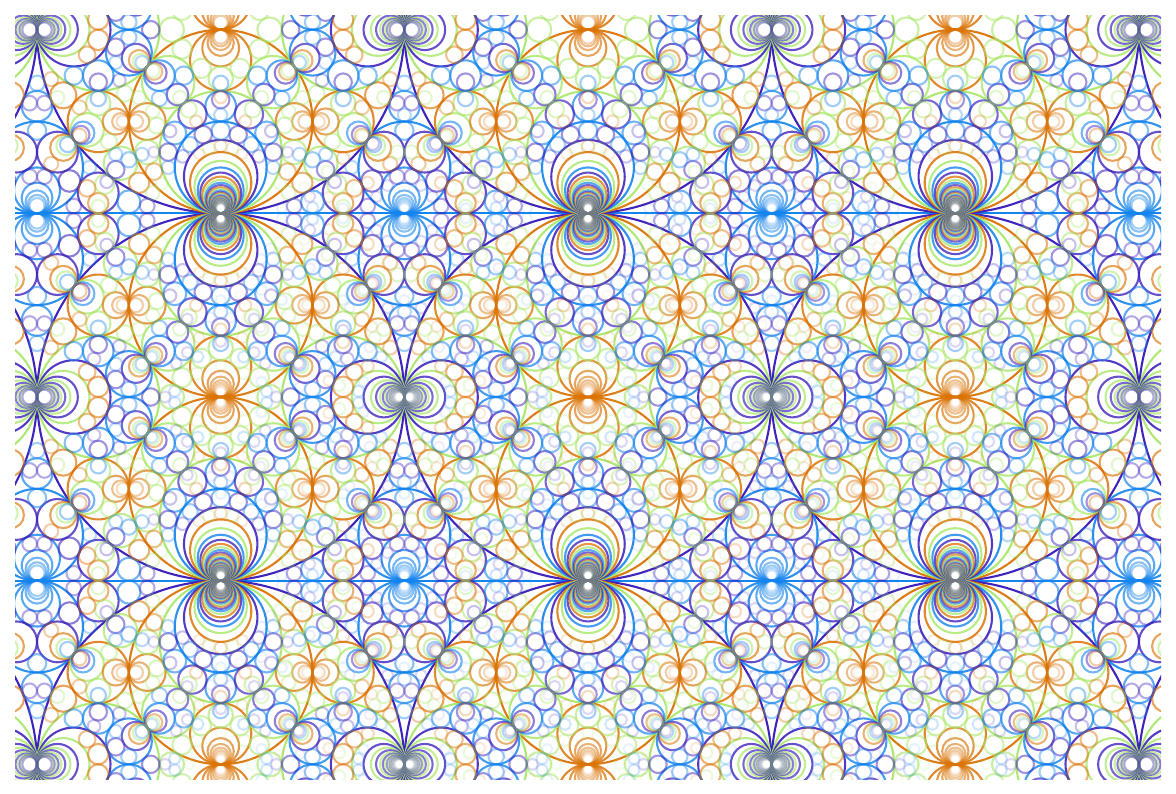}
		\caption{The Gaussian Schmidt arrangement.}
		\label{fig:gauss-schmidt}
	\end{figure}

	The definition of the Schmidt arrangement can be given more generally in terms of an imaginary quadratic field, so that Figure~\ref{fig:gauss-schmidt} is for the Gaussian field $\QQ(i)$.  Another example is given in Figure~\ref{fig:Kschmidt}.
	\begin{definition}
		The \emph{Schmidt arrangement} $\SK$ for an imaginary quadratic ring $\mathcal{O}_K$ is the image of $\widehat{\RR}$ under $\PSL_2(\mathcal{O}_K)$.
	\end{definition}
	We will be mainly interested in the case $\mathcal{O}_K = \ZZ[i]$, because of its connection to the Apollonian packing.

	\begin{figure}
		\includegraphics[height=4.0in]{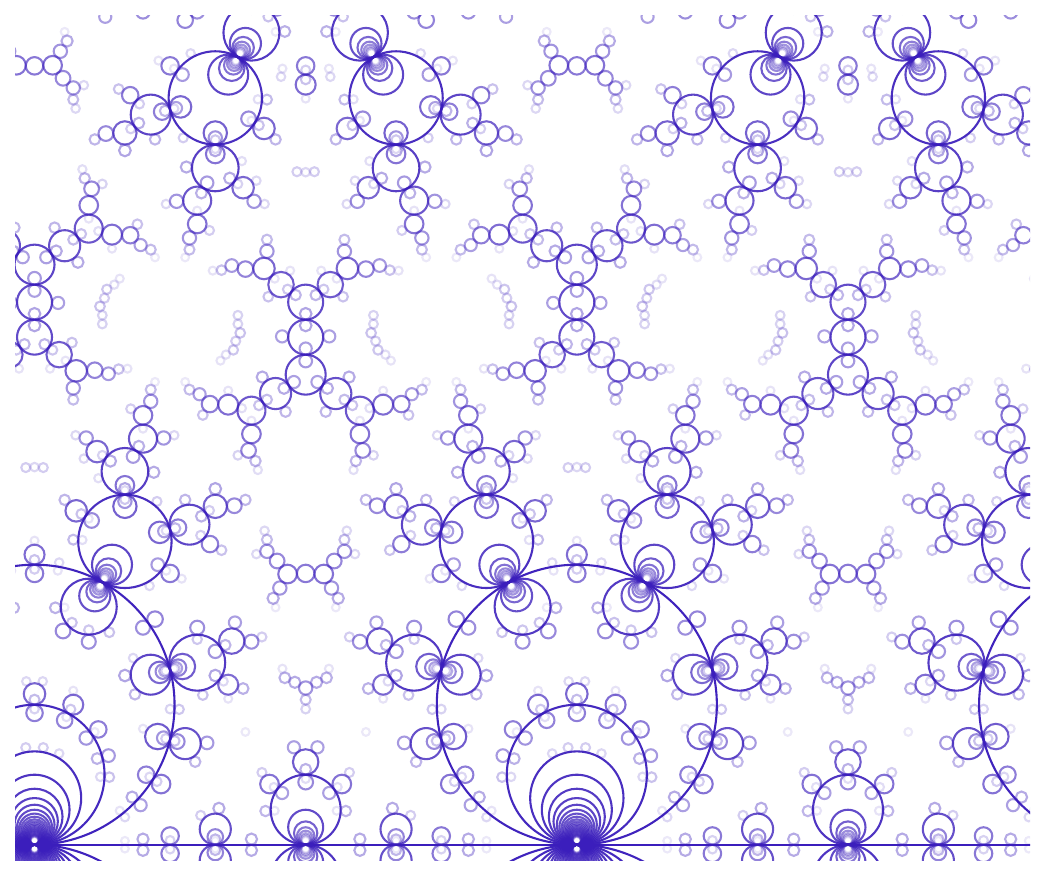}
		\caption{The Schmidt arrangement for $\QQ(\sqrt{-15})$.}
		\label{fig:Kschmidt}
	\end{figure}

	We begin with some basic properties, all of which are a consequence of Proposition~\ref{prop:mobR} (see also the proof of Theorem~\ref{thm:curv-form}).

	\begin{proposition}
		\begin{enumerate}
			\item The Schmidt arrangement has symmetry by translation by $\ZZ[i]$ and by rotation of $\pi$ about the origin.
			\item The curvatures of the circles in the Schmidt arrangement lie in $2\ZZ$.
			\item The circles are either tangent or disjoint.
			\item The points of tangency are Gaussian rationals.
		\end{enumerate}
	\end{proposition}
	In particular, up to a scaling factor of $2$, we'll think of the circles as having integral curvature.  In fact, we'll call half the curvature the \emph{reduced curvature} for convenience; this lies in $\ZZ$.

	The following fact was known to Graham, Lagarias, Mallows, Wilks and Yan in the context of the Apollonian super-packing.

	\begin{theorem}[{\cite{GLMWY-geometry,StangeBianchi}}]
		\label{thm:every-app}
		Every single primitive integral Apollonian circle packing appears in the Schmidt arrangement of $\QQ(i)$ exactly once up to translation by $\ZZ[i]$ and rotation about the origin by $\pi$.
	\end{theorem}

	We saw in Theorem~\ref{thm:curv-form} that every circle in a packing has a natural quadratic form associated to it.  We also saw in Section~\ref{sec:quadratic-forms} that quadratic forms are in bijection with lattices.  In fact, every circle has a lattice associated to it.

	\begin{theorem}
		\label{thm:lattice-of-circ}
		Let $\mathcal{C}$ be a circle in the Schmidt arrangement for $\QQ(i)$, given as an image of $\widehat{\RR}$ by $\begin{pmatrix} \alpha & \beta \\ \gamma & \delta \end{pmatrix}$.  Let $\Lambda = \gamma \ZZ + \delta \ZZ$.  Then the denominators of the tangency points touching $\mathcal{C}$ are exactly the elements of $\Lambda$, and at tangency point $\sigma/\rho$, the circles tangent have curvatures $\Im(\gamma \overline{\delta}) + 2 k N(\rho)$, $k \in \ZZ$. 

	\end{theorem}

	\subsection{Lattice in the space of circles}

	As an orbit of a circle under M\"obius transformations, we can ask what the Schmidt arrangement is as a subset of the space of circles.  It is actually a very nicely described subset, as given by Daniel Martin (who did this for general Schmidt arrangements).  This is very useful for drawing images.

	\begin{proposition}[{\cite[Definition 3.1 and Theorem 3.11]{MartinStudy}}]
			The circles of the Schmidt arrangement are exactly those circles having curvature $2p'$ and curvature center $2t' + (2s'+1)i$ whenever $p' \mid r'^2 + s'^2 + s'$.
	\end{proposition}

	For example, for $p'=1$, the condition is always satisfied and we have centers $t' + (s' + 1/2)i$ for all $s',t' \in \ZZ$.  

	\begin{proof}
	We will describe the Gaussian Schmidt arrangement as the intersection of the one-sheeted hyperboloid with a lattice in the space of circles.
		The lattice will be 
		\[
			\Lambda := \{ (p,q,r,s) \in \ZZ^4 : (p,q,r,s) \equiv (0,0,0,1) \pmod{2} \} \subset \RR^4,
		\]
		coordinates representing curvature, co-curvature, real and imaginary parts of curvature-centre in the space of circles, as usual.  
		By Exercise~\ref{ex:mobR}, circles of the Schmidt arrangement lie in $\Lambda$.  

		The image $G$ of $\PSL_2(\ZZ[i])$ in $O_Q(\RR)$ (under the exceptional isomorphism \eqref{eqn:exceptional}) lies in $O_Q(\ZZ)$, and satisfies $G \Lambda \subseteq \Lambda$.
		The extended real line $\widehat{\RR}$ corresponds to $(0,0,0,-1) \in \Lambda$, which lies on the one-sheeted hyperboloid $r^2 + s^2 - pq = 1$.  %and extending by the action of $\PSL_2(\ZZ[i])$, that is, by $G$.  
		Thus, since $G$ preserves length, the entire Schmidt arrangement lies in the one-sheeted hyperboloid $r^2 + s^2 - pq = 1$.  
	%	In particular, we have also shown that $G \Lambda \subseteq \Lambda$.
%	The Schmidt arrangement is the super-Apollonian group orbit of the four circles given as columns in
%	\[
%		\begin{pmatrix}
%			0 & 0 & 2 & 2 \\
%			0 & 2 & 0 & 2 \\
%			0 & 0 & 0 & 2 \\
%			-1 & 1 & 1 & 1
%		\end{pmatrix}
%	\]

		Conversely, assume we have a circle $\mathcal{C}$ in the lattice and on the hyperboloid.  Choose a Gaussian rational point $\alpha/\beta \in \mathcal{C}$.  There is an element of $\PSL_2(\ZZ[i])$ mapping $\alpha/\beta$ to $0$ because there is a solution to $\alpha x + \beta y = 1$ (by coprimality of numerator and denominator).  So without loss of generality, we can assume our circle touches $0$, and therefore has co-curvature $q = 0$ (exercise).  Suppose it has curvature $p$ and curvature-center $r + si$.  It will still lie in the lattice, since $G \Lambda \subseteq \Lambda$.  Consider the matrix 
	\[
	M :=	\begin{pmatrix}
			0 & -s + ir \\
			1 & ip/2 
		\end{pmatrix}.
	\]
		Its determinant has absolute value $|s - ir| = r^2 + s^2 - p\cdot 0 = 1$, so it has determinant in $\{ \pm 1, \pm i \}$.  As $s$ is odd, the determinant is $\pm 1$, and so $M \in \PSL_2(\ZZ[i])$.  Since $\widehat{C} = M \cdot \widehat{\RR}$, it is in the Schmidt arrangement.

	Finally, we simply work out the condition of intersecting the lattice with the one-sheeted hyperboloid.
	Writing $p = 2p'$, $q = 2q'$, $r = 2r'$, $s=2s'+1$, the condition that $q' \in \ZZ$ and $pq - r^2 - s^2 +1 = 0$ becomes
	\[
		p' \mid r'^2 + s'^2 + s'.
	\]
	\end{proof}

\subsection{Visual structure of imaginary quadratic fields}

The Schmidt arrangement $\SK$ gives visual form to the arithmetic of $K$.  $K$-Bianchi circles intersect only at $K$-points and only at `unit angles', i.e. angles $\theta$ such that $e^{i\theta}$ lies in the unit group of $\OK$.  Their curvatures lie in $\sqrt{-\Delta}\ZZ$.  Furthermore, the circles themselves are in bijection with certain ideal classes:

\begin{theorem}[{\cite[Theorem 1.4]{StangeVis}}]
        \label{thm:visoneA}
$K$-Bianchi circles of curvature $f$, modulo translation into the fundamental region of $\OK$, and rotation by unit angles, are in bijection with the invertible ideal classes of the order of conductor $f$ which extend to the trivial class in $\OK$.
\end{theorem}

This bijection is very explicit:  if $M \cdot \widehat{\RR}$ is a $K$-Bianchi circle of curvature $f$, where $M = \tiny \begin{pmatrix} \alpha & \beta \\ \gamma & \delta \end{pmatrix}$\normalsize, then $\gamma \ZZ + \delta \ZZ$ is an ideal with conductor $f$.  This is the lattice we saw as the lattice of denominators in the last section.

Furthemore the connectedness of $\SK$ is easily characterised:

\begin{theorem}[{\cite[Theorem 1.5]{StangeVis}}]
        \label{thm:visoneB}
        $\SK$ is connected if and only if $\OK$ is Euclidean.
\end{theorem}

	\begin{figure}
                \includegraphics[height=2in]{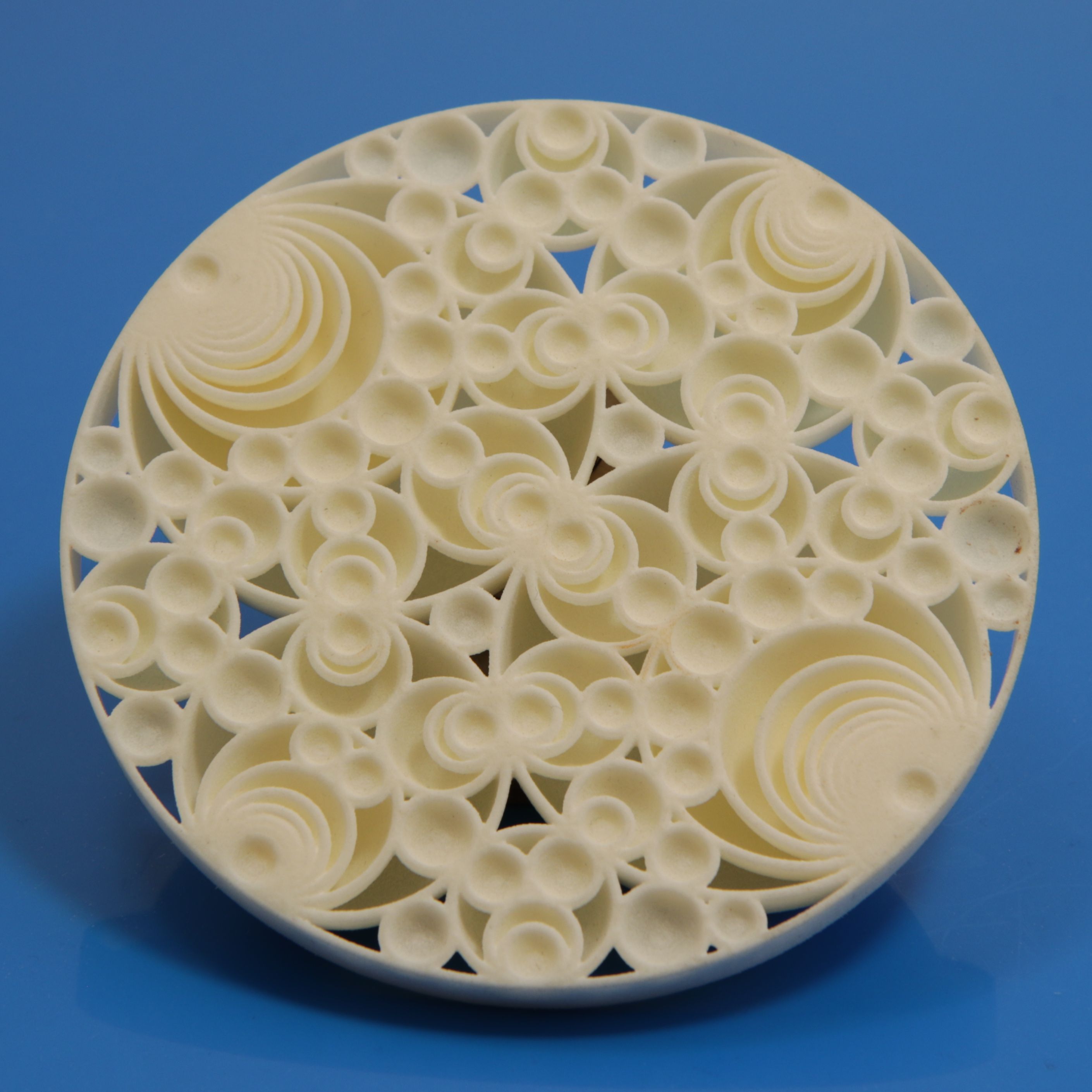} \quad
                \includegraphics[height=2in]{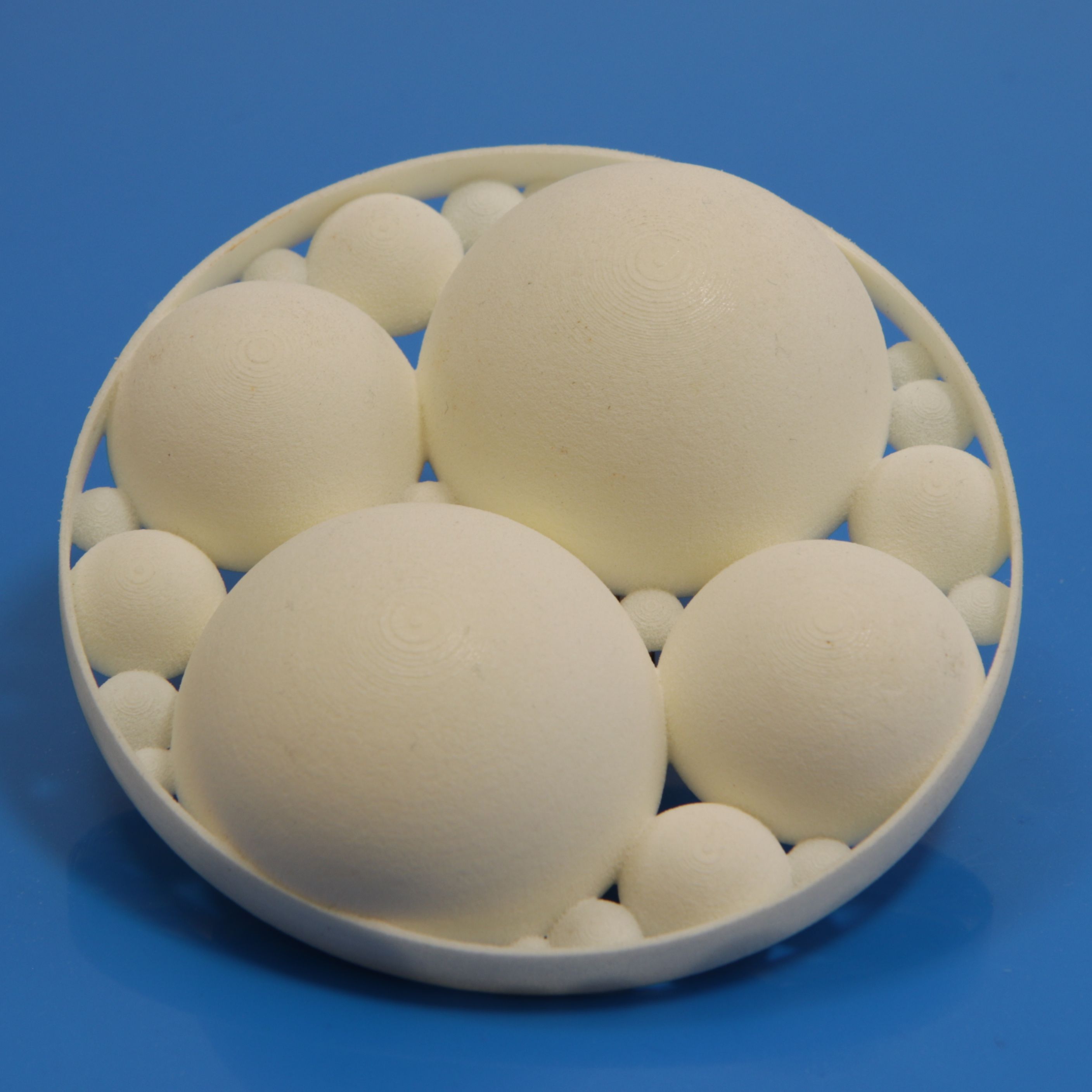} 
		\caption{The geodesic surfaces that form the Schmidt arrangement, 3d printed (front and back views).}
		\label{fig:gauss-schmidt-3d}
	\end{figure}

	The arithmetic of Kleinian groups has a long history.  The M\"obius transformations of $\widehat{\CC}$ extend to hyperbolic isometries on the upper half plane model of hyperbolic $3$-space for which $\widehat{\CC}$ is the boundary.  The quotient of this space by a Bianchi group defines a \emph{Bianchi orbifold}, and the arithmetic of the field is known to play an important role in the topology and geometry of these orbifolds (see, for example \cite{MR}).  As the simplest example, the cusps of the Bianchi orbifold are in bijection with the class number.  The Schmidt arrangement is another aspect of the orbifold:  in essence, $\SK$ represents a particular choice of geodesic surface in the manifold.  The classification of geodesic surfaces in Bianchi orbifolds is not yet well understood. %, and it is not clear why the Schmidt arrangement surface has all the properties that it does.

	\begin{figure}
                \includegraphics[height=4.0in]{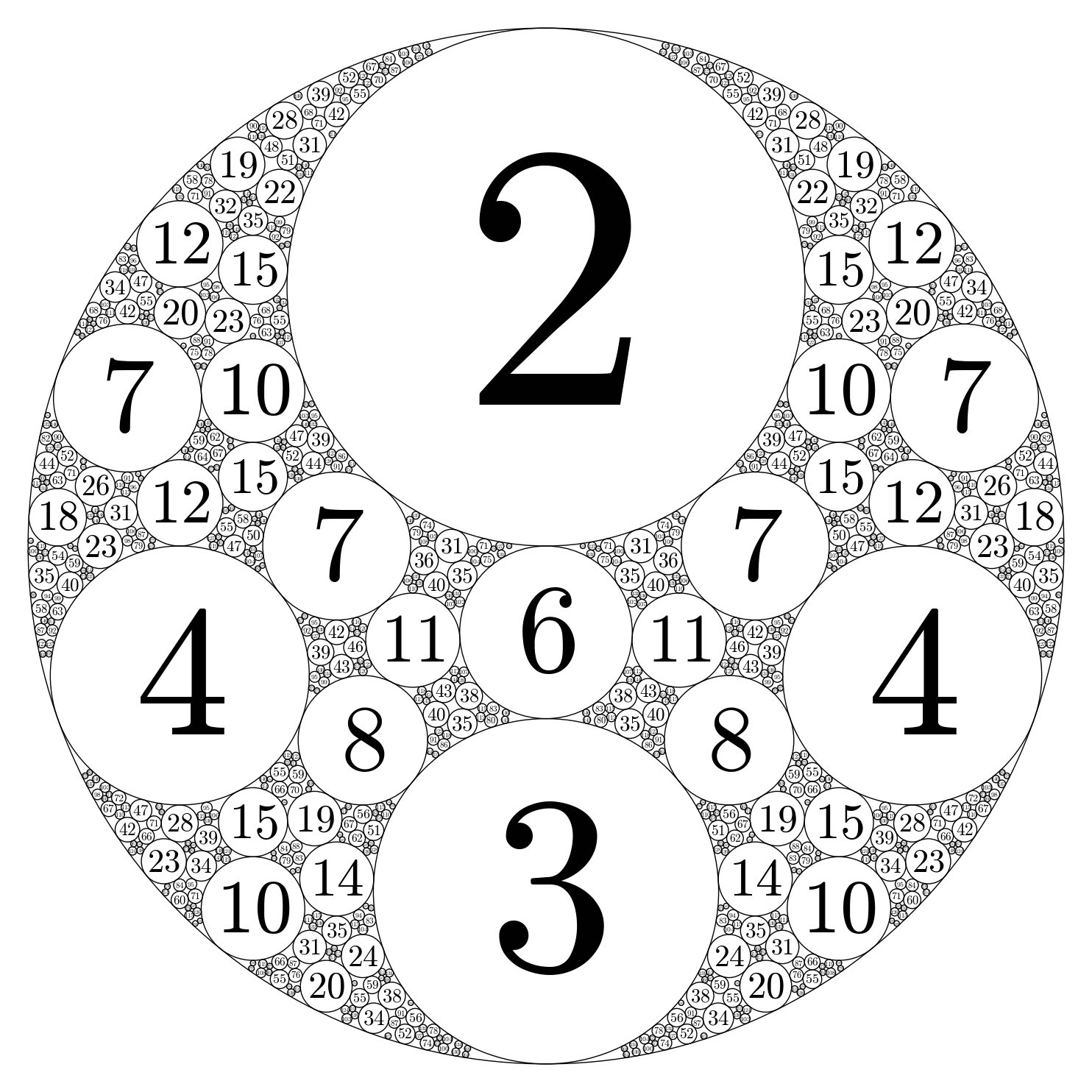}
		\caption{A $\QQ(\sqrt{-2})-$Apollonian packing with curvatures shown.}
		\label{fig:2-app}
	\end{figure}

	\begin{figure}
                \includegraphics[height=2.0in]{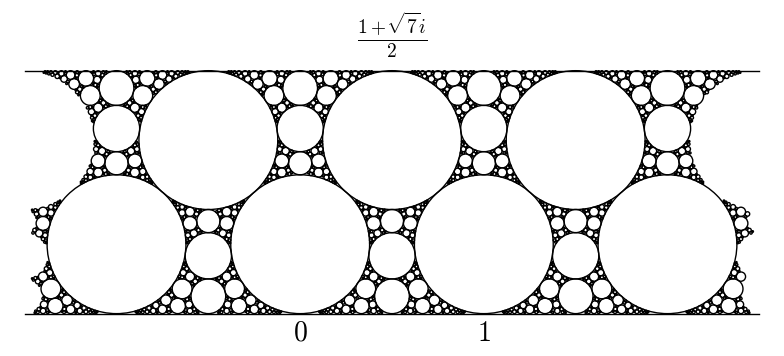}
		\caption{A $\QQ(\sqrt{-7})-$Apollonian packing.}
		\label{fig:7-app}
	\end{figure}

	There is a simple geometric criterion for an Apollonian circle packing as a subset of $\Scal_{\QQ(i)}$:  it is obtained from any one circle by adding on the largest exteriorly tangent circle at each tangency point.  This criterion, applied to other Schmidt arrangements, gives $K$-Apollonian packings for other imaginary quadratic fields (Figures~\ref{fig:2-app} and \ref{fig:7-app}).  Along with these packings come $K$-Apollonian groups for which these packings are the limit set.  These are thin groups acting on appropriate clusters of circles (analogous to the notion of Descartes quadruple) to generate the packing.  For example, in $\QQ(\sqrt{-2})$ the relevant cluster has as tangency graph a cube, and there are six swaps through faces, so that the $K$-Apollonian group is a free product of six copies of $\ZZ/2\ZZ$.  The description of the local obstructions can be extended to $K$-Apollonian packings \cite{StangeBianchi}.

	\section{A postscript}

	I have had the good fortune of a great deal of freedom in choosing and exploring mathematical projects.  I have found myself attracted to subjects in number theory that have an essential geometric aspect, in particular one which can illuminate proofs and demand its own questions.  I fell in love with quadratic forms and continued fractions through the Farey subdivision and Conway's topograph.  I studied curves in graduate school, and learned that it was the genus -- their topology -- that controlled their Diophantine behaviour.  I believe strongly that number theory -- probably all number theory -- is essentially geometric, and that algebra, although powerful and in possession of a beauty of its own, can at times obscure a hidden geometric splendour.

	Of my own research explored in these notes, every project involved extensive computer experiments, often visual ones.  As I discovered the properties of Schmidt arrangements, I held an old fashioned compass to a computer print-out to find patterns.  As I worked with Harriss and Trettel on algebraic numbers, we made an explicit decision to let aesthetics drive our computer experiments, which is how choices like sizing by discriminant and measuring approximation in hyperbolic geometry were born.  

	We are just discovering the many ways in which Apollonian circle packings are essential and natural objects in number theory, much like elliptic curves.  But my favourite justification is geometric:  draw the Schmidt arrangement, which is, in a very real sense, the way that Gaussian rationals choose to organize themselves, and the Apollonian packings are an essential intermediate geometric piece.  They pop out of the picture.  They cannot be avoided: they are dictated by nature.

	I've learned many things from these experiences.  I've learned that the visual cortex is a powerful tool for creativity and intuition, as well as reasoning.  Patterns emerge to our visual cortex that will pass unnoticed in numerical data.  We are essentially visual and social beings, and as such, we vividly recall our visual and social encounters.  This is why mathematics is often best conveyed in stories and pictures.  We remember mathematical objects that appear, to us, to have personality and shape.  As mathematicians, we do this, more-or-less unconsciously, with the most abstract mathematical objects; they become our friends, our tormentors, our landscapes.  Why not also do it explicitly?

	I encourage the reader to wander the pages of John H. Conway and Francis Y. C. Fung's \emph{The Sensual Quadratic Form} \cite{ConwayFung}, Martin H. Weissman's \emph{An Illustrated Theory of Numbers} \cite{Weissman}, and Allen Hatcher's \emph{Topology of Numbers} \cite{Hatcher}; to value an illustrative and visual approach to mathematics; and, at the risk of sentimentality, to follow one's heart in mathematical research and elsewhere.

	\bibliographystyle{alpha}
	\bibliography{app-pack-bib}

	\end{document}